\documentclass[11pt,a4paper]{article}
\usepackage[utf8]{inputenc}
\usepackage[english]{babel}
\usepackage{empheq}
\usepackage{amsmath}
\usepackage{amsfonts}
\usepackage{amssymb}
\usepackage{amsthm}
\usepackage{graphicx}
\usepackage[framemethod=TikZ]{mdframed}
\usepackage[margin=1in]{geometry}
\usepackage{hyperref}
\usepackage{authblk}
\usepackage{mathtools}
\usepackage{esint}
\usepackage{subfigure}

\usepackage{xcolor}

\newcommand{\ep}{\epsilon}

\newtheorem{theorem}{Theorem}[section]
\newtheorem{lemma}[theorem]{Lemma}
\newtheorem{claim}[theorem]{Claim}
\newtheorem{proposition}[theorem]{Proposition}
\newtheorem{corollary}[theorem]{Corollary}

\newtheorem{Theorem}{Theorem}

\theoremstyle{definition}
\newtheorem{definition}[theorem]{Definition}

\theoremstyle{remark}
\newtheorem{remark}[theorem]{Remark}

\numberwithin{equation}{section}

\begin{document}
\title{Asymptotic spreading of KPP reactive fronts in heterogeneous shifting environments II: Flux-limited solutions}
\author[King-Yeung Lam, Gregoire Nadin and X. Yu]{King-Yeung Lam\thanks{Department of Mathematics, The Ohio State University, Columbus, OH 43210, USA (lam.184@osu.edu)}\, ,
	 Gr\'{e}goire Nadin\thanks{Institut Denis Poisson, Universit\'{e} d’Orl\'{e}ans, Universit\'{e} de Tours, CNRS, Orl\'{e}ans, France (gregoire.nadin@cnrs.fr)}\, and
	  Xiao Yu\thanks{School of Mathematical Sciences, South China Normal University, Guangzhou 510631, China (xymath19@m.scnu.edu.cn)}
}


	\maketitle
	
	\begin{abstract}
	    We consider the spreading dynamics of the Fisher-KPP equation in a shifting environment, by analyzing the limit of the rate function of the solutions.  
For environments with a weak monotone condition, it was demonstrated in a previous paper that the rate function converges to the unique Ishii solution of the underlying Hamilton-Jacobi equations. In case the environment does not satisfy the weak monotone condition, we show that the rate function is then characterized by the 
     Hamilton-Jacobi equation with a dynamic junction condition, which depends additionally on the generalized eigenvalue derived from the environmental function.  Our results applies to the case when the environment has multiple shifting speeds, and clarify the connection with previous results on nonlocally pulled fronts and forced traveling waves.
	\end{abstract}

\tableofcontents

	\section{Introduction}
	
	

Consider the KPP equation with heterogeneous coefficients
	\begin{equation}
		\begin{cases}
			u_t -  u_{xx} = u(g(x-c_1 t)-u) &\text{ for }(x,t)\in \mathbb{R}\times (0,\infty),\\
			u(x,0) = u_0(x) &\text{ for }x\in \mathbb{R},
		\end{cases}
		\label{eq:1.1}
	\end{equation}
	where $c_1\in \mathbb{R}$,  $g$ is a continuous and positive function on $\mathbb{R}$ with $g(\pm \infty)>0$, and $u_0\in L^\infty(\mathbb{R})$ is nonneagative and compactly supported. This type of equations models the population growth in a shifting habitat. They may arise from the ecological question whether the species can survive in the midst of climate change \cite{Potapov2004,Berestycki2009,Li2014}. In a previous work \cite{Lam2022asymptotic}, we considered   
\eqref{eq:1.1} as a special case of a class of integro-differential equations with a distributed time-delay in heterogeneous shifting environments. Under the assumption $\sup g\le\max\{g(\pm\infty)\}$, 
 we utilized the theory of viscosity solutions of Hamilton-Jacobi equations, specifically the uniqueness of viscosity solutions in the sense of Ishii \cite{Ishii1985hamilton}, to obtain the complete explicit formulas of rightward spreading speeds for \eqref{eq:1.1} in terms of $c_1$ as well as leftward spreading speeds. However, the assumption $\sup g\le\max\{g(\pm\infty)\}$ was crucial in \cite{Lam2022asymptotic}, as uniqueness of viscosity solutions in the sense of Ishii can no longer be expected if it was relaxed  \cite{barles2023modern,Giga2013hamilton}.
 
To recover uniqueness and further develop the Hamilton-Jacobi approach, we will utilize the notion of flux-limited solution, which was recently introduced by Imbert and Monneau \cite{Imbert2017flux,Imbert2017quasi} to study Hamilton-Jacobi equations on networks. 
Throughout our paper, we are mainly working on \eqref{eq:1.1} which has one shifting speed $c_1$, and we are able to determine the spreading properties of \eqref{eq:1.1} for any $g$ for which $g(\pm \infty)$ exist. Furthermore, our treatment naturally extends to the case of multiple shifting speeds; see Section \ref{extension} for the precise statements.  Before stating our main results, we provide a brief review on some related works. 
	
 The asymptotic speed of spread, or spreading speed in short, is a crucial 
quantity in spatial ecology that determines the expansion boundary of a population under the joint influence of the diffusion rate  and environmental conditions. For simplicity, the diffusion rate has been normalized to $1$ in \eqref{eq:1.1}. In a homogeneous environment, i.e. $g(\cdot)\equiv g_0$ for a positive constant $g_0$, model \eqref{eq:1.1} reduces to the classical Fisher-KPP equation.  
{A well known result of Kologorov et al.  \cite{Kolmogorov1937}  states} that there is a number $c_* = 2\sqrt{g_0}>0$ such that 
 $$
 \lim_{t\to\infty} \sup_{x \geq ct} u(t,x) = 0\quad \text{ for $c\in(c_*,+\infty)$, \quad and }\quad \liminf_{t\to\infty} \inf_{0 \leq x \leq c t} u(t,x)>0 \quad \text{ for }c\in(0,c_*). 
 $$ 
Moreover, the same value also coincides with the minimal wave speed of traveling wave solutions $U(x-ct)$ of \eqref{eq:1.1}.
This result was later extended to more general nonlinearity and in higher dimensions in \cite{Weinberger1978}. It is also remarkable that in homogeneous environment, the spreading speed can be obtained via local information, where  $c_*=2\sqrt{g_0}$ is the smallest value of $c>0$ such that
$$
-c \phi_x + \phi_{xx} + g_0 \phi = \Lambda \phi
$$
admits a zero eigenvalue.


Since then, spreading speed for various reaction-diffusion models including Fisher-KPP equations are intensively investigated \cite{Weinberger1982,LiangZhao2007, BHN2008, Shen2010,Nadin2012,FYZ2017,Nadin2020}.   Among those, an elaborate method was proposed by Weinberger \cite{Weinberger1982} to establish the existence of spreading speeds for discrete-time order-preserving recursions with a monostable structure and its characterization as the minimal speed of traveling wave solutions. These results were subsequently generalized to monotone semiflows \cite{LiangZhao2007,FYZ2017}. 
 By combining the Hamilton-Jacobi approach \cite{Evans1989pde} and homogenization ideas \cite{lions1987homogenization,evans1992periodic}, Berestycki and Nadin \cite{Nadin2012,Nadin2020} showed the existence of spreading speed for spatially almost periodic, random stationary ergodic, and other general environments, whose speed was characterized as a min-max formula in terms of suitable notions of generalized principal eigenvalues in unbounded domains.

The heterogeneous shifting environment, which is the focus of this article, was introduced by Potapov and Lewis \cite{Potapov2004} and Berestycki et al. \cite{Berestycki2009} to investigate the impact of shifting climate on the persistence of one or several focal species. As a simple formulation, the temporal-spatial heterogeneity $x-c_1t$ was incorporated into various diffusion models including \eqref{eq:1.1} for the single species, where $c_1$ is regarded as the shift velocity of climate.  For \eqref{eq:1.1},  the propagation dynamics have been rigorously explored in  \cite{Potapov2004, Berestycki2009,Berestycki2008reaction,Berestycki2009reaction} for the case of a moving patch of a finite length, and in Li et al. \cite{Li2014} for a retreating semi-infinite patch. The latter problem is a special case of \eqref{eq:1.1} in case $g$ is increasing and $g(-\infty)<0<g(+\infty)$, where it is proved that the species persists if and only if it can spread faster than the environment with the spreading speed being given by the KPP formula $c_*= 2\sqrt{g(+\infty)}$.  Hu et al. \cite{HSL2019} provided further investigation without assuming $g(-\infty)<0$, and proved that the rightward spreading speed $c_*=2\sqrt{g(+\infty)}$ if $c_1\le 2\sqrt{g(+\infty)}$ and $c_*=2\sqrt{g(-\infty)}$ if $c_1\gg 1$.  A shifting environment can also arise in other ways. Holzer and Scheel \cite{Holzer2014accerlated} considered a partially decoupled reaction-diffusion system of two equations, where a wave solution for the first equation induces a shifting environment for the second one. See also \cite{Ducrot2021asymptotic,Dong2021persistence,Girardin2019,Liu2021stacked} for further results on competition or prey-predator systems. Similar modeling idea was also adapted in Fang et al. \cite{FLW2016},  where \eqref{eq:1.1} was also retrieved from an SIS disease model to study whether pathogen can keep pace with its host. If $g$ is non-increasing, then \eqref{eq:1.1} becomes a special case of the cylinder problem studied by Hamel \cite{Hamel1997reaction}. Du et al.\cite{DWZ18} proposed a free boundary version of \eqref{eq:1.1}, see also \cite{DHL23}.  Yi and Zhao \cite{Yi2020, YZDCDS23} established a general theory on the propagation dynamics without spatial translational invariance. See also \cite{Faye2022asymptotic} for a model with shifting diffusivity.  
We refer to Wang et al. \cite{WLFQ2022} for a survey on reaction-diffusion models in shifting environments.
 
   Indeed,  the shifting habitat does bring about new spreading phenomena in  case that the intrinsic growth rate function $g$ is strictly positive everywhere. When $0<\inf{g}<\sup {g}\le \max\{g(\pm\infty)\}$, the results of \cite{Holzer2014accerlated,Lam2022asymptotic} clarified that, for a certain range of shift speed $c_1$, the initially compactly supported population spreads at a supercritical speed $c_* > 2\sqrt{g(-\infty)}$ 
in a phenomenon called non-local pulling \cite{Holzer2014accerlated, Girardin2019}. This falls into the biological scenario when the species fails to keep up with the climate shifting, but is still influenced by the presence of a favorable habitat which is at a distance of order $t$ ahead of the front.
 When $\sup {g}>\max\{g(\pm\infty)\}$, then  Holzer and Scheel \cite{Holzer2014accerlated} proved the existence of forced traveling wave solution, which moves in the same speed as the environment. Subsequently, 
    Berestycki and Fang \cite{Berestycki2018forced}  classified such forced traveling wave solutions and proved global attractivity results.

Our main contribution, in the case of \eqref{eq:1.1}, where the environment has a single speed $c_1$, is to completely determine the existence of rightward spreading speed $c_*$ and its dependence on the environmental speed $c_1$, whenever $g(\pm \infty)$ exist and $\inf g >0$. Moreover, our framework provides the context in which the spreading results in \cite{Holzer2014accerlated,Lam2022asymptotic} (where $c_* < c_1$ with nonlocal pulling) connects with those in \cite{Berestycki2018forced} (where $c_* = c_1$). Furthermore, our method readily generalizes to the case when the environment has more than one shifting speed (Subsection \ref{extension}).

 \subsection*{Organization of the paper}
Our approach is to study the spreading speed via the asymptotic limit of the rate function, following \cite{Evans1989pde} (see also \cite[Chapter 29]{barles2023modern} and \cite{Lam2022asymptotic}).
However, the consideration of a shifting habitat leads to a discontinuous Hamiltonian. Also, the rate function has unbounded and discontinuous initial data since the initial population was compactly supported. In our previous work \cite{Lam2022asymptotic},
the solution concept of Ishii was used and the corresponding comparison principle was established. This is discussed in Subsection \ref{sec:ishii}.  However, the previous results in \cite{Lam2022asymptotic} is not applicable in case $\sup g > \max\{g(\pm \infty)\}$, because
then the invasion is enhanced by the specific profile of $g$, and 
the solutions in the sense of Ishii are non-unique.  To overcome the non-uniqueness of viscosity solution and connect with the results regarding forced waves, we need to incorporate further information of \eqref{eq:1.1} in deriving the limiting Hamilton-Jacobi equation. To this end, we recall some results of an eigenvalue problem depending on the coefficient $g(\cdot)$. Then 
in Subsection \ref{sec:fl}, we introduce the concept of a flux-limited solution and prove the comparison principle needed in our context.  In Subsection \ref{sec:main}, we state our main theorems and extensions of our results. {We also discuss the viscosity solutions in sense of Ishii and recall some earlier results from \cite{Lam2022asymptotic} in Subsection \ref{sec:ishii}.}

{Section \ref{sec:applications} presents the application of Theorem \ref{thm:2.13}, and we place it immediately after stating our main theorem. In this section, we take Theorem \ref{thm:2.13} for granted and apply it to obtain several explicit formulas for the spreading speed in terms of $g(+\infty), g(-\infty)$ and $\Lambda_1$, where $\Lambda_1$ is the principal eigenvalue given in \eqref{eq:lambda1}. This provides a general context connecting previous results of \cite{Berestycki2018forced, Holzer2014accerlated} concerning forced wave (where $c_* = c_1$) and of \cite{Lam2022asymptotic} concerning nonlocal pulling (where $c_*<c_1$ but is influenced by the presence of the shifting environment).}

{In Section \ref{sec:aux}, we present preliminary results. In particular, we recall the properties of $\Lambda_1$ in Proposition \ref{prop:2.1} (Subsection \ref{sec:eigen}), as well as a few technical results for Hamilton-Jacobi equations.}

{Section \ref{sec:proof} is devoted to the proof of the main results, namely, Proposition \ref{prop:2.10}, Corollary \ref{prop:2.12}, and Theorem \ref{thm:2.13}.
This section only logically depends on Proposition \ref{prop:2.1} (proved in Appendix \ref{sec:prop21}), Lemma \ref{cor:im1rho'} concerning continuity of subsolutions (proved in Subsection \ref{sec:4.4}), the critical slope lemmas inspired by Imbert and Monneau \cite{Imbert2017flux} (Lemmas \ref{cor:subsolr} and \ref{cor:supersolr}, proved in Appendix \ref{corsubsup}), as well as the comparison principle (proved in Appendix \ref{sec:comp}). }

{Finally, the appendices presents the proofs of the technical results mentioned above. }

\section{Preliminaries and Statements of Results}

	The concepts of maximal and minimal spreading speeds are introduced in \cite[Definition 1.2]{Hamel2012spreading} for a single species;  see also \cite{Garnier2012maximal,Liu2021stacked}. In our setting, we define
	\begin{equation}\label{eq:speeds}
		\begin{cases}
			\smallskip
			\overline{c}_*=\inf{\{c>0~|~\limsup \limits_{t\rightarrow \infty}\sup\limits_{x>ct} u(t,x)=0\}},\\
			\smallskip
			\underline{c}_*=\sup{\{c>0~~|\liminf \limits_{t\rightarrow \infty}\inf\limits_{0<x<ct} u(t,x)>0\}}, 
		\end{cases}
	\end{equation}
	where $\overline{c}_*$ and $\underline{c}_*$  are the maximal and minimal (rightward) spreading speeds of species $u$, respectively. If $\overline{c}_* = \underline{c}_*>0$, we say that the population has the (rightward) spreading speed given by the common value $c_*$.

	Motivated by the large deviations technique \cite{Evans1989pde,Freidlin1985limit}, we introduce, for fixed solution $u$ of \eqref{eq:1.1}, the scaling $u^\ep(t,x) = u(\frac{t}{\ep},\frac{x}{\ep})$ with $\epsilon>0$. The resulting function $u^\ep(t,x)$ 
	satisfies the following equation:
	\begin{equation}
		\begin{cases}
			u^\ep_t - \ep u^\ep_{xx} = \ep^{-1} u^\ep \left( g\left( \frac{x - c_1t}{\ep} \right) - u^\ep\right) &\text{ for }(t,x) \in (0,\infty)\times \mathbb{R},\\
			u^\ep(0,x) = u_0(x/\ep) &\text{ for }x \in \mathbb{R}.
		\end{cases}
	\end{equation}
	Observe that the spreading speed of the population is given by $c_*>0$ if and only if  
	\begin{equation}\label{eq:spread1}
		\lim_{\ep \to 0}u^\ep(t,x) \to 0 \quad \text{ in }C_{loc}(\{x>c_* t\}),\quad \liminf_{\ep \to 0}u^\ep(t,x) >0 \quad \text{ in }C_{loc}(\{0\le x<c_* t\}). 
	\end{equation}
To fully characterize these limits, we introduce the  following eigenvalue problem:
	\begin{equation}
		\Phi'' +  g(y)\Phi =\Lambda \Phi  \quad \text{ for }y \in \mathbb{R}.
		\label{eq:pev0}
	\end{equation}
	In this paper, we define the principal eigenvalue ${\Lambda_1}$ of \eqref{eq:pev0} as follows:
	\begin{equation}\label{eq:lambda1}
		{\Lambda_1}:= \Lambda_1(g) = \inf \left\{ \Lambda \in \mathbb{R}:~ \exists \phi \in C^2_{loc}(\mathbb{R}),~\phi>0,~ \phi'' + g(y)\phi \leq \Lambda \phi \text{ in } \mathbb{R} \right\}.
	\end{equation}
	This and several other notions of principal eigenvalues are analyzed in \cite{Berestycki2014generalizations}.
We will recall some basic properties of $\Lambda_1$ and the associated positive eigenfunction in Proposition \ref{prop:2.1}.

	As we shall see in Section \ref{sec:applications}, the four quantities $c_1$, $\Lambda_1$, $g(+\infty)$ and $g(-\infty)$ completely determine the spreading speed.


	\subsection{Flux-limited solution due to Imbert and Monneau}\label{sec:fl}
	
	To determine the exact spreading speed in Theorem \ref{thm:main0}, we will study the rate function $w^\ep(t,x):=-\ep \log u^\ep(t,x)$. More precisely, 
	we will show that $w^\ep(t,x) \to t \hat\rho(x/t)$ in $C_{loc}$, where the limit $\hat\rho$ is to be interpreted using the notion of {\it flux-limited solutions} introduced by Imbert and Monneau \cite{Imbert2017flux}. This notion is well-adapted to catch the influence of the coefficients along a discontinuity at $x=c_1$.
 
 We begin with a few notations regarding the effective Hamiltonian and effective junction condition.
	\begin{definition}
		For each $A \in \mathbb{R}$, define the flux-limited junction condition to be
		\begin{equation}\label{eq:FA}
			F_A(\tilde{p}_+, \tilde{p}_-) = \max\{A, H^-(c_1+,\tilde{p}_+), H^+(c_1-, \tilde{p}_-)\},
		\end{equation}
		where $H^+(c_1\pm,\cdot)$ and $H^-(c_1\pm,\cdot)$ are 
		\begin{equation}\label{eq:Hplus}
			H^+(c_1\pm,p)= \begin{cases}
				-\frac{|c_1|^2}{4} + g(\pm \infty),\,\,   &\text{ for }p \leq c_1/2,\\
				-c_1p + p^2 + g(\pm \infty)\,\, &\text{ for }p > c_1/2,
			\end{cases} 
		\end{equation}
		\begin{equation}\label{eq:Hminus}
			H^-(c_1\pm,p) =\begin{cases}
				-c_1p + p^2 + g(\pm \infty)\,\, &\text{ for }p \leq c_1/2, \\
				-\frac{|c_1|^2}{4} + g(\pm \infty) \,\, &\text{ for }p > c_1/2.
			\end{cases} 
		\end{equation}
  Note that they are, respectively, the increasing and decreasing parts (in the variable $p$) of 
  $$
  H(s,p) = -sp + p^2 + \chi_{\{s > c_1\}} g(+\infty) + \chi_{\{s \leq  c_1\}} g(-\infty).
  $$
	\end{definition}
	The information of the profile of $g$ can be incorporated into the Hamilton-Jacobi equation by an additional {\it junction condition} as follows:
	\begin{equation}\label{eq:fl}
		\begin{cases}
			\min\{\rho, \rho + H(s,\rho')\} = 0 \quad \text{ for }s> 0, s \neq c_1, \\
			\min\{\rho(c_1), \rho(c_1) + F_A(\rho'(c_1+), \rho'(c_1-))\}=0, 
		\end{cases}
	\end{equation}
	where $A$ and $F_A(\tilde{p}_+, \tilde{p}_-)$ are given in \eqref{eq:A} and \eqref{eq:FA} respectively. 
	The above equations are to be considered using  {\it piecewise} $C^1$ test functions whose left and right derivative at $c_1$ are well defined but maybe unequal:
	\begin{equation}\label{eq:c1pw}
		C^1_{pw} = \{\psi \in C((0,\infty)):~ C^1((0,c_1]) \cap C^1([c_1,\infty)).
	\end{equation}
	\begin{definition}\label{def:rhoflsol}
		Let $A \in \mathbb{R}$ be given.  \begin{itemize}  
			\item[{\rm (a)}] We say that $\hat\rho:(0,\infty) \to \mathbb{R}$ is a FL-subsolution of \eqref{eq:fl} provided (i) $\hat\rho$ is upper semicontinuous, and (ii) if $\hat\rho - \psi$ (with $\psi \in C^1_{pw}$) attains a local maximum point at some $s_0 >0$ such that $\hat\rho(s_0) >0$, then
			$$
			\hat\rho(s_0) + H(s_0,\psi'(s_0)) \leq 0 \quad \text{ in case }s_0 \neq c_1,
			$$
			$$
			\hat\rho(c_1) + F_A(\psi'(c_1+), \psi'(c_1-))\leq 0 \quad \text{ in case }s_0 = c_1.
			$$   
			\item[{\rm (b)}]  We say that $\hat\rho:(0,\infty) \to \mathbb{R}$ is a FL-supersolution of \eqref{eq:fl} provided (i) $\hat\rho$ is lower semicontinuous, (ii) $\hat\rho \geq 0$ for all $s >0$, and (iii) if $\hat\rho - \psi$ (with $\psi \in C^1_{pw}$) attains a local minimum point at some $s_0 >0$, then
			$$
			s_0 \neq c_1 \quad \Longrightarrow \quad  \hat\rho(s_0) + H(s_0,\psi'(s_0)) \geq 0; 
			$$
			$$
			s_0 = c_1 \quad \Longrightarrow \quad \hat\rho(c_1) + F_A(\psi'(c_1+), \psi'(c_1-))\geq 0. 
			$$   
			\item[{\rm (c)}] We say that $\hat\rho$ is a FL-solution of \eqref{eq:fl} if it is both FL-subsolution and FL-supersolution of \eqref{eq:fl}. 
		\end{itemize}
	\end{definition}
	Next, we discuss the uniqueness of FL-solution of \eqref{eq:fl} by first showing the following comparison principle.
	\begin{proposition}\label{prop:2.10}
		Let $A \in \mathbb{R}$ be given. If $\underline\rho$ and $\overline\rho$ are, respectively, the FL-subsolution and FL-supersolution of \eqref{eq:fl}, and such that
		\begin{equation}\label{eq:2.8rr}
			\underline\rho(0) \leq \overline\rho(0) \quad \text{ and }\quad \lim_{s \to +\infty} \frac{\overline\rho(s)}{s} = +\infty, 
		\end{equation}
		then $\underline\rho(s) \leq \overline\rho(s)$ in $[0,+\infty)$. 
	\end{proposition}

	\begin{corollary}\label{prop:2.12}
		For each $A\in \mathbb{R}$, \eqref{eq:fl} has a unique FL-solution $\hat\rho_A$ which satisfies the following boundary conditions (in a strong sense) 
  \begin{equation}\label{eq:ishii2}
		\rho(0) = 0 \quad \text{ and }\quad \lim_{s \to +\infty} \frac{\rho(s)}{s} =+\infty.
	\end{equation}
	\end{corollary}
	These two results will be proved in Subsection \ref{subsec:5.4}.

	
	We apply the half-relaxed limit method, due to Barles and Perthame \cite{BarlesPerthame87, BarlesPerthame88}, to pass to the (upper and lower) limits of $w^\epsilon(t,x)$.	Moreover, we can show that $w^\ep(t,x) \to t\hat\rho_A(x/t)$ in $C_{loc}$, where the {\it flux limiter} $A$ is identified by 
	\begin{equation}\label{eq:A}
		A = \Lambda_1 - \frac{|c_1|^2}{4}.
	\end{equation} 
{The spreading speed $c^*$ will then be fully characterized by $\hat\rho_A$ with the specific flux-limiter $A$.}

\begin{remark}
		Note that $A= \Lambda_1 - \frac{|c_1|^2}{4}$ could be treated as the principal eigenvalue of 
		 of $\Psi'' - c_1\Psi' + g(y) \Psi = A\Psi$ in the sense that:
		\begin{equation}\label{eq:lambda2}
			A = \inf \left\{ \Lambda \in \mathbb{R}:~ \exists \phi \in C^2_{loc}(\mathbb{R}),~\phi>0,~ \phi''-c_1\phi' + g(y)\phi \leq \Lambda \phi \text{ in } \mathbb{R} \right\},
		\end{equation}
		which quantifies the influence of the coefficient $g(x-c_1t)$ in the moving coordinate $y=x-c_1 t$.
	\end{remark}

\subsection{Main results}\label{sec:main}
	
	We are now in position to state our main result. 
	
	\begin{Theorem}\label{thm:2.13}
		Let $u$ be a solution of \eqref{eq:1.1}. Then the following statements hold.
		\begin{itemize}
			\item[{\rm(a)}] The spreading speed $c_*$ of $u$ exists, and is given by
			\begin{equation}\label{eq:sa}
				c_* = \hat{s}_A = \sup\{s \in [0,\infty): \hat\rho_A(s) = 0\},
			\end{equation}
			where $\hat\rho_A$ is the unique FL-solution of \eqref{eq:fl}with $A = \Lambda_1- \frac{c_1^2}{4}$ that also satisfies the boundary conditions \eqref{eq:ishii2}.
			\item[{\rm(b)}] Furthermore, if $\Lambda_1  = \max\{g(\pm\infty)\}$, then $c_* = \hat{s}_{base} = \sup\{s: \hat\rho_{base}(s) = 0\},$ where $\hat\rho_{base}$ is the unique viscosity solution of \eqref{eq:ishii1}--\eqref{eq:ishii2} in the sense of Ishii (see Definition \ref{def:1.2}). 
		\end{itemize}

	\end{Theorem}

	\begin{remark}
In Section \ref{sec:applications}, we will give explicit formulas of $c_*$ in terms of of $g(\pm\infty)$, $c_1$ and $\Lambda_1$.
	\end{remark}

\begin{remark}
After the research of this work has finished, the preprint of Giletti et al. \cite{girardin2024spreading} was brought to our attention. In this work, the authors treated the case when 
\begin{equation}\label{e.1003gg}
    g(t,x) = r_1 \chi_{\{x<A(t)\}} + r_2 \chi_{\{A(t) \leq x < A(t) + L\}} + r_3 \chi_{\{x\geq A(t)+L\}},
\end{equation}
where $t\mapsto A(t)$ is either linear or slowly oscillating between two shifting speeds. Interestingly, they obtained the formula of Theorem \ref{thm:main1} assuming that $g$ is given by \eqref{e.1003gg} with $A(t) = c_1 t$.

Furthermore, it was remarked that their construction can be generalized to treat \eqref{eq:1.1} provided that $g(y)$ is constant near $y=\pm \infty$. They also conjectured that the last condition may not be necessary. Their proof is based on the direct construction of super/subsolution for the parabolic problem.

Incidentally, our main result can be considered as an affirmative answer of their conjecture, by passing to the limiting Hamilton-Jacobi problem with junction condition. It is worth mentioning that (i) we need only  $A(t) = c_1t + o(t)$ and (ii) we merely require $g(\pm \infty)$ exist (but not necessarily constant for $|x| \gg 1$). In particular, the spreading speed can be determined by the value of the eigenvalue $\Lambda_1$ and the exact shape of $g$ is not important.

\end{remark}

	\subsection{Monotonicity of the flux-limited solutions}
	
{Since $F_A(\tilde{p}_+, \tilde{p}_-)$ is monotone increasing in the variable $A$, the effect of the flux limiter $A$ is as follows. 
	
	\begin{corollary} \label{prop:2.11}
		Let $A \in \mathbb{R}$ and $\hat\rho_A$ be the {unique} FL-solution of \eqref{eq:fl}-\eqref{eq:ishii2}.
		\begin{itemize}
			\item[{\rm(a)}] If $A \geq A'$, then 
			$\hat\rho_A(s) \leq \hat\rho_{A'}(s)$ for all $s \geq 0$. In particular, the free boundary point $\hat s_A$ is monotone increasing with respect to $A$, i.e. $\hat{s}_A \geq \hat{s}_{A'}$, where
			\begin{equation}
				\hat{s}_A:= \sup\{s: \hat\rho_A(s) = 0\}.
			\end{equation}
			\item[{\rm(b)}] If $ A_0:= \max\{g(\pm \infty)\}-\frac{c_1^2}{4}$ and $A\le A_0$, then $\hat\rho_A$ coincides with the unique viscosity solution of \eqref{eq:ishii1}-\eqref{eq:ishii2} in the sense of {Ishii}. In case  $A>A_0$, it is a   viscosity subsolution (might not be a supersolution) in the sense of Ishii.
		\end{itemize}
	\end{corollary}
	\begin{proof}
		This is a direct consequence of Proposition \ref{prop:2.10}, since $\hat\rho_A(0)=\hat\rho_{A'}(0)=0$, $\hat\rho_{A'}$ is a FL-supersolution of \eqref{eq:fl} and satisfies \eqref{eq:ishii2}.
	\end{proof}
}


\subsection{Generalizations}\label{extension}
	
	Our arguments can also be applied in the following setting where $f=f(t,x,u)$ possesses multiple junction points $\mathcal{P}=\{c_i\}_{i=1}^n$ for some $0<c_1<c_2<...<c_n$. 
	\begin{equation}
		\begin{cases}
			u_t -  u_{xx} = f(t,x,u) &\text{ for }(x,t)\in \mathbb{R}\times (0,\infty),\\
			u(x,0) = u_0(x) &\text{ for }x\in \mathbb{R}.
		\end{cases}
		\label{eq:1.1prime}
	\end{equation}
	
	\begin{description}
		\item[{\rm(H1$'$)}] $f(t,x,0) = 0$  for all $x \in \mathbb{R}$ and $t\ge0$.   
		\item[{\rm(H2$'$)}] {There exists $g_i \in L^\infty(\mathbb{R})$ with $g_i(\pm \infty)>0$, $i=1,...,n$,}
  and for some $\delta_1>0$,  
		$$
		\lim_{t\to\infty} {\mathrm{ess\ sup}}_{|x-c_it|<\delta_1 t} \left|f_u\left(t,x,0 \right) - g_i(x - c_i t)\right|=0.
		$$
		\item[{\rm(H3$'$)}] There exist ${R},\underline{R}:[0,\infty) \to [0,\infty)$ such that $\inf \underline{R} >0$, 
		$R(s) = \underline{R}(s)$ a.e. and 
  		$$
		R(s)=\limsup_{\ep \to 0^+ \atop (t,x) \to (1,s)}f_u\left(\frac{t}{\ep},\frac{x}{\ep},0\right) 
		\quad \text{ and }\quad 
		\underline{R}(s)=\liminf_{\ep \to 0^+ \atop (t,x) \to (1,s)}f_u\left(\frac{t}{\ep},\frac{x}{\ep},0\right). 
		$$
  Moreover, 
   $R$ is locally monotone in $\mathbb{R} \setminus \{c_i\}_{i=1}^n$, i.e.
for each $s_0 \in \mathbb{R} \setminus \{c_i\}_{i=1}^n,$ 
$$
\hspace{-1cm}\text{either}\quad \liminf_{\delta \to 0} \inf_{s_0-\delta <s<s'<s_0+\delta} R(s)- R(s') \geq 0 \quad \text{ or }\quad  \limsup_{\delta \to 0} \sup_{s_0-\delta <s<s'<s_0+\delta} R(s)- R(s') \leq 0. 
$$

		\item[{\rm(H4$'$)}] For each $(t,x) \in \mathbb{R}_+\times \mathbb{R}$, $f(t,x,u)/u$ is non-increasing in $u>0$. 
		\item[{\rm(H5$'$)}] There exists $M>0$ such that $f(t,x,u) <0$ for $(t,x,u) \in \mathbb{R}_+\times  \mathbb{R} \times [M,\infty)$.
		\item[{\rm(H6$'$)}] For each $i=1,...,n$, 
		{let $\Lambda^{(i)}_1$ be the principal eigenvalue given by}
	\begin{equation}\label{eq: gevp-i}
		{\Lambda^{(i)}_1} = \inf \left\{ \Lambda \in \mathbb{R}:~ \exists \phi \in C^2_{loc}(\mathbb{R}),~\phi>0,~ \phi'' + {g_i(y)}\phi \leq  \Lambda \phi   \right\}.
	\end{equation}

	\end{description}
	\begin{theorem}\label{thm:main2}
		Given $\{c_i\}_{i=1}^n$, and $f(t,x,u)$ satisfying {\rm (H1$'$) -- (H6$'$)}.
		If $u_0$ is compactly supported and nontrivial, then the population spreads to the right at speed $c_*$, where 
		$$
		c_* = \sup\{s \geq 0:~ \rho^\dagger(s) = 0\},
		$$
		where
		$\rho^\dagger$ is the unique FL-solution of 
		\begin{equation}\label{eq:3.27}
			\begin{cases}
				\min\{\rho,\rho + H(s,\rho')\}=0 & \hspace{-1.8cm}\text{ for }s \in (0,\infty) \setminus\{c_i\},\\
				\min\{\rho(c_i), \rho(c_i) + F^{(i)}(\rho'(c_i+),\rho'(c_i-)) = 0 &\text{ for each }i,\\
				\rho(0)=0 \quad \text{ and }\quad  \lim\limits_{s\to+\infty}\rho(s)/s=+\infty.
			\end{cases}
		\end{equation}
		where $$H(s,p) = -sp + p^2 + R(s),$$ 
		$$
		F^{(i)}(\tilde{p}_+, \tilde{p}_-) = \max\left\{\Lambda^{(i)}_1 - \frac{c_i^2}{4},~H^-(c_i+,\tilde{p}_+),~H^+(c_i-,\tilde{p}_-)\right\}, 
		$$
		such that $H^\mp(c_i\pm,\cdot)$ are the decreasing/increasing parts of  $H(c_i\pm,\cdot)$ given by
		$$
		H^-(c_i+,p)=H(c_i+,\min\{p,c_i/2\}),\quad H^+(c_i-,p)=H(c_i-,\max\{p,c_i/2\}),
		$$
		and
		$\Lambda^{(i)}_1$ is the principal eigenvalue given by \eqref{eq: gevp-i}.
	\end{theorem}

\subsection{Related {optimal} control formulations}
{It is well known that the viscosity solutions of Hamilton-Jacobi equations correspond naturally to the value functions of certain optimal control problems, whereas the viscosity solutions of variational inequalities correspond to the value functions of two-player, zero-sum deterministic differential games with stopping time \cite{Evans1989pde}. 
In fact, following the control formulation for Hamilton-Jacobi equations stated in \cite{Imbert2017flux}, we conjecture that the following formulation should hold:
\begin{equation}\label{e.conAA}
  {\rho(s) = \inf_{z \in X\atop z(0)=s}\left[ \sup_{T \in [0,\infty]}\int_0^{T} e^{-s'} \ell(z(s'), z(s')-\dot{z}(s'))\,ds'\right],} 
\end{equation}
where $X=H^1([0,\infty);[0,\infty))$, $T \in [0,\infty]$ is any constant and the cost function $\ell$ is given by
$$\ell (s,\alpha):= \frac{\alpha^2}{4}- R (s) \hbox{ if } s \in \mathbb{R} \setminus \{c_i\}_{i=1}^n, \quad \ell (s,\alpha):= \frac{\alpha^2}{4}- \Lambda^{(i)} \hbox{ if } s =c_i. $$
By a change of variables $\tau = t(1-e^{-s'})$, one can show that $w(t,x) = t\rho(x/t)$ satisfies 
\begin{equation}\label{e.conBB}
{w(t,x) = \inf_{\gamma(0)=x\atop   \gamma(t)=0}
\left[\sup_{\theta \in [0,t]}  \int_0^{t\wedge \theta} 
\ell \left(\frac{\gamma(\tau)}{t-\tau}, -\dot\gamma(\tau) \right)
\,d\tau\right].}
\end{equation}
By applying the arguments in \cite[Lemma 2.4]{Freidlin1996wave}, one can check that the above is consistent with the known max-min formulas involving stopping times when the running cost $\ell$ is a continuous function \cite{Evans1989pde,Nadin2020}. When the minimum with $\rho$ is not taken in the problem \eqref{eq:3.27}, then the unique viscosity solution can be characterized via the optimal control formulation (with $T$ and $\theta$ taken to be $+\infty$ in \eqref{e.conAA} and \eqref{e.conBB} respectively); See \cite{Imbert2017flux,barles2023modern}.
}

\subsection{An earlier result: Viscosity solution in the sense of Ishii}\label{sec:ishii}

First, let $\tilde{H}$ be the truncated version of the Hamiltonian ${H}$, given by
	\begin{equation}\label{eq:ishii3}
		\tilde{H}(s,p) = -sp + p^2 + g(-\infty) \quad \text{ for }s < c_1, \quad \tilde{H}(s,p) = -sp + p^2 + g(+\infty) \quad \text{ for }s > c_1
	\end{equation}
 and $\tilde{H}(c_1,p) = -c_1p + p^2 + g(-\infty) \vee g(+\infty)$.
	Note that $\tilde{H}$ uses only the information $g(\pm\infty)$ but does not depend on the specific form of the profile of $g$.
	The following left and right limits of $H(s,p)$ at $s=c_1$, as functions of $p$, will be used later.
	\begin{equation}\label{eq:Hcpm}
		\tilde{H}(c_1-,p) = -c_1p + p^2 + g(-\infty) \quad \text{ and }\quad \tilde{H}(c_1+,p) = -c_1p + p^2 + g(+\infty).
	\end{equation} 
 In a previous paper \cite{Lam2022asymptotic}, we studied  equation \eqref{eq:1.1} in the case 
	\begin{equation}\label{ineq: upbound}
	\sup_{y\in\mathbb{R}}	g(y) \leq \max\{g(+\infty),g(-\infty)\},
	\end{equation}
which is connected to the unique $\hat \rho_{base}$ such that \eqref{eq:ishii2} holds and satisfies, in viscosity sense,  
	\begin{equation}\label{eq:ishii1}
		\min\{\rho, \rho + \tilde{H}(s,\rho')\} = 0 \quad \text{ for }s > 0.
	\end{equation}

	Since the Hamiltonian function $\tilde{H}$ is discontinuous at $c_1$, the viscosity solution needs to be interpreted in a relaxed sense introduced by Ishii \cite{Ishii1985hamilton}. As this definition is well known, we skip it here and refer the reader to  Definition \ref{def:1.2} in the appendix.

We will show in Section \ref{subsec:5.2} that $\rho$ is an Ishii solution to the equation \eqref{eq:ishii1} if and only if it is a FL-solution of \eqref{eq:fl} with $A = A_0:= \max\{\min\limits_{\mathbb{R}} H(c_1+,\cdot), \min\limits_{\mathbb{R}} H(c_1-,\cdot)\}$. See Proposition \ref{prop:equi}. 

    In \cite{Lam2022asymptotic}, we showed that the rate function of problem \eqref{eq:1.1} selects $\hat\rho_{base}$  provided that \eqref{ineq: upbound} holds. This means that the spreading speed $s_{base}$ is  as predicted by the 
    the equations \eqref{eq:ishii1}-\eqref{eq:ishii2} (with solution in the sense of Ishii).
	\begin{theorem} 
		For each $c_1 >0$, there exists a unique $\hat\rho_{base}$ which satisfies \eqref{eq:ishii1} in the sense of Ishii and the boundary conditions \eqref{eq:ishii2}. Furthermore, if \eqref{ineq: upbound} is valid, then
		$\bar{c}_* = \underline{c}_*= s_{base}$, where $s_{base} = \sup\{s\geq 0:~ \hat\rho_{base}(s)=0\}$.
		\label{thm:LY1}
	\end{theorem}
	\begin{proof}
		See \cite[Proposition 1.7, Theorem 1 and Remark 1.6]{Lam2022asymptotic}. 
	\end{proof}

For the case when $g(y)$ does not satisfy \eqref{ineq: upbound}, we may compare the solution $u$ of \eqref{eq:1.1} with the solution $\tilde{u}$ of the same problem with $g$ replaced by the truncation $\min\{g, \max\{g(\pm \infty)\}\}$, 
	to deduce that the spreading speed is always bounded from below by $s_{base}$. 
{However, we will show in Corollary \ref{cor:main0} below that if $\Lambda_1>\max\{g(\pm\infty)\}$, then this lower bound is not optimal.}

	\begin{corollary}
		Suppose $\sup g > \max\{g(\pm \infty)\}$. For each $c_1>0$,
		$$
		\underline{c}_* \geq s_{base},
		$$
		where $\underline{c}_*$ is the minimal spreading speed and $s_{base}$ be given in Theorem \ref{thm:LY1}.
	\end{corollary}
	
	\begin{remark}\label{rmk:1.5}
In case \eqref{ineq: upbound} holds, the spreading speed can then be determined as soon as explicit solution $\hat\rho_{base}$ of \eqref{eq:ishii1}-\eqref{eq:ishii2} (in the sense of Ishii) can be constructed. This gives an alternative verification of the formula  \eqref{eq:cstarold} (in case $g(-\infty) \geq g(+\infty)$) and formula \eqref{eq:cstarold2} (in case $g(+\infty)\geq  g(-\infty)$) based on the viscosity solution in sense of Ishii \cite{Lam2022asymptotic}. 
%
	\end{remark}

	\section{Applications and Explicit Formulas}
	\label{sec:applications}
	As applications, we apply Theorem \ref{thm:2.13} to treat \eqref{eq:1.1}, which concerns the case when there is a single environmental shifting speed $c_1$. We will derive explicit formulas for the spreading speed in terms of $c_1$, $g(+\infty), g(-\infty)$ and $\Lambda_1$, where $\Lambda_1$ is the principal eigenvalue of \eqref{eq:pev0}  defined by \eqref{eq:lambda1}.  To simplify the notations, we denote for the remainder of this section
	$$
	r_- =g(-\infty),\quad \text{ and }\quad r_+ = g(+\infty).
	$$ 
  Thanks to standard properties of the principal eigenvalue $\Lambda_1$ (see Proposition \ref{prop:2.1}(a)) we have 
	$$\Lambda_1\in [\max\{r_-, r_+\},\infty).$$

\begin{remark}
			For $g(y)=r_-\chi_{\{y<0\}}+r_m\chi_{\{0\le y<L \}}+r_+\chi_{\{y\ge L\}}$ with $r_m>\max\{r_-, r_+\}$  that there exists  $\underline L\geq 0$   
   such that 
   $$
   \Lambda_1>\max\{r_-, r_+\} \quad \text{ if and only if } \quad L>\underline{L}$$ 
   where $\underline L = 0$ when $r_-=r_+$ and 
   \begin{equation*}
    \underline{L} = \frac{1}{\sqrt{r_m - \max\{r_-,r_+\}}} {\rm arccot} \left( \sqrt{\frac{r_m - \max\{r_-,r_+\}}{|r_- -r_+|}} \right)>0 \quad \text{ when }r_-\neq r_+.
\end{equation*}
See \cite[Theorem 2.1]{girardin2024spreading} for the precise statement.
	\end{remark}

	The following theorem says that the spreading speed $c_*$ is enhanced according to the specific profile of $g(\cdot)$. 
	\begin{theorem}\label{thm:main0}
		
  Let $u$ be a solution of \eqref{eq:1.1} with compactly supported, nonnegative, nontrivial initial data,  then the rightward spreading speed
		$$
		c_*=\overline{c}_* =\underline{c}_*=\begin{cases}
			2\sqrt{r_+} &\text{ if } c_1 \le 2\sqrt{r_+},\\
			c_1 &\text{ if } 2\sqrt{r_+} < c_1 \le 2\sqrt{\Lambda_1},\\
			\frac{c_1}{2} - \sqrt{\Lambda_1 - r_-} + \frac{{r_-}}{\frac{c_1}{2} - \sqrt{\Lambda_1 - r_-}} &\text{ if } 2\sqrt{\Lambda_1} < c_1 \le 2(\sqrt{r_-} + \sqrt{\Lambda_1 -r_-}),\\
			2\sqrt{r_-}  &\text{ if }  c_1 > 2(\sqrt{r_-} + \sqrt{\Lambda_1 -r_-}).
		\end{cases}
		$$
In particular, the mapping $c_1 \mapsto c_*$ is in general non-monotone, see  panels (b), (d) and (f) of Figure \ref{fig1} for illustration.
	\end{theorem}

\begin{figure}[htbp]
 \subfigure[$\Lambda_1=r_+>r_->0$]{
	\includegraphics[width=0.5\textwidth]{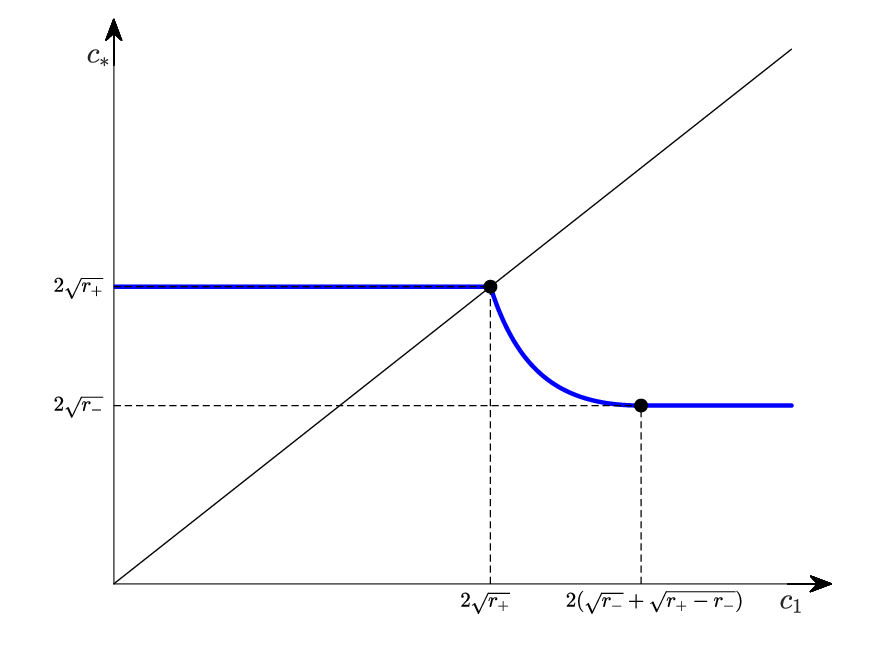}
	}
	\subfigure[$\Lambda_1>r_+>r_->0$]{
	\includegraphics[width=0.5\textwidth]{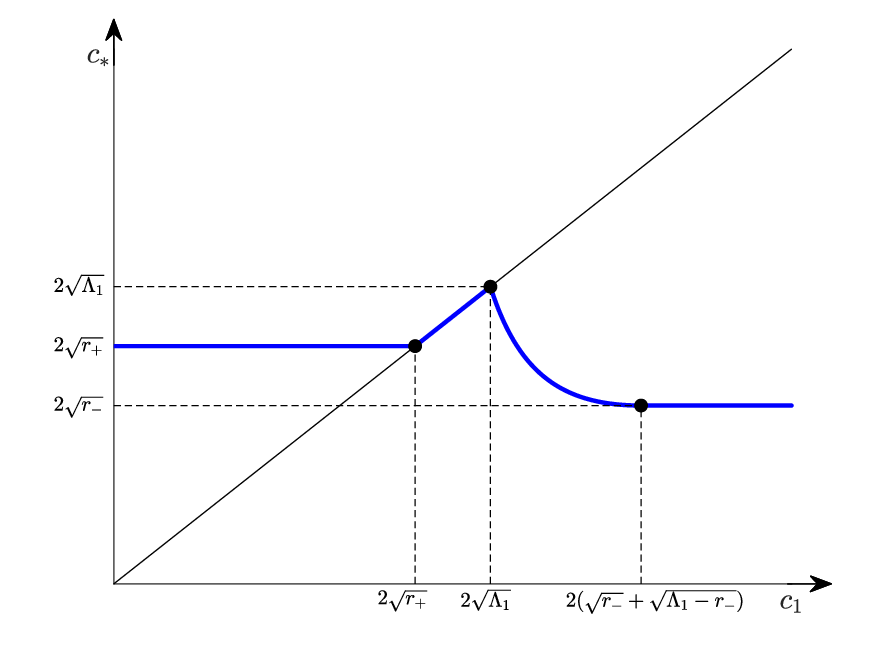}}
\subfigure[$\Lambda_1=r_->r_+>0$]{
		\includegraphics[width=0.5\textwidth]{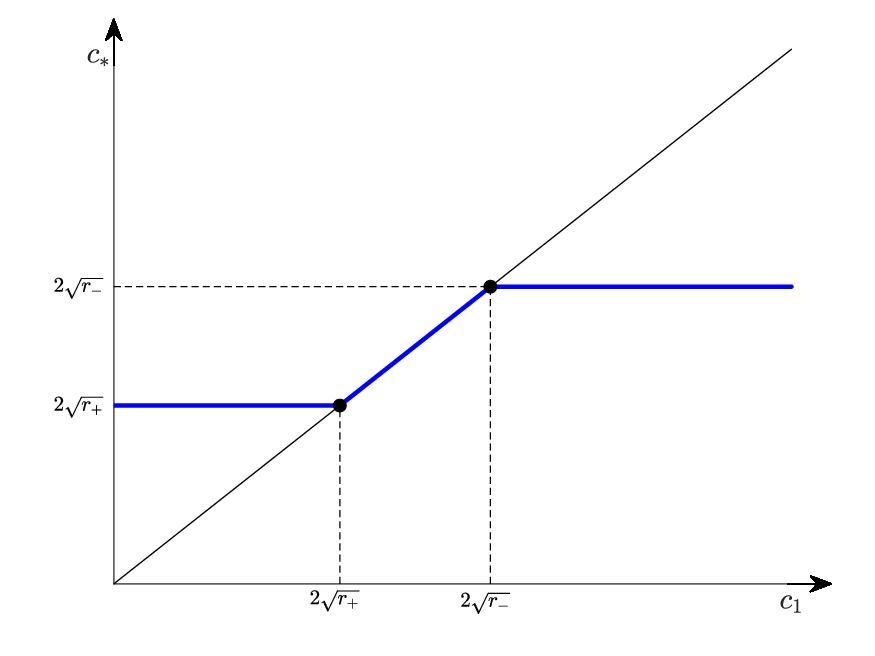}
	}
	\subfigure[$\Lambda_1>r_->r_+>0$]{
		\includegraphics[width=0.5\textwidth]{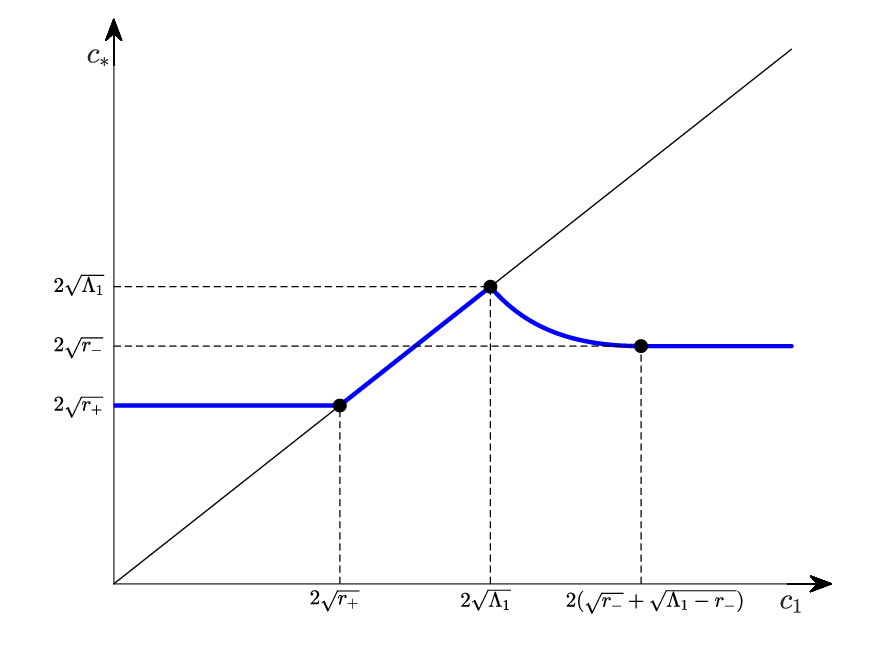}}
\subfigure[$\Lambda_1=r_-=r_+>0$]{
		\includegraphics[width=0.5\textwidth]{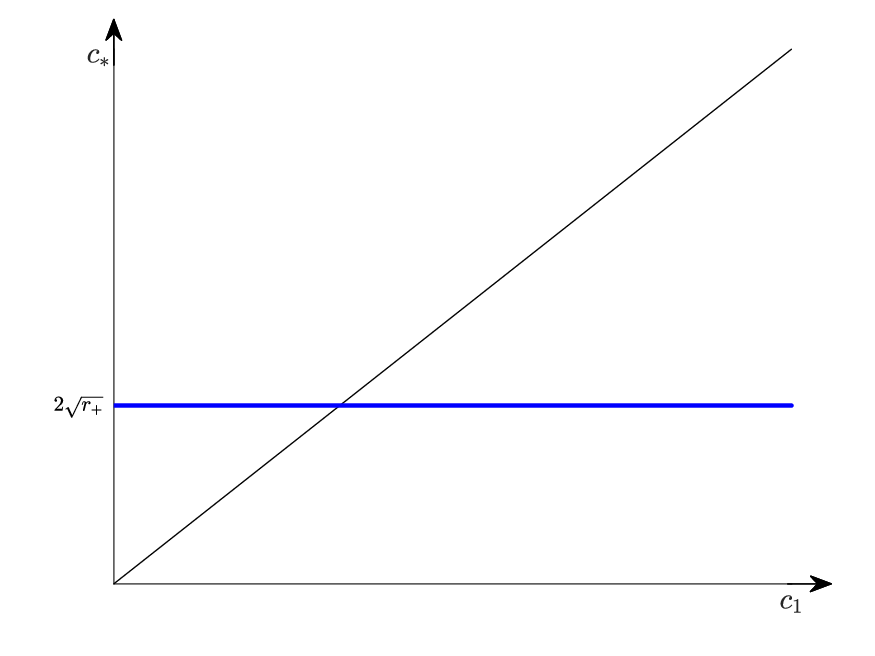}
	}
	\subfigure[$\Lambda_1>r_-=r_+>0$]{
		\includegraphics[width=0.5\textwidth]{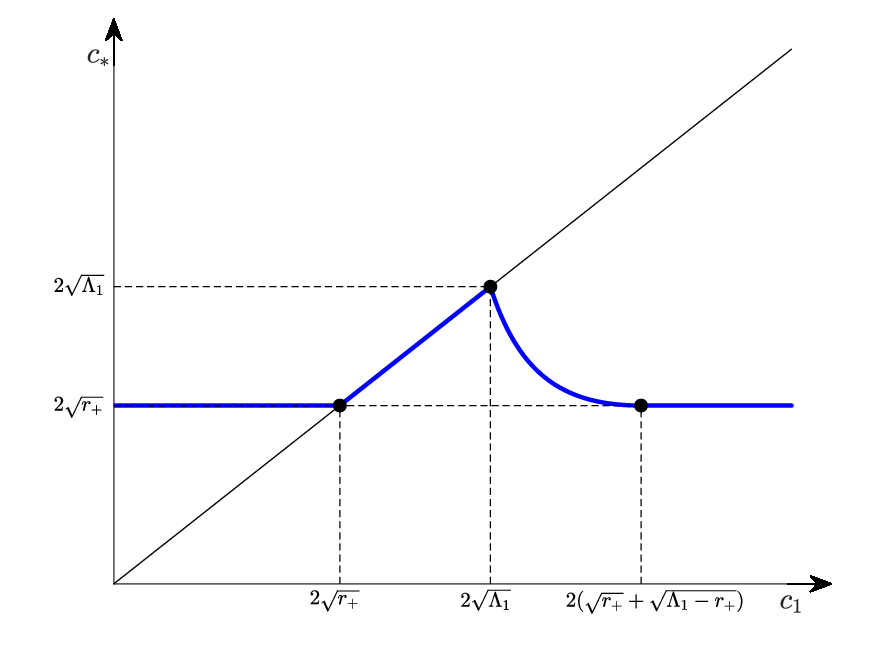}}
 \caption{The dependence of spreading speed $c_*$ on $c_1$. Here $r_\pm = g(\pm \infty)$. The case $c_*= c_1$ is also indicated in \cite{Berestycki2018forced}. Nonlocal pulling is illustrated by the curved part of the blue lines in panels (a), (b), (d) and (f). The part where $c_*$ coincides either with the KPP speed of the limiting system at $\pm\infty$ is indicated by the horizontal part of the blue lines in (a)-(f).}
   \label{fig1}
 \end{figure}

	\begin{remark}
		The case $c_1 \in (2\sqrt{r_+}, 2\sqrt{\Lambda_1})$ is contained in \cite{Berestycki2018forced}. In fact, under this assumption, they proved the existence of a family of forced waves. Moreover, it is proved that solutions with sufficiently fast decaying initial data (including those that are compactly supported) converge locally uniformly to the unique minimal forced wave.
	\end{remark}
		\begin{remark}
			In case that $\Lambda_1=r_->r_+$, then case (iii) in Theorem \ref{thm:main0} is eliminated. 	In case that $\Lambda_1=r_+>r_-$, then case (ii) in Theorem \ref{thm:main0} is eliminated. 	In case that $\Lambda_1=r_-=r_+$, then cases (ii) and (iii) in Theorem \ref{thm:main0} are eliminated. In our former paper \cite{Lam2022asymptotic}, we showed that $c_* = s_{base}$ when $\sup g \le \max\{r_\pm\}$ with three cases:  $r_+>r_-$, $r_+<r_-$  and $r_+=r_-$. Next corollaries extend the validity of $c_* = s_{base}$ to all $g$ such that $\Lambda_1 = \max\{r_\pm\}$. See Figure \ref{fig1}.
	\end{remark}

	\begin{corollary}\label{cor:3.2}
		Suppose $\Lambda_1=r_+ > r_-$ and let $u$ be a solution of \eqref{eq:1.1} with compactly supported, nonnegative, nontrivial initial data,  then
		\begin{equation}\label{eq:cstarold}
		c_*=\overline{c}_* =\underline{c}_* = \begin{cases}
				2\sqrt{r_+} &\text{ if } c_1 \le 2 \sqrt{r_+},\\
				\frac{c_1}{2} - \sqrt{r_+ - r_-} + \frac{r_-}{\frac{c_1}{2} - \sqrt{r_+-r_-}} &\text{ if } 2\sqrt{r_+} < c_1 \le 2(\sqrt{r_-} + \sqrt{r_+ - r_-})\\
				2\sqrt{r_-}  &\text{ if }  c_1 > 2(\sqrt{r_-} + \sqrt{r_+ - r_-})
			\end{cases}
		\end{equation}
  See Figure \ref{fig1}(a) for the dependence of $c_*$ on the shifting speed $c_1$.
	\end{corollary}

		\begin{corollary}\label{cor:3.3}
		Suppose $\Lambda_1=r_->r_+$ and 
$u$ be a solution of \eqref{eq:1.1} with compactly supported, nonnegative, nontrivial initial data,  then
		\begin{equation}\label{eq:cstarold2}
			c_*=\overline{c}_* =\underline{c}_* = \begin{cases}
				2\sqrt{r_+} &\text{ if } c_1 \le 2 \sqrt{r_+},\\
				c_1 &\text{ if } 2\sqrt{r_+} < c_1 \le 2\sqrt{r_-},\\
				2\sqrt{r_-}  &\text{ if }  c_1 > 2 \sqrt{r_-}.
			\end{cases}
		\end{equation}
    See Figure \ref{fig1}(c) for the dependence of $c_*$ on the shifting speed $c_1$.
	\end{corollary}
		\begin{corollary}\label{cor:3.4}
	Suppose $\Lambda_1=r_-=r_+$ and $u$ be a solution of \eqref{eq:1.1} with compactly supported, nonnegative, nontrivial initial data,  then
	\begin{equation}\label{eq:cstarold3}
		c_*=\overline{c}_* =\underline{c}_* = 2\sqrt{\Lambda_1}.
	\end{equation}
 See Figure \ref{fig1}(e) for the dependence of $c_*$ on the shifting speed $c_1$.
\end{corollary}
\begin{remark}
If $c_1\le 0$, then $c_*=2\sqrt{r_+}$ (as $R(s)\equiv r_+$ for $s>0$), see also \cite{HSL2019}.
\end{remark}
Finally, it is easy to obtain the following result on the leftward spreading speed by considering the rightward spreading speed of $v(t,x):=u(t,-x)$ and apply Theorem \ref{thm:main0} and Corollaries \ref{cor:3.2}--\ref{cor:3.4}.
\begin{theorem}\label{thm:main1}
		 Let $u$ be a solution of \eqref{eq:1.1} with compactly supported, nonnegative, nontrivial initial data,  then the leftward spreading speed
		$$
		c^*=\overline{c}^* =\underline{c}^*=\begin{cases}
			2\sqrt{r_-} &\text{ if } c_1 \ge -2\sqrt{r_-},\\
			c_1 &\text{ if } -2\sqrt{r_-} >c_1 \ge -2\sqrt{\Lambda_1},\\
			\frac{c_1}{2} - \sqrt{\Lambda_1 - r_+} + \frac{r_+}{\frac{c_1}{2} - \sqrt{\Lambda_1 - r_+}} &\text{ if } -2\sqrt{\Lambda_1} >c_1 \ge -2(\sqrt{r_+} + \sqrt{\Lambda_1 -r_+}),\\
			2\sqrt{r_+}  &\text{ if }  c_1 <- 2(\sqrt{r_+} + \sqrt{\Lambda_1 -r_+}).
		\end{cases}
		$$
		where
			\begin{equation}\label{eq:speeds~}
			\begin{cases}
				\smallskip
				\overline{c}^*=\inf{\{c>0~|~\limsup \limits_{t\rightarrow \infty}\sup\limits_{x<-ct} u(t,x)=0\}},\\
				\smallskip
				\underline{c}^*=\sup{\{c>0~~|\liminf \limits_{t\rightarrow \infty}\inf\limits_{-ct<x<-ct+1} u(t,x)>0\}}.
			\end{cases}
		\end{equation}
	\end{theorem}

	\begin{proof}[Proof of Theorem \ref{thm:main0}]

		 We will give an explicit formula for the {unique} FL-solution of \eqref{eq:fl} {satisfying} \eqref{eq:ishii2}.
		
		\medskip
		
		\noindent Case (i):  $0<c_1\le  2\sqrt{r_+}$.
		Define \begin{equation*}
			\rho_1(s):=\max\{s^2/4-r_+,0\}.
		\end{equation*}
		It is easy to see that $\rho_1(s)$ is nonnegative and satisfies the first equation of \eqref{eq:fl} in a classical sense. Clearly, $\rho_1$ is automatically a FL-subsolution as $\rho_1(c_1)=0$. Note that $A=\Lambda_1-\frac{c^2_1}{4}\geq r_+-\frac{c^2_1}{4}\ge0$.  Therefore, for any $\psi\in C^1_{pw}$, we have $F_A(\psi'_+(c_1+),\psi'_-(c_1-))\ge A \geq 0$. This implies  that $\rho_1$ is also a FL-supersolution, and hence, $\rho_1(s)$ is a FL-solution of \eqref{eq:fl}.  
		Since $\rho_1$ also satisfies \eqref{eq:ishii2}, $\hat\rho_A = \rho_1$ by uniqueness, 
		It then follows from the definition of $c_*$ in \eqref{eq:sa} that $c_*=2\sqrt{r_+}$.

		\noindent Case (ii):  $2\sqrt{r_+}< c_1\le2\sqrt{\Lambda_1}$. 
		
		Define \begin{equation}\label{e.decrease}
			\rho_2(s):=\begin{cases}
				s^2/4-r_+,& s\ge c_1+ \sqrt{c_1^2-4 r_+},\\
				\max\{\frac{c_1+\sqrt{c_1^2-4r_+}}{2}(s-c_1),0\} & 0\le s \le  c_1+ \sqrt{c_1^2-4r_+}.
			\end{cases}
		\end{equation}
		One could directly check that $\rho_2\in C^1((0,c_1))\cap C^1((c_1,+\infty))$  and  satisfies the first equation of \eqref{eq:fl}. Since $\rho_2(c_1)=0$ and $A\ge0$,  we infer that $\rho_2$ is a FL-solution of \eqref{eq:fl} due to the same reason  as in Case (i). Therefore, by locating the free boundary point in \eqref{e.decrease}, we get
  $c_*=c_1.$

		\noindent Case (iii):  $2\sqrt{\Lambda_1}< c_1\le 2(\sqrt{r_-} + \sqrt{\Lambda_1-r_-})$.
		Set 	\begin{equation}\label{eq:mupmr}
				\mu_+ = \frac{c_1}{2} + \sqrt{\Lambda_1-r_+},\qquad \mu_- = \frac{c_1}{2}-\sqrt{\Lambda_1-r_-}>0.
		\end{equation}
		Define \begin{equation*}
			\rho_3(s):=\begin{cases}
				s^2/4-r_+,& s\ge 2\mu_+,\\
				\mu_+s-(\mu^2_++r_+), & c_1\le s\le  2\mu_+,\\
				\mu_-s-(\mu^2_-+r_-), & c_*\le s\le c_1,\\
				0, &0\le s\le c_*.
			\end{cases}
		\end{equation*}
		where $\mu_+$ and $\mu_-$ are as in \eqref{eq:mupmr}. Noting that $$\rho_3(c_1)=-A=\frac{c^2_1}{4}-\Lambda_1>0.
		$$
		It is easy to check that $\rho_3$ satisfies \eqref{eq:fl} in the classical sense when $s\not=c_1$ or $c_* = \mu_- + r_-/\mu_-$. The FL-subsolution property at $s=c_*$ holds since $\rho_3(c_*) = 0$.
To show it is indeed a FL-subsolution, it suffices to consider the case that $\rho_3-\psi$ attains a global maximum at $s_0=c_1$ for some test function $\psi\in C^1_{pw}$. It then follows from $\rho(s)-\psi(s)\le \rho(c_1)-\psi(c_1)$ for any $s$ close to $c_1$ that $\psi'(c_1+)\ge \rho'_3(c_1+)=\mu_+ \geq \frac{c_1}{2}$, and ${\psi'}(c_1-)\le \rho'_3(c_1-)=\mu_- \leq\frac{c_1}{2}$.
  This implies $$\rho_3(c_1)+H^-(c_1+,\psi'(c_1+))=\rho_3(c_1) + H(c_1+,\frac{c_1}{2}) =\frac{c^2_1}{4}-\Lambda_1-\frac{c^2_1}{4}+r_+=r_+-
		\Lambda_1 \le 0,$$
		and 
		$$\rho_3(c_1)+H^+(c_1-,\psi'(c_1-))=\rho_3(c_1) +  H(c_1-,\frac{c_1}{2}) =\frac{c^2_1}{4}-\Lambda_1-\frac{c^2_1}{4}+r_-=r_--
		\Lambda_1 \le 0,$$
		where we used $\rho_3(c_1)=-A=\tfrac{c_1^2}{4} - \Lambda_1$. Also, we obtain that 
		$$\rho_3(c_1)+F_A(\psi'(c_1+), \psi'(c_1-))\le0.$$
		This yields $\rho_3$ is a  FL-subsolution of \eqref{eq:fl}.

		Next, we verify  that $\rho_3$ is a FL-supersolution. Note that $\rho_3$ is a classical solution (including at the point $s=2\mu_+$) except for two non-differentiable points $s_1=c_1$ and $s_2=c_*=\frac{r_-}{\mu_-}+\mu_-$.  Also, observe from $c_1 \leq 2(\sqrt{r_-} +\sqrt{\Lambda_1-r_-})$ that $\mu_- \le \sqrt{r_-}$, and hence
		\begin{equation}\label{eq:auxr1}
			s_2 = \mu_- + \frac{r_-}{\mu_-}\ge 2\mu_-.
		\end{equation}
		Suppose first that a test function $\psi\in C^1_{pw}$ touches $\rho_3$ from below at $c_1$, then 
$$\rho_3(c_1)+F_A(\psi'(c_1+), \psi'(c_1-))\ge\rho_3(c_1)+A=0.$$
		Suppose next that a test function $\psi\in C^1_{pw}$ touches $\rho_3$ from below at $s_2=\frac{r_-}{\mu_-}+\mu_-$,  it then follows 
 that $0\le \psi'(s_2)\le \mu_-$ (note that $\psi \in C^1$ near $s_2$).$$\rho_3(s_2)+H(s_2,\psi'(s_2))=-s_2\psi'(s_2)+[\psi'(s_2)]^2+r_-\ge -c_*\mu_-+\mu^2_-+r_-=0,$$ where the first inequality is due to $\psi'(s_2) \in [0,\mu_-]$, and that $p \mapsto -s_2 p + p^2 + r_+$ is monotone decreasing in  $[0,\mu_-]$ (thanks to \eqref{eq:auxr1}).
We can then conclude that $\rho_3(s)$ is a FL-solution of \eqref{eq:fl} in $(0,\infty)$, and hence, $c_*=s_2=\frac{r_-}{\mu_-}+\mu_-.$
		
		\medskip
		
		\noindent Case (iv): $c_1> 2(\sqrt{r_-} + \sqrt{\Lambda_1 -r_-})$. Define \begin{equation*}
			\rho_4(s):=\begin{cases}
				s^2/4-r_+,& s\ge 2\mu_+,\\
				\mu_+s-(\mu^2_++r_+), & c_1\le s\le  2\mu_+,\\
				\mu_-s-(\mu^2_-+r_-), & 2\mu_-\le s\le c_1,\\
				\max\{\frac{s^2}{4}-r_-,0\}, &0\le s\le 2\mu_-.
			\end{cases}
		\end{equation*}
		Noting that $\rho_4$ is a classical solution except {at} two points $c_1$ and $s_3=2\sqrt{r_-}$. For $s=c_1$, we could argue similarly to that in the case (iii) to obtain that the junction condition for super- and subsolution hold true.  For $s_3=2\sqrt{r_-}$, 
  $\rho_4(s_3) = 0$ implies that the junction condition for subsolution hold at $s=s_3$.
 Now suppose that $\rho_4-\psi$  attains a global minimum at $s_3$ for some test function $\psi\in C^1_{pw}$, then we infer that $0\le \psi'(s_3)\le \mu_-$ and 
$$\rho_4(s_3)+H(s_3,\psi'(s_3))=-s_3\psi'(s_3)+[\psi'(s_3)]^2+r_-\ge r_--\frac{s^2_3}{4}=0,$$
where we used  $\min_p (-s_3 p + p^2) = -s_3^2/4$ and $s_3 = 2\sqrt{r_-}$. 
		As a consequence, $\rho_4$ is a FL-solution and $c_*=s_3=2\sqrt{r_-}$. 
	\end{proof}
	\begin{proof}[Proof of Corollaries 3.4-3.6]
		By part (b) of Theorem \ref{thm:2.13} implies that $c_* = s_{base}$. Then it is a direct consequence of \cite[Theorem 6 (iv) and Theorem 7]{Lam2022asymptotic}.
	\end{proof}
	

	\section{Preliminary Results} \label{sec:aux}

We will give some preliminary results in this section in preparation of the proof of the main result (Theorem \ref{thm:2.13}) in the next section.

	To study the behavior of $u$ at the leading edge, we consider the rate function 
	\begin{equation}\label{eq:ratefunc}
		w^\ep(t,x) = -\ep \log u^\ep(t,x).
	\end{equation}
	We first observe that $\hat{w}(t,x):=\lim_{\ep \to 0^+}w^\ep(t,x)$, if exists, is $1$-homogeneous.
	\begin{lemma}\label{lem:1hom}
		Suppose $w^\ep \to \hat w$ in $C_{loc}((0,\infty)\times [0,\infty))$, as $\ep\to 0^+$, then $\hat w(t,x) = t\hat\rho(x/t)$ for some function $\hat\rho.$
	\end{lemma}
	\begin{proof}
		Fix a constant $A>0$, then
		$$
		\hat w(At, Ax) = \lim_{\ep \to 0^+} -\ep \log u\left( \frac{At}{\ep}, \frac{Ax}{\ep}\right) =A \lim_{(\ep/A) \to 0^+} - (\ep/A)  \log u\left( \frac{t}{\ep/A}, \frac{x}{\ep/A}\right) = A \hat w(t,x).
		$$
		And the lemma follows if we take $\hat\rho(s) := \hat{w}(1,s)$.
	\end{proof}
	Suppose the limit function $\hat{w}(t,x) = t\hat\rho(x/t)$ exists, and define \begin{equation}\label{eq:hats1}
		\hat  s = \sup\{s \geq 0:~ \hat \rho(s) = 0\},\end{equation} then  we immediately have 
	$$
	u^\ep(t,x) = \exp\left( -\frac{w^\ep(t,x)}{\ep} \right) =  \exp\left( -\frac{\hat w(t,x) + o(1)}{\ep} \right)  =  \exp\left( -\frac{t\hat \rho(x/t) + o(1)}{\ep} \right) \to 0  
	$$
	whenever $x/t > \hat s$. This gives the first part of \eqref{eq:spread1}. 
	
	Furthermore, it can be shown (e.g. \cite[Lemma 3.1]{Liu2021asymptotic} or \cite[Sect. 4]{Evans1989pde}) that $u^\ep(t,x)$ is bounded away from zero in the interior of $\{(t,x): \hat w(t,x) = 0\}$, i.e. the second part of \eqref{eq:spread1} holds. Hence, the study of the spreading speed $c_*$ reduces to the determination of the free boundary point $\hat s$ of $\hat \rho$, given in \eqref{eq:hats1}. Next,  we collect the properties of the eigenvalue problem \eqref{eq:pev0} as well as a few technical results for Hamilton-Jacobi equations with general Hamiltonians $H(s,p)$.
	


 	\subsection{The eigenvalue problem associated with $g(y)$}\label{sec:eigen}

	Note that $s_{base}$ depends only on the values of $g(\pm \infty)$. 
	The next questions are if and when the invasion is enhanced by the specific profile of $g$. The answer is completely determined by the eigenvalue $\Lambda_1$ given by \eqref{eq:lambda1}.	
	This and several other notions of principal eigenvalues are analyzed in \cite{Berestycki2014generalizations}. Here, we recall some basic properties of $\Lambda_1$ and the associated positive eigenfunction.

	\begin{proposition}
		\label{prop:2.1}
		Let $\Lambda_1$ be given by \eqref{eq:lambda1}.     \begin{itemize} 
			\item[{\rm(a)}] Then $\Lambda_1 \geq \max\{g(-\infty), g(+\infty)\}$ and the eigenvalue problem \eqref{eq:pev0} has a postive solution in $C^2_{loc}(\mathbb{R})$ if and only if $\Lambda \in [\Lambda_1,\infty)$. 
			\item[{\rm(b)}]
			If, in addition, $\Lambda_1 > \max\{g(-\infty), g(+\infty)\}$, then $\Lambda_1$ is a simple eigenvalue of \eqref{eq:pev0} and the following statements hold. 
			\begin{itemize} 
				\item[{\rm(i)}] Let $\lambda_\pm:=\sqrt{\Lambda_1-g(\pm\infty)}$ and $\Phi_1(y)$ be the positive eigenfunction corresponding to $\Lambda =\Lambda_1$, then 
				$$
				\Phi_1(y) = \exp( - \lambda_\pm y + o(y))  \quad \text{ as }y \to \pm \infty,
				$$
				i.e. 
				for any sufficiently small $\eta>0$, there exist positive numbers $\overline{C}_\eta, \underline{C}_\eta$, such that
				\begin{equation}\label{est-decay-rate}
					\begin{cases}
						\underline{C}_\eta e^{-(\lambda_++\eta) y}\le \Phi_1(y)\le \bar C_\eta e^{-(\lambda_+-\eta) y} & \text{ if }\, y\ge0,\\
						\underline{C}_\eta e^{(\lambda_-+\eta) y}\le \Phi_1(y)\le \bar C_\eta e^{(\lambda_--\eta) y} &\text{ if }\, y\le 0.
					\end{cases}
				\end{equation}


				\item[{\rm(ii)}] Suppose \eqref{eq:pev0} has a positive eigenfunction $\tilde\Phi \in C^{2}_{loc}(\mathbb{R})$ for some $\tilde\Lambda \in \mathbb{R}$ such that $\tilde\Phi \to 0$ as $|y| \to \infty$, then $\tilde\Lambda = \Lambda_1$ and $\tilde\Phi \in {\rm span}\,\{\Phi_1\}$.
			\end{itemize}
	\item[{\rm(c)}]If $\Lambda_1=\max\{g(-\infty),g(+\infty)\}$. Then for any $\eta>0$, there exists $g_\eta:\mathbb{R}\to \mathbb{R}$, such that 
$$\begin{cases}
	g_\eta(x)=g(x) \quad \text{ for all }\quad |x|\gg 1,\\
	g(x)\le g_\eta(x)\le g(x)+\eta  \quad \text{ for all }\quad x\in\mathbb{R},\\
	\Lambda_1^{\eta}:=\Lambda_1(g_\eta) \quad \text{ satisfies } \quad \Lambda_1^{\eta}>\max\{g(\pm\infty)\}.
\end{cases}$$
		\end{itemize}
	\end{proposition}
	For the convenience of the reader, we provide the proof of the above result in Appendix \ref{sec:prop21}.

 The following result describes the effect of $\Lambda_1$ in enhancing the spreading speed $c_*$.

	\begin{corollary}\label{cor:main0}
		Let $\Lambda_1$ be given by \eqref{eq:lambda1} and let $s_{base}$ be given by Section \ref{sec:ishii}.
		\begin{itemize}
			\item[{\rm(a)}] If $\Lambda_1 = \max\{g(\pm \infty)\}$, then for each $c_1>0$, we have $c_* = s_{base}$.
			\item[{\rm(b)}] If $\Lambda_1 > \max\{g(\pm \infty)\}$, then for some $c_1>0$, we have  $c_* > s_{base}$.
		\end{itemize}
	\end{corollary}
	\begin{proof}
 Statement (a) follows from Theorem \ref{thm:2.13}. Statement (b) is a direct consequence of Theorem \ref{thm:main0}.
	\end{proof}

	

	\subsection{The continuity of subsolutions}\label{sec:4.4}
	
	
	We discuss the weak continuity condition for sub-solutions, which are half-relaxed limits of solutions to reaction-diffusion equations. This property first appeared in \cite{Barles1990comparison}. 	
	\begin{lemma}\label{cor:im1rho'}
		Suppose $\underline\rho$ is nonnegative and  satisfies $\underline\rho(0)=0$, and satisfies 
		$$
		\min\{ \rho, \rho+ H(s,\rho')\} \leq 0 \qquad  \text{ in }[0,\infty)
		$$
		in viscosity sense (of Ishii).
		\begin{itemize}
			\item[{\rm(a)}] If $H(s,0) >0$ for each $s>0$, then $\underline\rho$ is nondecreasing. 
			\item[{\rm(b)}] If $\lim\limits_{|p| \to \infty} \inf\limits_{s\in K} H(s,p) \to \infty$ for each compact set $K\subset [0,\infty)$, then $\underline\rho \in {\rm Lip}_{loc}([0,\infty))$. 
			In particular, it satisfies the weak continuity condition:
				\begin{equation}\label{eq:weakcontrho}
				\underline \rho(c_1) = \limsup_{s \to c_1+} \underline \rho(s) \quad \text{ and } \quad \underline \rho(c_1) = \limsup_{s \to c_1-} \underline \rho(s).
			\end{equation}
		\end{itemize}
		
	\end{lemma}
	\begin{proof}
		Part (a) is due to \cite[Lemma 2.9]{Lam2022asymptotic}.
		For Part (b), fix a bounded interval $K=[0,\bar s]$ with $\bar s>0$, and let $M>0$ be given such that
		$$
		H(s,p)>0 \quad \text{ for all }s \in [0,\bar s +1],~|p|\geq M. 
		$$
		Fix any point $s_0 \in [0,\bar s]$, we claim that
		\begin{equation}\label{eq:lippp}
			\underline\rho(s) - \underline\rho(s_0) \leq M|s-s_0| \quad \text{ for all }s\in [0,\bar s].  
		\end{equation}
		For this purpose, define
		$$
		\Psi(s) = \begin{cases}
			\alpha + M|s-s_0| &\text{ for }s \in [0,\bar s],\\
			\alpha + M\left[|\bar s-s_0| + \frac{1}{\bar s + 1-s}-1\right] &\text{ for }s \in [\bar s,\bar s+1),
		\end{cases}
		$$
		for any $\alpha>0$, then $\Psi$ is continuously differentiable except at $s_0$, such that 
		\begin{equation}\label{eq:lippp1}
			|\Psi'(s)| \geq M \quad \text{ for }s \in [0,\bar s + 1)\setminus\{s_0\}.
		\end{equation}
		Next, take the minimal $\alpha$ such that $\Psi$ touches $\underline\rho$ from above at some point $s_1\in[0,\bar s+1)$, {which is possible since $\underline{\rho}$ is upper semicontinuous, and thus bounded on $[0,\bar s+1]$.}
		
		If $s_1 = s_0$, then $\alpha = \rho(s_0)$ and we obtain \eqref{eq:lippp}. 
		Suppose to the contrary that $s_1\neq s_0$.  We first observe that $\underline\rho(s_1)>0$, which follows from $\Psi(s_0)<\Psi(s)$ for $s\in[0,\bar s+1)$ and 
  $$
  \underline{\rho}(s_1) = \Psi(s_1) > \Psi(s_0) \geq \underline{\rho}(s_0) \geq 0.
  $$

		
		 Now  $\underline{\rho}(s_1)>0$, so the definition of viscosity subsolution implies that
		$$
		0 \geq \underline\rho(s_1) + H(s_1, \Psi'(s_1)).
		$$
But the right hand side is strictly positive thanks to \eqref{eq:lippp1} and the choice of $M$.  This is a contradiction.  Therefore, \eqref{eq:lippp} is proved.
		
		Since $M$ depends only on $K=[0,\bar s]$ but not on $s_0$, we can reverse the role of $s$ and $s_0$ in \eqref{eq:lippp} to conclude that
		$$
		|\underline\rho(s)-\underline\rho(s_0)| \leq M|s-s_0| \quad \text{ for all }s,s_0 \in [0,\bar s].
		$$
		This proves the Lipschitz continuity of $\underline\rho$ in any compact subset of $[0,\infty)$.
	\end{proof}

	\subsection{Critical slope lemmas}\label{sec:4.5}
	
	
	Let $U$ be an open interval in $\mathbb{R}$ containing $c_1$, and recall that
	$$
	C^1_{pw}(U) = C(U) \cap C^1(U \cap (-\infty,c_1])  \cap C^1(U \cap [c_1,\infty)).
	$$
		\begin{description}
			\item[(HH)] Assume $p \mapsto H(s,p)$ is convex and coercive,
   and $H(c_1+,p)$ and $H(c_1-,p)$ exist. 
	\end{description}
	\begin{lemma}\label{cor:subsolr}
		Assume that {\bf(HH)} holds.
  Let $\underline\rho: U \to [0,\infty)$ satisfy the following: 
		\begin{itemize}
			\item[{\rm(i)}] $\underline\rho$ is a viscosity subsolution of 
			$$
			\min\{\rho,\rho + H(s,\rho')\}=0 \quad \text{ in }\{s \in U:~ s>c_1\}.
			$$
			\item[{\rm(ii)}] $\underline\rho$ satisfies the weak continuity condition \eqref{eq:weakcontrho}.
			\item[{\rm(iii)}] $\underline \rho(c_1) >0$.
		\end{itemize}
		Suppose 
		there is a test function 
		$\varphi \in C^1_{pw}(U)$ that touches $\underline \rho$ from above only at $c_1$. Let $p_+ = \varphi'(c_1+)$, and 
		\begin{equation}\label{eq:crit_barp1r}
			\overline{p}_+:=\inf\left\{\overline p \in\mathbb{R}:~ \exists r>0,~ \varphi(s) + \overline p (s-c_1) \geq \underline \rho(s) ~\text{ for } 0\le s-c_1<r\right\}.
		\end{equation}
		Then $-\infty<\bar{p}_+ \leq 0$ and
		$$
		\underline\rho(c_1) + H(c_1+,p_+ + \bar{p}_+) \leq 0.
		$$
	\end{lemma}
	The proof is a modification of that in \cite[Lemmas 2.9 and 2.10]{Imbert2017flux}, where we use the weak continuity condition.	 	For the convenience of the reader, we provide the proof in Section \ref{corsubsup}.		

	\begin{remark}\label{rem:subsoll}
		Suppose, in addition, 
  that $\underline\rho$ is a viscosity subsolution of 
		$\min\{\rho,\rho + H(s,\rho')\}=0$ in $\{s \in U:~ s<c_1\}.$ Then 
		$$
		\underline\rho(c_1) + H(c_1-,-p_- - \bar{p}_-) \leq 0,
		$$
		where $-p_- = \varphi'(c_1-)$ and $\overline{p}_- \in (-\infty,0]$ is given by
		\begin{equation}\label{eq:crit_barp1'r}
			\overline{p}_-:=\inf\left\{\overline p \in\mathbb{R}:~ \exists r>0,~ \varphi(s) - \overline p (s-c_1) \geq \underline \rho(s) ~\text{ for } -r<s-c_1\le 0\right\}.    
		\end{equation}
	\end{remark}

	\begin{lemma}\label{cor:supersolr}
		Assume that {\bf(HH)} is valid.	Suppose $\overline\rho$ is a viscosity supersolution of 
		$$
		\rho + H(s,\rho')=0 \quad \text{ in }\{s \in U:~ s>c_1\}.
		$$
		and there is a test function 
		$\varphi \in C^1_{pw}(U)$ that touches $\overline\rho$ from below only at $c_1$. Let $p_+ = \varphi'(c_1+)$, and 
		\begin{equation}\label{eq:crit_barp2r}
			\underline{p}_+:=\sup\left\{\underline p \in\mathbb{R}:~ \exists r>0,~ \varphi(s) + \underline p (s-c_1) \leq \overline\rho(s) ~\text{ for }0\le s-c_1<r\right\}.
		\end{equation}
		If $\underline{p}_+ <\infty$, then
		$$
		\overline\rho(c_1) + H(c_1+,p_+ + \underline{p}_+) \geq 0, \quad \text{ with }\quad \underline{p}_+ \geq 0.
		$$
	\end{lemma}
	 Note that no weak continuity condition is needed as we do not claim the finiteness of $\underline{p}_+$. The proof of Lemma \ref{cor:supersolr} is also included in Appendix \ref{corsubsup}.
	\begin{remark}\label{rem: supsoll}
		Suppose, in addition, $\overline\rho$ is a viscosity supersolution of 
		$\min\{\rho,\rho + H_-(s,\rho')\}=0$ in $\{s \in U:~ s<c_1\}.$ Let $-p_- = \varphi'(c_1-)$ and 
		\begin{equation}\label{eq:crit_barp2'r}
			\underline{p}_-:=\sup\left\{\underline p \in\mathbb{R}:~ \exists r>0,~ \varphi(s) - \underline p (s-c_1) \leq \overline \rho(s) ~\text{ for } -r<s-c_1\leq0\right\}.
		\end{equation} If $\underline{p}_- <+\infty$, then
		$$
		\overline\rho(c_1) + H(c_1-,-p_- - \underline{p}_-) \geq 0 \quad \text{ and }\quad \underline{p}_- \geq 0.
		$$ 
	\end{remark}

	\section{Proof of Main Results}
\label{sec:proof}

 \begin{lemma}\label{lem:spread1}
  Let $u$ be a nonnegative, nontrivial solution of \eqref{eq:1.1} and assume $\inf g>  \delta_0>0$. Then for any $\eta\in(0,2\sqrt{\delta_0})$
		\begin{equation}\label{eq:spread1'}
			\liminf_{t\to\infty} \inf_{|x|<(2\sqrt{\delta_0} - \eta)t} u(t,x) \geq \delta_0. 
		\end{equation}
	\end{lemma}
	\begin{proof}
		Choose $\delta_1 \in (\delta_0,\inf g)$, then
		$u$ is a subsolution to
		$$
		\tilde{u}_t - \tilde{u}_{xx} = \tilde{u}(\delta_1 - \tilde{u}) \quad \text{ in }(0,\infty)\times\mathbb{R}
		$$
		with compactly supported initial data $u_0$. It is a classical result that \eqref{eq:spread1'} holds for $\tilde u$. It follows by comparison principle that \eqref{eq:spread1'} holds also for $u$.
	\end{proof}
 
	Fix a solution $u(t,x)$ of \eqref{eq:1.1}, and let
	$$
	u^\ep(t,x) = u\left(\frac{t}{\ep},\frac{x}{\ep} \right),\quad \text{ and }\quad w^\ep(t,x) = -\ep \log u^\ep(t,x), 
	$$
	then $w^\ep$ satisfies
	\begin{equation}\label{eq:wepr}
		\begin{cases}
			w^\ep_t - \ep w^\ep_{xx} + |w^\ep_x|^2 + g\left( \frac{x- c_1 t}{\ep}\right) - e^{-w^\ep/\ep} = 0 &\text{ for }(t,x) \in (0,\infty)\times \mathbb{R},\\
			w^\ep(0,x) =\begin{cases}
				-\ep \log u_0(x/\ep) &\text{ if } x/\ep \in {\rm Int}({\rm supp}\,u_0),\\
				+\infty &\text{ otherwise. }
			\end{cases} 
		\end{cases}
	\end{equation}
	Consider the half-relaxed limits \cite{Barles1990comparison}:
	\begin{equation}\label{eq:wstars}
		w^*(t,x) = \limsup_{ \ep \to 0 \atop (t',x') \to (t,x) }w^\ep(t,x) \quad \text{ and }\quad w_*(t,x) = \liminf_{ \ep \to 0 \atop (t',x') \to (t,x) }w^\ep(t,x) 
	\end{equation}
	\begin{lemma}\label{lem:3.1}
		Let $w^*$ and $w_*$ be given as above. Then $w^*(t,x) = t\rho^*(x/t)$ and $w_*(t,x)=t\rho_*(x/t)$, for some upper semicontinuous function $\rho^*$ and lower semicontinuous function $\rho_*$.
	\end{lemma}
	\begin{proof}
		The existence of $\rho^*$ and $\rho_*$ is similar to Lemma \ref{lem:1hom} and is omitted. The semicontinuity are due to the half-relaxed limits in the definition of $w^*,w_*$.
	\end{proof}
	\begin{lemma}\label{lem:5.2}
		$\rho^*(s) \geq \rho_*(s) \geq 0$ for all $s\geq 0$ and $\rho^*(0)= 0$. Moreover, $\rho_*(s)/s\to +\infty$ as $s\to+\infty$.
	\end{lemma}
	\begin{proof}
		By maximum principle, one can establish uniform upper bound of $u$, i.e. $u(t,x) \leq M_0:=\max \{\sup |u_0|, \sup g\}$, so that 
		$w^\ep(t,x) \geq -\ep \log M_0.$ This implies $w^*\geq w_* \geq 0$ and hence $\rho^* \geq \rho_* \geq 0.$

	To show $\rho^*(0)=0$, it suffices to prove $w^*(t,0) = 0$ for all $t>0$. By Lemma \ref{lem:spread1}, we have
		$$
		w^*(t,0) = \limsup_{\ep \to 0\atop (t',x') \to (t,0)} w^\ep(t',x') \leq \limsup_{\ep \to 0\atop (t',x') \to (t,0)} -\ep \log u\left( \frac{t'}{\ep}. \frac{x'}{\ep}\right) \leq -\lim_{\ep \to 0} \ep \log \delta_0 =0.
		$$
		Finally,	by similar argument to that in \cite[Lemma B.3]{Lam2022asymptotic}, we have $w_*(0,x)=+\infty$ for all $x>0$. It then follows from lower semicontinuity of $w_*$ that
		$$\liminf_{s\to+\infty}\frac{\rho_*(s)}{s}=\liminf_{s\to+\infty}w_*\left(\frac{1}{s},1\right)\ge w_*(0,1)=+\infty.$$
		This completes the proof.   
	\end{proof}
	
	\subsection{Verification of flux-limited solutions property
	}
	The main result of this subection is the following.
	\begin{proposition}
		\label{prop:5.3} 
		Let $\rho^*$ and $\rho_*$ be given in Lemma \ref{lem:3.1}. Then $\rho^*$ (resp. $\rho_*$) is a FL-subsolution (resp. FL-supersolution) of \eqref{eq:fl} with {\it flux limiter} given by \eqref{eq:A}, that is,
		$A = \Lambda_1 - \frac{|c_1|^2}{4}.$ 
	\end{proposition}
	
	We divide the proof of Proposition \ref{prop:5.3} into the verification of FL-subsolution and supersolution.
	
	\begin{lemma} \label{lem:flsub}
		Let $\rho^*$ be given by Lemma \ref{lem:3.1}. Then
		\begin{itemize}
			\item[{\rm(a)}] $\rho^* \in {\rm Lip}_{loc}([0,\infty))$;  
			\item[{\rm(b)}] $\rho^*$ satisfies the weak continuity condition \eqref{eq:weakcontrho};
			\item[{\rm(c)}] $\rho^*$ is a FL-subsolution of \eqref{eq:fl} with $A=\Lambda_1-\frac{c_1^2}{4}$.
		\end{itemize}
	\end{lemma}
	\begin{proof}
		By construction $\rho^*: [0,\infty) \to [0,\infty)$ is upper semicontinuous. 
		It is standard to show that $w^*(t,x)=t\rho^*(x/t)$ is a viscosity subsolution to
		\begin{equation}\label{eq:wr}
	\begin{cases}
		\min\{w^*,w^*_t + |w^*_x|^2 + g(-\infty)\} = 0  &\text{ in }\{(t,x): 0< x < c_1t\},\\ 
		\min\{w^*,w^*_t + |w^*_x|^2 + g(+\infty)\} = 0  &\text{ in }\{(t,x): x > c_1t>0\},\\
		\min\{w^*,w^*_t + |w^*_x|^2 + \inf_{y\in\mathbb{R}} g(y)\} = 0  &\text{ in }(0,+\infty)\times(0,+\infty).
	\end{cases}
\end{equation}
From the third equation, we deduce as in \cite[Lemma 2.3]{Liu2021stacked} that, in viscosity sense, 
		\begin{equation}
			\min\{\rho^*,\rho^* -s(\rho^*)' + |(\rho^*)'|^2 + \inf_\mathbb{R} g\}\le 0 \quad \text{ in }(0,\infty).
		\end{equation}
		Since also $\rho^*(0)=0$ (thanks to Lemma \ref{lem:5.2}), we infer from Lemma \ref{cor:im1rho'} that $\rho^* \in {\rm Lip}_{loc}([0,\infty))$. This proves assertion (a). Assertion (a) implies (b).
		
		The proof of (c) is inspired by \cite{Guerand2017effective}. From the first two equations of \eqref{eq:wr}, we deduce  that, in viscosity sense, 
		$$
		\min\{\rho^*,\rho^* + H(s,(\rho^*)')\} \leq 0 \quad \text{ for }s \in (0,\infty)\setminus c_1,
		$$
		with $H(s,p)$ given in \eqref{eq:ishii3}.
		It remains to show that $\rho^*$ is a subsolution to the second equation of \eqref{eq:fl}.  For this purpose, let $\psi \in C^1_{pw}$ and suppose $\rho^* - \psi$ has a strict global maximum point\footnote{See \cite[Proposition 3.1]{Barles2013introduction} for  several equivalent definitions of viscosity solution.} at $c_1$, and that $\psi(c_1) =\rho^*(c_1) >0$.  Denote 
		$$
		A = \Lambda_1 - \frac{c_1^2}{4},\qquad \lambda = -\rho^*(c_1) = -\psi(c_1),\qquad p_+ = \psi'(c_1+),\qquad p_- = -\psi'(c_1-).
		$$
		(Note the negative sign in the definition of $p_-$.)    We want to show
		\begin{equation}\label{eq:flsub}
		-\lambda+\max\left\{A, H^-(c_1+,p_+), H^+(c_1-,-p_-) \right\} \leq 0.
		\end{equation}
  {{(Observe that if $A \geq 0$, then any nonnegative (sub)solution $\rho^*(s)$ satisfying the junction condition $\min\{\rho, \rho + F_A(\rho'(c_1+), \rho'(c_1-)) \leq 0$ must vanish at the point $c_1$, i.e. the case $\rho^*(c_1)>0$ is null.)}}
		We first claim that
		\begin{equation}\label{eq:fl2}
			H^-(c_1+,p_+) \leq \lambda, \quad \text{ and }\quad H^+(c_1-,-p_-)  \leq \lambda.
		\end{equation}
		By Lemma \ref{cor:subsolr}, 
		\begin{equation}\label{eq:critslope1}
			H(c_1+,p_+ + \bar{p}_+) \leq \lambda \quad \text{ for some }\bar{p}_+ \in (-\infty,0].
		\end{equation}
		Hence, using the fact that $H^-$ is decreasing in $p$ and $H^- \leq H$, 
		$$
		H^-(c_1+,p_+) \leq H^-(c_1+,p_+ + \bar{p}_+) \leq H(c_1+,p_+ + \bar{p}_+) \leq \lambda.
		$$
		Arguing similarly, we also have $H^+(c_1-,-p_-)  \leq \lambda$. This proves \eqref{eq:fl2}. 
		
		It remains to show
		$A\leq \lambda$, where $A = \Lambda_1 - \frac{c_1^2}{4}$.
		
		Suppose for contradiction $A > \lambda$.  Then by \eqref{eq:fl2}, we have 
		\begin{equation}\label{eq:Amin}
			A > \max\{\min_p H(c_1 \pm,\,p)\} = \max\{g(\pm \infty)\} - \frac{c_1^2}{4},
		\end{equation}
		In particular, $\Lambda_1>\max\{g(\pm\infty)\}$.
		Define
		\begin{equation}\label{eq:phi0r}
			\phi_0(x) = {\mu_+}\max\{x,0\} - {\mu_-} \min\{x,0\},
		\end{equation}
		where 
		\begin{equation}\label{eq:mu-pm}
			\mu_+ = \frac{c_1}{2}+\sqrt{\Lambda_1-g(+\infty)},\quad \text{ and }\quad \mu_- = \frac{-c_1  }{2}+\sqrt{\Lambda_1-g(-\infty)}.
		\end{equation}
		Note that $\mu_\pm$ are also determined uniquely (thanks to \eqref{eq:Amin}) by  
		\begin{equation}\label{eq:mupm}
			\begin{cases}
				H(c_1+,\mu_+) = A & \text{ and }\quad \mu_+ \geq {\rm argmin}\, H(c_1+,\cdot),\\
				H(c_1-,-\mu_-) = A&\text{ and }\quad -\mu_- \leq {\rm argmin}\, H(c_1-,\cdot).\end{cases}
		\end{equation}
		By \eqref{eq:critslope1} and that $\lambda<A=H(c_1+,{\mu_+})$, we have $H(c_1+,p_+ + \bar{p}_+) < H(c_1+,{\mu_+})$ and thus $p_+ + \bar{p}_+ < {\mu_+}$ (here we have used the  fact that ${\mu_+}$ is the larger root of $H(c_1+, p) = A$). By definition of $\bar{p}_+$ in \eqref{eq:crit_barp1r}, we deduce that there  exists a small neighborhood $(c_1-r,c_1+r)$ of $c_1$ (with $0<r<\min\{1,c_1\}$) such that
(by Lipschitz continuity of $\rho^*$) 		$$
		\frac{\rho^*(c_1)}{2}<\rho^*(s) \leq  -\lambda + \phi_0(s-c_1) \quad \text{ for }c_1 \leq s < c_1+r,
		$$
		with the second {inequality being an equality} iff $s = c_1$.
		By arguing similarly, along with the definition of $\bar p_-$ in \eqref{eq:crit_barp1'r}, we have
		$$
			\frac{\rho^*(c_1)}{2}<\rho^*(s) \leq  -\lambda + \phi_0(s-c_1) \quad \text{ for }c_1-r < s \leq c_1,
		$$
		with the second equality holds iff $s=c_1$. In other words, $-\lambda+\phi_0(s-c_1)$ is also a test function touching $\rho^*$ from above at $c_1$ in $(c_1-r,c_1+r)$. Hence, letting 
		\begin{equation}\label{def:Qr}
			Q_r:=\{(t,x): x/t \in [c_1-r, c_1+r],~ |t-1|<r\},
		\end{equation}
		we get
		$$
		w^*(t,x)=t \rho^*(x/t) \leq \varphi^0(t,x):=\frac{A-\lambda}{4}(t-1)^2 
		+ t\left(-\lambda + \phi_0\left( \frac{x-c_1t}{t}\right)\right) 
		\quad \text{ in } Q_r
		$$
		with equality iff $(t,x) = (1,c_1)$. We can then choose $\delta\in(0,(1-r)\rho^*(c_1))$ small such that
		\begin{equation}\label{eq:b12rr}
			w^*(t,x) + \delta \leq \varphi^0(t,x)  \quad \text{ on }\partial Q_r.
		\end{equation}
		Next, define 
		$$
		\varphi^\ep(t,x):= \frac{A-\lambda}{4}(t-1)^2 -t\lambda - \ep \log \Psi\left(\frac{x-c_1t}{\ep}\right) \quad \text{ for }(t,x) \in Q_r,
		$$
		where $\Psi(y) = e^{-\frac{c_1 y}{2}}\Phi_1(y)$ and $\Phi_1$ is the positive eigenfunction given in Proposition \ref{prop:2.1}. Thanks to Proposition \ref{prop:2.1}(b)(i),
		{one has} $\varphi^\ep \to \varphi^0$ in $C_{loc}$ since
		$$
		-\ep \log\left[ e^{-\frac{c_1}{2}\frac{x-c_1t}{\ep}} \Phi_1\left(\frac{x-c_1t}{\ep}\right)\right] \to \phi_0(x-c_1t) = t\phi_0\left(\frac{x-c_1t}{t}\right)\quad \text{ locally uniformly.}
		$$ 
Hence, we deduce from \eqref{eq:b12rr} that 
		\begin{equation}\label{eq:bdrycondr}
			w^\ep(t,x) + \delta/2 \leq \varphi^\ep(t,x) \quad \text{ on }\partial Q_r,
		\end{equation}
		for sufficiently small $\ep$. 
		Next, we observe that $\varphi^\ep-\delta/2$ satisfies
		\begin{align*}
			&\quad     \varphi^\ep_t - \ep \varphi^\ep_{xx} + |\varphi^\ep_x|^2  + g\left(\frac{x-c_1t}{\ep}\right) - e^{-\frac{(2\varphi^\ep-\delta)}{2\ep}}= \frac{A-\lambda}{2}(t-1)-\lambda + A + o(1) >0, 
		\end{align*}
		where we used  
		\begin{eqnarray*}
		&&\varphi^\ep(t,x)-\frac{\delta}{2} \ge t\left(-\lambda +  \phi_0\big(\frac{x}{t} - c_1\big)\right)-\frac{\delta}{2} + o(1)\\
		&&\geq (1-r)\rho^*\left(\frac{x}{t}\right)-\frac{\delta}{2}+o(1)\ge\frac{(1-r)\rho^*(c_1)}{2}+o(1)>0 \text{ for } (t,x)\in Q_r
		\end{eqnarray*} to deduce  $e^{-{\frac{2\varphi^\ep-\delta}{2\ep}}} =o(1)$ in the last equality, and $|t-1|\leq r <1$ in the strict inequality.
		Hence $\varphi^\ep-\delta/2$ is a supersolution to the equation \eqref{eq:wepr} of $w^\ep$. In view of the boundary condition \eqref{eq:bdrycondr}, the comparison principle yields
		\begin{equation}\label{eq:1.33rr}
			w^\ep(t,x) + \delta/2 \leq \varphi^\ep(t,x) = 
   \varphi^0(t,x)
   + o(1) \quad \text{ in }Q_r.
		\end{equation}
		By definition of $w^*(1,c_1)=\rho^* (c_1) = -\lambda$ (recall that $\psi$ touches $\rho^*$ from above), there exists $(t^\ep,x^\ep) \to (1,c_1)$ such that $w^\ep(t^\ep,x^\ep) \to -\lambda$. Substituting $(t,x) = (t^\ep,x^\ep)$ into \eqref{eq:1.33rr} and letting $\ep \to 0$, we have
		$$
		-\lambda+ \delta/2 \leq \varphi^0(1,c_1) = -\lambda \quad \text{ for some }\delta>0,
		$$
		which leads to  a contradiction. Therefore, $A \leq \lambda$. This concludes the proof.
	\end{proof}

	Next, we show the FL-supersolution property of {$\rho_*$}.
	\begin{lemma}\label{lem:flsupp}
		The lower limit $\rho_*$ is a FL-supersolution of \eqref{eq:fl} with $A = \Lambda_1 - \frac{c_1^2}{4}$. 
	\end{lemma}
	\begin{proof}
		Again, it is standard to check that  
		$w_*$ is a viscosity supersolution to the first two equations of \eqref{eq:wr} in the viscosity sense. This implies again, by \cite[Lemma 2.3]{Liu2021stacked} that $\rho_*$ is the viscosity supersolution of the first equation of \eqref{eq:fl}. 
		
		It remains to verify the remaining junction condition of \eqref{eq:fl}. Suppose  there is a test function 
		$\psi \in C^1_{pw}$ that touches $\rho_*$ from {below} only at $s=c_1$, and denote 
		$$
		\lambda = - \rho_*(c_1)=-\psi(c_1), \qquad p_+ = \psi'(c_1+), \qquad  p_- = -\psi'(c_1-).
		$$
		By way of contradiction, we further assume that
		\begin{equation}\label{eq:supfl2}\max\{A, H^-(c_1+,p_+), H^+(c_1-,-p_-)\}< \lambda.
		\end{equation}
  Let $A_0 =\max\{g(\pm\infty)\} - |c_1|^2/4$, 
  $\mu_\pm$ be given in terms of $A \in [A_0,\infty)$ as in \eqref{eq:mupm}, and $\phi_0$ be defined as in \eqref{eq:phi0r}.  By Proposition \ref{prop:2.1}(c), for any $\eta>0$,  there exists $g_\eta\in C(\mathbb{R},\mathbb{R})$ such that $\Lambda^\eta_1>\max\{g(\pm\infty)\}$ and $\|g-g_\eta\|_{\infty}\le \eta$. In particular, if $A>A_0$ (or say $\Lambda_1>\max\{g(\pm\infty)\}$), we could set $\eta=0$ and $g_0\equiv g$.  Let $\Psi^\eta(y)=e^{-\frac{c_1}{2}y}\Phi^\eta_1(y)$ where $\Phi^\eta_1(y)$ is a positive and bounded eigenfuction associated with  $\Lambda^\eta_1$. For any $\eta\ge 0$, we could also define $\mu^\eta_+$ and $\mu^\eta_-$  with $\Lambda_1$ replace by $\Lambda^\eta_1$ in \eqref{eq:mu-pm}. Accordingly, $\phi^\eta_0$ would be defined in a similar manner. 	We first prove the following claim.
		

		\begin{claim}\label{claim:6.8}
  There exists $\eta_0>0$, such that  $-\lambda + \phi^\eta_0(s-c_1)$ touches $\rho_*$ from below strictly at $c_1$ for any $\eta\in[0,\eta_0)$.
		\end{claim}
		
	Let $\underline{p}_\pm$ be given in \eqref{eq:crit_barp2r} and \eqref{eq:crit_barp2'r}. In the case that $0\leq \underline{p}_+<+\infty$,  we have
		$$
		H^-(c_1+,\,p_+) < \lambda \leq H(c_1+,\,p_+ + \underline{p}_+),  
		$$
		which implies $p_+ +\underline{p}_+ \geq {\rm argmin}\,H(c_1+,\cdot)$. 
		Together with 
		$$
		A <\lambda, \qquad H(c_1+,\mu_+) = A
		$$ with  ${\mu_+}$ being the larger root of $p\mapsto H(c_1+,p)- A$, we deduce that $p_+ + \underline{p}_+> {\mu_+}$. Note that $\mu^\eta_+\to \mu_+$ as $\eta\to0^+$. Therefore, $p_+ + \underline{p}_+> \mu^\eta_+$ for any sufficiently small $\eta\ge0$. This yields $\rho_*(s) > -\lambda + \phi^\eta_0(s-c_1)$ in a right neighborhood of $c_1$ (which depends on $\eta$). Moreover, this last statement is obviously valid if $\underline{p}_+=+\infty$. 
		
		In the case that $\underline{p}_-<+\infty$, we could argue similarly to get for any sufficiently small $\eta\ge0$, $\rho_*(s) > -\lambda + \phi^\eta_0(s-c_1)$ in a left neighborhood of $c_1$, which is clearly true in the case that $\underline{p}_-=+\infty$.
			As a consequence, the Claim \ref{claim:6.8} is proved.
		
		Now for fixed $\eta\in(0,\min\{\eta_0,\frac{\lambda-A}{4}\})$, there exists $r\in(0,1)$, such that $-\lambda+\phi^\eta_0(s-c_1)$ touching $\rho_*$ from below strictly at $c_1$ in $(c_1-r,c_1+r)$. Letting $Q_r$ be given in \eqref{def:Qr},
		we get
		$$
		w_*(t,x)=t \rho_*(x/t) \geq \varphi^{0,\eta}(t,x):=\frac{A-\lambda}{4}(t-1)^2 
		+ t\left(-\lambda + \phi^\eta_0\left( \frac{x-c_1t}{t}\right)\right) 
		\quad \text{ in } Q_r
		$$
		with equality holds iff $(t,x) = (1,c_1)$.  Then there exists $\delta(\eta)>0$ such that
		$$w_*(t,x)\ge \varphi^{0,\eta}(t,x)+\delta\, \text{ on }\, \partial Q_r.$$
		Define $$\varphi^{\ep,\eta}(t,x) = \frac{A-\lambda}{4}(t-1)^2 - t\lambda - \ep \log \left[e^{-\frac{c_1}{2}\cdot \frac{x-c_1 t}{\ep}} \Phi^\eta(\frac{x-c_1t}{\ep})\right], \text{ for } (t,x)\in Q_r.$$ Clearly,  $\varphi^{\ep,\eta} \to \varphi^{0,\eta}$ in $C_{loc}$ since
		$$
		-\ep \log\left[ e^{-\frac{c_1}{2}\frac{x-c_1t}{\ep}} \Phi^\eta(\frac{x-c_1t}{\ep})\right] \to \phi^\eta_0(x-c_1t) = t\phi^\eta_0(\frac{x-c_1t}{t})\quad \text{ locally uniformly.}
		$$
		Therefore, 	$$w^\epsilon(t,x)\ge \varphi^{\epsilon,\eta}(t,x)+\delta/2\, \text{ on }\, \partial Q_r.$$
		for sufficiently small $\epsilon$. Now we verify $\varphi^{\epsilon,\eta}+\delta/2$
		  is a subsolution of \eqref{eq:wepr}.  Indeed, 	\begin{align*}
		  	&\quad     \varphi^{\ep,\eta}_t - \ep \varphi^{\ep,\eta}_{xx} + |\varphi^{\ep,\eta}_x|^2  + g\left(\frac{x-c_1t}{\ep}\right) - e^{-\frac{(2\varphi^{\ep,\eta}+\delta)}{2\ep}}\\
		  	&\le\frac{A-\lambda}{2}(t-1)-\lambda + \Lambda_1^{\eta}-\frac{c_1^2}{4}+ \|g-g_\eta\|_\infty,\\
		  	 & \le\frac{A-\lambda}{2}(t-1)-\lambda +A +\Lambda_1^{\eta}-\Lambda_1+ \|g-g_\eta\|_\infty\\
		  	  & \le\frac{A-\lambda}{2}(t-1)-\lambda +A +2{\eta}\\
		  	   	  & \le\frac{-\lambda +A}{2} +2{\eta}<0.
		  \end{align*}
	It then follows from the maximum principle  that, for all small $\ep$,
		\begin{equation}\label{eq:1.34r}
			w^\ep(t,x) -\delta/2 \geq \varphi^{\ep,\eta}(t,x) = \frac{A-\lambda}{4}(t-1)^2 - t\lambda + \phi^\eta_0(x-c_1t) + o(1) \quad \text{ for }(t,x) \in Q_r.
		\end{equation}
		Choose $(t^\ep,x^\ep) \to (1,c_1)$ such that $w^\ep(t^\ep,x^\ep) \to w_*(1,c_1)=\rho_*(c_1) = -\lambda$. Evaluating \eqref{eq:1.34r} at $(t^\ep,x^\ep)$ and then letting $\ep \to 0$, we again deduce that $-\lambda - \delta/2 \geq \varphi^{0,\eta}(1,c_1) = -\lambda$, which is a contradiction. This concludes the proof.
	\end{proof}
	
	\begin{proof}[Proof of Proposition \ref{prop:5.3}]
		It is a direct consequence of Lemmas \ref{lem:flsub} and \ref{lem:flsupp}.
	\end{proof}
	
	\subsection{Equivalence between Ishii solution and FL solution with $A=A_0$}\label{subsec:5.2}

	This section is a special case of \cite[Section 7]{Imbert2017flux} with the general Hamilton $H(s,p)$ being discontinuous at $c_1$.

	\begin{proposition}\label{prop:equi}
		Let $H(s,p):[(0,\infty)\setminus\{c_1\}]\times \mathbb{R}$ be convex in $p$ and such that $ H(c_1\pm,p)$ are well-defined and  coercive, and ${\rm argmin}\,H(c_1+,\cdot) = {\rm argmin}\,H(c_1-,\cdot)$.

  Define 
		$$
		\tilde{H}(s,p) = H(s,p)\quad \text{ if }s \neq c_1,\quad \text{ and }\quad \tilde H(c_1,p):= \max\{H(c_1-,p), H(c_1+,p)\}.
		$$
		Then for any given nonnegative function $\rho$, it is
  a FL-supersolution (resp. FL-subsolution) to
		\begin{equation}\label{eq:fl0}
			\begin{cases}
				\min\{\rho,\rho + H(s,\rho')\}=0 \qquad \text{ in }(0,\infty)\setminus\{c_1\},\\
				\min\{\rho(c_1),\,\rho(c_1) + \max\{A_0, H^-(c_1+,\rho'(c_1+)), H^+(c_1-,\rho'(c_1-))\} =0,&
			\end{cases}
		\end{equation}
		with $A_0:= \max\{\min\limits_\mathbb{R} H(c_1+,\cdot), \min\limits_\mathbb{R} H(c_1-,\cdot)\}$, if and only if it is 
  a viscosity supersolution (resp. subsolution) in the sense of Ishii to
		\begin{equation}\label{eq:ishii0}
			\min\{\rho,\rho + \tilde H(s,\rho')\} = 0 \quad \text{ in }(0,\infty),
		\end{equation}
where the definition of viscosity sub/supersolutions of \eqref{eq:ishii0} in sense of Ishii is given in Definition \ref{def:1.2}.
	\end{proposition}

	Note also that FL-supersolution (resp. subsolution) with $A \leq A_0$ is equivalent to the case $A = A_0$.
{For the particular Hamiltonians satisfying 
	\begin{equation}
		{H}(c_1-,p) = -c_1p + p^2 + g(-\infty) \quad \text{ and }\quad {H}(c_1+,p) = -c_1p + p^2 + g(+\infty).
	\end{equation} 
that we consider in this paper, one has 
 $A_0=\max\{g(-\infty), g(+\infty)\} - c_1^2/4$.}
	
	\begin{proof}[Proof of Proposition \ref{prop:equi}]
		Denote  $\tilde p$ be the common value of ${\rm argmin}\, H(c_1\pm, \cdot)$.
		First, we show sufficiency, i.e. super/subsolution in sense of Ishii implies FL-super/subsolution. 
		
		Let $\rho$ be a viscosity supersolution of \eqref{eq:ishii0} in the sense of Ishii. Then $\rho \geq 0$ for all $s$.  
		
		Let $\psi(s)$ be a vertex test function touching $\rho$  from below at $s=c_1$. Denote 
		\begin{equation}\label{eq:not12}
			\lambda = -\rho(c_1), \quad p_+ = \psi'(c_1+),\quad p_- = -\psi'(c_1-).
		\end{equation}
		We need to show ${F_{A_0}}(p+,-p_-) \geq \lambda$, where
		$$
		{F_{A_0}}(p+,-p_-) = \max\{{A_0}, H^-(c_1+,p_+), H^+(c_1-,-p_-)\}.
		$$
		
		By the critical slope results (Lemma \ref{cor:supersolr}) and Remark \ref{rem: supsoll}), there exist $\underline{p}_\pm \geq 0$ such that
		\begin{equation}\label{eq:not13}
			H(c_1+,p_+ + \underline{p}_+) \geq \lambda \quad \text{ and }\quad H (c_1-,-p_- - \underline{p}_-) \geq \lambda. 
		\end{equation}
		(These $\underline{p}_\pm$ are given in \eqref{eq:crit_barp2r}-\eqref{eq:crit_barp2'r}. If any of them is infinite, then simply take a large enough positive number {satisfying} \eqref{eq:not13}.)

		If $A_0 \geq \lambda$, then  $F_{{A_0}}(p_+,-p_-)\geq A_0 \geq \lambda$, and we are done.
		
		If $p_+ + \underline{p}_+ \leq \tilde p$ ~(resp. $-p_- - \underline{p}_- \geq \tilde p$), then we are done, since
		$$
		H^-(c_1+,p_+) \geq H^-(c_1+,p_+ + \underline{p}_+) = H(c_1+,p_+ + \underline{p}_+) \geq \lambda \quad ({\rm resp. }~ H^+(c_1-,-p_-) \geq \lambda).
		$$
		Henceforth, we assume
		\begin{equation}
	\begin{cases}
	A_0 < \lambda \leq \min\{ H (c_1+, p_+ + \underline{p}_+),H (c_1-,- p_- - \underline{p}_-)\},\\
	-\infty<	-p_- - \underline{p}_- < \tilde p<p_+ + \underline{p}_+<+\infty.
\end{cases}
\label{eq:5.20}
		\end{equation}
		By the definition of the critical slopes,
		the second line in \eqref{eq:5.20} means that $\rho - \tilde\psi(s)$ has a strict local minimum at $s=c_1$, where $\tilde\psi \in C^1$ is the special smooth test function
		$$
		\tilde\psi(s) 
		= \psi(c_1) + \tilde{p} (s-c_1).
		$$
		By  solution property in the Ishii sense (see Definition \ref{def:1.2}(b)), we have
		$$
		A_0 = \max\{ \min H(c_1\pm,\cdot)\}=\max\{H(c_1\pm, \tilde p)\}  \geq \lambda.
		$$
		This is a contradiction with \eqref{eq:5.20},
   and shows that $\rho$ is FL-supersolution with flux limiter $A_0$. 
		
		Next, we show subsolution in sense of Ishii implies FL-subsolution.
		
		Let $\rho(c_1)>0$ and let $\psi \in C^1_{pw}$ be a vertex test function touching $\rho$  from above at $s=c_1$. 
		We need to show
		\begin{equation}\label{eq:im2.20r}
			\max\{A_0, H^-(c_1+,~p_+), H^+(c_1-,\,-p_-)\} \leq \lambda, 
		\end{equation}
		where $\lambda,\,p_+,\, p_-$ are as in \eqref{eq:not12}, and $H^-(s,\cdot)$ and $H^+(s,\cdot)$ denotes the decreasing and increasing part of $H(s,\cdot)$, respectively.
		
		By critical slope results in 
	Lemma \ref{cor:subsolr} and Remark \ref{rem:subsoll} ($\rho$ enjoys weak continuity property thanks to Lemma \ref{cor:im1rho'}), there exist finite real numbers $\bar p_\pm \leq 0$ (given by \eqref{eq:crit_barp1r}-\eqref{eq:crit_barp1'r})  such that
		\begin{equation}\label{eq:im2.21r}
			H (c_1+, p_+ + \bar p_+) \le \lambda \quad \text{ and }\quad H(c_1-,-p_- - \bar p_-) \leq \lambda. 
		\end{equation}
		In particular, we deduce that
		\begin{equation}\label{eq:im2.22r}
			A_0 \leq \lambda. 
		\end{equation}
		Moreover, \eqref{eq:im2.21r} also implies that at $(t_0,0)$,
		\begin{align}
			H^-(c_1+,p_+) &\leq H^-(c_1+,p_+ + \bar{p}_+) \quad \text{ since }H^-(c_1+,\cdot)\text{ is nonincreasing}, \notag\\
			&\leq  H(c_1+, p_+ + \bar{p}_+) \leq  \lambda. \label{eq:im2.22br}
		\end{align}
		Similarly, we also obtain $H^+(c_1-,-p_-) \leq \lambda$.
		Combining with  \eqref{eq:im2.22r} and \eqref{eq:im2.22br}, we obtain \eqref{eq:im2.20r}. This proves that $w$ is a FL-subsolution with $A=A_0$.
		
		Next, we show the converse statement, i.e. FL-super/subsolution implies super/subsolution in sense of Ishii.
		
		Let $\rho$ be a FL-supersolution of \eqref{eq:fl0}, and let $\psi$ be a $C^1$ test function touching $\rho$ from below at $s=c_1$. 
		Then we have $\rho \geq 0$ for all $s$, and 
		$$
		\max\{H(c_1+, \psi'(c_1)), H(c_1-, \psi'(c_1))\} \geq 
		\max\{H^-(c_1+,\psi'(c_1), H^+(c_1-, \psi'(c_1)\} \geq -\psi(c_1).
		$$
		This proves that $\rho$ is viscosity supersolution of \eqref{eq:ishii0} in the Ishii sense.
		
		Finally, let $\rho$ be a FL-subsolution of \eqref{eq:fl0} with $A=A_0$, and $\psi$ be a $C^1$ test function touching $\rho$ from below at $s=c_1$. Then we have
		$$
	\tilde{H}(c_1,\psi'(c_1))=\max\{H^-(c_1+,\psi'(c_1), H^+(c_1-, \psi'(c_1)\} \leq -\rho(c_1).
		$$
		Now, since $\tilde p={\rm argmin}\, H(c_1+,\cdot)= {\rm argmin}\, H(c_1-,\cdot)$, we either have $\psi'(c_1) \geq \tilde p$ or $\psi'(c_1) < \tilde p$. In the former case, we have $H(c_1-,\psi'(c_1)) = H^+(c_1-, \psi'(c_1)) \leq -\rho(c_1).$ In the latter case, we have $H(c_1+,\psi'(c_1)) = H^-(c_1+, \psi'(c_1) \leq -\rho(c_1).$ This implies that 
		$$
		\tilde{H}_*(c_1,\psi'(c_1))=\min\{H(c_1-,\psi'(c_1)), H(c_1+,\psi'(c_1))\} \leq -\rho(c_1). 
		$$
		i.e. $\rho$ is a viscosity subsolution of \eqref{eq:ishii0} in the sense of Ishii.
	\end{proof}
	
	{Next, we specialize to the class of Hamiltonian defined in \eqref{eq:ishii3}, and prove the first part of Theorem \ref{thm:2.13}.}
	
	\begin{proof}[Proof of Theorem \ref{thm:2.13}, first part.] We establish Theorem \ref{thm:2.13} in case $\Lambda_1 \leq \max\{g(\pm\infty)\}$.
Then $\Lambda_1 = \max\{g(\pm\infty)\}$ (thanks to Proposition \ref{prop:2.1}(a)). By Proposition \ref{prop:5.3}, $\rho^*$ (resp. $\rho_*$) is a FL-subsolution (resp. FL-supersolution) of \eqref{eq:fl} with 
		$$
		A = A_0:= \max\{g(\pm\infty)\} - \frac{c_1^2}{4}.
		$$
		Thanks to Proposition \ref{prop:equi}, $\rho^*$ and $\rho_*$ are viscosity sub- and supersolution of \eqref{eq:ishii1} in the Ishii sense.
		Moreover, it follows from Lemma \ref{lem:5.2} that
		\begin{equation}
			\rho^*(0) =0 \leq \rho_*(0), \quad \text{ and }\quad \rho_*(s)/s \to +\infty \quad \text{ as }s \to +\infty. 
		\end{equation}
		Hence, we may apply the comparison principle \cite[Proposition 2.11]{Lam2022asymptotic} for viscosity solutions in the Ishii sense to deduce that
		$$
		\rho^*(s) \leq \rho_*(s) \quad \text{ for all }s \geq 0.
		$$
		Since also $\rho^* \geq \rho_*$ by construction (see \eqref{eq:wstars}), we conclude that $\rho^* \equiv \rho_*$. We define $\hat\rho_{A_0}$ to be the the common value. 
		This proves the existence and uniqueness of $\hat\rho_{A_0}$ stated in Proposition \ref{prop:2.10}. (Note that this also settles the case $A \leq A_0$, as they yield the same equation \eqref{eq:fl}.)
		
		Furthermore, $w^\ep(t,x) \to t \hat\rho_{A_0}(x/t)$ in $C_{loc}((0,\infty)\times (0,\infty))$. Let $s_{base}=\sup\{s \geq 0:~ \hat\rho_{A_0}(s)=0\}$, then $\hat\rho_{A_0}(s) >0$ for $s > s_{base}$. This gives 
		$$
		u^\ep(t,x) = e^{-\frac{w^\ep(t,x)}{\ep}} \to 0 \quad \text{ locally uniformly for }\{(t,x):~t>0,~x>s_{base} t\},
		$$
		i.e. $\overline{c}^* \leq s_{base}$, where $\overline{c}^*$ is the maximal spreading speed given in \eqref{eq:speeds}.
		
		Next, we observe that $\hat\rho_{A_0}$ is monotone increasing (Lemma \ref{cor:im1rho'}(a)), so that $\hat\rho_{A_0}(s) = 0$ for $s \in [0,s_{base}]$ and hence
		$$
		w^\ep(t,x) \to 0 \quad \text{ in }C_{loc}(\{(t,x):~t>0,~0\leq x <s_{base} t\}).
		$$
		It then follows as in \cite[Lemma 3.1]{Liu2021asymptotic} that 
		$$
		\liminf_{\ep \to 0} \inf_K u^\ep(t,x) \geq \inf g>0 
		$$
		for each compact subset $K \subset \{(t,x):~t>0,~0\leq x <s_{base} t\}$. For each $\eta>0$, we may take $K = \{(1,s):~ \inf g /2\leq s \leq  s_{base}-\eta\}$, and deduce
		$$
		\liminf_{t\to\infty} \inf_{\frac{\inf g}{2}t \leq x \leq (s_{base}-\eta)t} u(t,x) = \liminf_{\ep \to 0} \inf_K u^\ep(t,x) >0 \quad \text{ for any }\eta>0.
		$$
		Since $\eta>0$ is arbitrary, this implies $\underline{c}_* \geq s_{base}$. Combining with $\overline{c}_* \leq s_{base}$, we obtain $\overline{c}_*=\underline{c}_* = s_{base}.$ This concludes the proof of Theorem \ref{thm:2.13} in the case $\Lambda_1 = \max\{g(\pm\infty)\}.$
	\end{proof}

	Having verified that $w^*$ and $w_*$ are FL-subsolution and FL-supersolution of \eqref{eq:fl}, one may apply the arguments in \cite{Imbert2017flux} to obtain a comparison principle. Here, however, we will
 follow the arguments due to Lions and Souganidis \cite{Lions2017well} to show that they are in fact viscosity sub- and supersolutions of certain Kirchhoff junction conditions, and establish the more general comparison principle  (see Appendix \ref{sec:comp}).

 \begin{remark}
     The concept of FL-sub/supersolutions was originally introduced in \cite{Imbert2017flux,Imbert2017quasi}, in which the authors established the comparison principle based on the construction of certain ``vertex test functions".
 \end{remark}

	\subsection{Verification of Kirchhoff junction conditions}\label{sec:5.3}

	Let $B \in \mathbb{R}$ be given. 
	We consider the Hamilton-Jacobi equation with Kirchhoff junction condition:
	\begin{equation}\label{eq:kc1}
		\begin{cases}
			\min\{\rho, \rho+ H(s,\rho')\} = 0 \quad \text{ for }s \neq c_1,\\
			\min\{\rho(c_1), \min\{\rho(c_1) + H(c_1\pm,\rho'(c_1\pm)), \rho'(c_1-) - \rho'(c_1+) - B\}\} \leq 0,\\
			\min\{\rho(c_1), \max\{\rho(c_1) + H(c_1\pm,\rho'(c_1\pm)), \rho'(c_1-) - \rho'(c_1+) - B\} \}\geq 0.
		\end{cases}
	\end{equation}
	The definition of viscosity solution to the above problem also involves the use of piecewise $C^1$ test functions.
	\begin{definition}
		\begin{itemize}
			\item[{\rm(a)}] We say that $\underline\rho$ is a viscosity subsolution of \eqref{eq:kc1} provided (i) $\underline\rho$ is upper semicontinuous, and (ii) if $\underline\rho - \psi$ has a local maximum point at some $s_0$ such that $\psi \in C^1_{pw}$ and $\underline\rho(s_0)>0$, then 
			$$
			\underline\rho(s_0)+ H(s_0,\psi'(s_0)) \leq 0 \quad \text{ in case } s_0 \neq c_1 ,
			$$
			$$
			\min\{\underline\rho(c_1) + H(c_1\pm,\psi'(c_1\pm)), \psi'(c_1-) - \psi'(c_1+) - B\}\leq 0 \quad \text{ in case }s_0 = c_1.
			$$
			
			\item[{\rm(b)}]  We say that $\overline\rho$ is a viscosity supersolution of \eqref{eq:kc1} provided (i) $\overline\rho$ is lower semicontinuous, (ii) $\overline\rho \geq 0$ for all $s$, and (iii) if $\overline\rho - \psi$ has a local minimum point at some $s_0$ such that $\psi \in C^1_{pw}$ then
			$$
			\overline\rho(s_0)+ H(s_0,\psi'(s_0)) \geq 0 \quad \text{ in case } s_0 \neq c_1 ,
			$$
			$$
			\max\{\overline\rho(c_1) + H(c_1\pm,\psi'(c_1\pm)), \psi'(c_1-) - \psi'(c_1+) - B\}\geq 0 \quad \text{ in case }s_0 = c_1.
			$$
			
			\item[{\rm(c)}] We say that $\rho$ is a viscosity solution of \eqref{eq:kc1} provided it is a viscosity subsolution and supersolution of \eqref{eq:kc1}.
		\end{itemize}
	\end{definition}
	
	Next,  for each flux-limiter $A\ge A_0$, where $A_0 = \max\{g(\pm\infty)\} - \frac{c_1^2}{4}$, we associate a Kirchhoff junction condition parameter $B$ as follows:
	$$
	B = -\mu_+ - \mu_-,
	$$
	where $\mu_+, \mu_-$ are uniquely determined in terms of $A$ by
	\begin{equation}\label{eq:mupm11}
		H^+(c_1+,\mu_+) = A \quad \text{ and }\quad H^-(c_1-,-\mu_-) = A.
	\end{equation}
	By recalling the definition of $H^+(c_1\pm,p)$ and $H^-(c_1\pm,p)$ in \eqref{eq:Hplus} and \eqref{eq:Hminus}, we deduce 
	$$
	\mu_+ = \frac{1}{2}(c_1 + \sqrt{c_1^2 + 4(A - g(+\infty))}),\quad \mu_- = \frac{1}{2}(-c_1 + \sqrt{c_1^2 + 4(A - g(-\infty))})
	$$
	\begin{lemma} \label{lem:5.10}
		Let $A \ge A_0:= \max\{g(\pm\infty)\} - \frac{c_1^2}{4}$, and define $\mu_\pm$ in terms of $A$ as in \eqref{eq:mupm11}.
		If $\underline\rho$ is a FL-subsolution to \eqref{eq:fl}, 
  then it is a viscosity subsolution of the problem \eqref{eq:kc1} with Kirchhoff junction condition with parameter $B = -\mu_+ - \mu_-$.
	\end{lemma}
	\begin{proof}
		{It} remains to show that $\underline\rho$ is a subsolution to the second equation of \eqref{eq:kc1}. For this purpose, let $\psi \in C^1_{pw}$ and suppose $\underline\rho - \psi$ has a strict global maximum point at $c_1$, and that $\psi(c_1) =\underline\rho(c_1) >0$. Denote 
		$$
		\lambda = -\underline\rho(c_1) = -\psi(c_1),\qquad p_+ = \psi'(c_1+),\qquad p_- = -\psi'(c_1-).
		$$
		Suppose 
		\begin{equation}\label{eq:kirchhoffr1}
			H(c_1+,p_+)> \lambda \quad \text{ and }\quad H(c_1-,p_-)> \lambda,
		\end{equation}
		we need to show that
		\begin{equation}\label{eq:kirchhoffr}
			-p_- - p_+ + \mu_+ + \mu_- \leq 0.
		\end{equation}
		Thanks to the critical slope lemma (Lemma \ref{cor:subsolr}), $H(c_1+, p_+ + \bar p_+) \leq \lambda$ for some $\bar p_+ < 0$, it follows that $p_+ \ge \text{argmin}~H(c_1+,\cdot)$. Similarly, we have $-p_- \le \text{argmin}~H(c_1-,\cdot)$.

By the definition of FL-subsolution (see Defintion \ref{def:rhoflsol}), it follows that	$\lambda \geq A$. This, together with the fact that ${\mu_+}$ (resp. $-{\mu_-}$) is the larger (resp. smaller) root of $p \mapsto H(c_1\pm,p) - A$, implies
		\begin{equation}
			p_+ \ge {\mu_+} \quad \text{ and }\quad -p_- \le -{\mu_-}.
		\end{equation}
		Therefore, we obtain $-p_+ - p_- + {\mu_+} + {\mu_-} \leq 0$.
	\end{proof}
	
	\begin{lemma}\label{lem:5.11}
		Let $A\ge A_0$. 
		If $\overline\rho$ is a FL-supersolution of \eqref{eq:fl}, then it is a viscosity supersolution of the problem \eqref{eq:kc1} with Kirchhoff condition with parameter $B = -\mu_+ - \mu_-$.
	\end{lemma}
	\begin{proof} It remains to verify the second condition of \eqref{eq:kc1}. Suppose  there is a test function 
		$\psi \in C^1_{pw}$ that touches $\overline\rho$ from {below} only at $s=c_1$.
		
		Denote 
		$$
		\lambda = - \overline\rho(c_1), \qquad p_+ = \psi'(c_1+), \qquad  p_- = -\psi'(c_1-).
		$$
		Suppose 
		\begin{equation}\label{eq:kirchhoff3r}
			H(c_1+,\,p_+) < \lambda \quad \text{ and }\quad H(c_1-,\,-p_-) < \lambda,
		\end{equation}
		we need to show 
		\begin{equation}\label{eq:kirchhoff4r}
			-p_- - p_+ + {\mu_+} + {\mu_-} \geq 0.
		\end{equation}
		Since $\overline\rho$ is a FL-supersolution, we have $A \geq \lambda$. Upon considering
		\eqref{eq:kirchhoff3r}, and also   ${\mu_+}$ (resp. $-{\mu_-}$) being the larger (resp. smaller) root of $H^+(c_1+,p) = A$ (resp. $H^-(c_1-, p)=A$), we deduce\footnote{Note that we do not need $\bar{p}_+<+\infty$ here, comparing with the proof of the previous verification for junction subsolution.  This asymmetry in the arguments of super and subsolutions is due to the fact that $H$ is quasiconvex and coercive.}
		\begin{equation}
			p_+ < {\mu_+} \quad \text{ and }\quad -p_- > -{\mu_-}.
		\end{equation}
		This implies \eqref{eq:kirchhoff4r}.
	\end{proof}
	\begin{corollary}
		$\rho^*$ (resp. $\rho_*$) are viscosity subsolution (resp. supersolution) of \eqref{eq:kc1} with $B=-\mu_+ - \mu_-$ where $\mu_\pm$ are associated with $A = \Lambda_1 - \frac{c_1^2}{4}$ via \eqref{eq:mupm}.
	\end{corollary}
	\begin{proof}
		Fix $A=\Lambda_1 - \frac{c^2_1}{4} > \max\{g(\pm\infty)\} - \frac{c^2_1}{4}.$ Define $\mu_+,\mu_-$ by \eqref{eq:mupm}. 
		
		By Proposition \ref{prop:5.3}, $\rho^*$ and $\rho_*$ are FL-subsolution and supersolution of \eqref{eq:fl} with $A = \Lambda_1 - \frac{c_1^2}{4}$, respectively. By Lemmas \ref{lem:5.10} and \ref{lem:5.11}, they are viscosity sub- and supersolutions of \eqref{eq:kc1} with $B=-\mu_+-\mu_-$.
	\end{proof}
	
	\subsection{Proof of main results}\label{subsec:5.4}
	
	\begin{proof}[Proof of Proposition \ref{prop:2.10}]
		Recall that $A_0 = \max\{g(\pm\infty)\} - \frac{c_1^2}{4}$. 
		Let $\underline\rho$ and $\overline\rho$ be a pair of FL-subsolution and FL-supersolution of \eqref{eq:fl} for some $A \geq A_0$, such that \eqref{eq:2.8rr} holds.
		
		If $A\le A_0$ holds, then by Proposition  \ref{prop:equi}, $\underline\rho$, $\overline\rho$ is a pair of viscosity sub- and supersolution of \eqref{eq:ishii1} in the sense of Ishii. The comparison principle follows from \cite[Proposition 2.11]{Lam2022asymptotic}.

		Henceforth, we assume $A > A_0$. Then, it follows from Lemmas \ref{lem:5.10} and \ref{lem:5.11} that $\underline\rho$, $\overline\rho$ are a pair of sub- and supersolutions of \eqref{eq:kc1}. The comparison principle then follows from Theorem \ref{thm:comparison}. 
	\end{proof}
	
	\begin{proof}[Proof of Corollary \ref{prop:2.12} and Theorem \ref{thm:2.13}]
		For given function $g$, denote $A=\Lambda_1 - \frac{c_1^2}{4}$, where $\Lambda_1$ is given in \eqref{eq:lambda1}. By Proposition \ref{prop:2.1}(a), $A \geq A_0:= \max\{g(\pm\infty)\} - \frac{c_1^2}{4}$.

		For the case $A =A_0$, in view of  Proposition \ref{prop:equi}, the problem is equivalent to \eqref{eq:ishii1} and \eqref{eq:ishii2}. The existence and uniqueness of $\hat\rho_{A_0}$ follows from \cite[Proposition 1.7(b)]{Lam2022asymptotic}. The convergence $w^\ep(t,x) \to t\hat\rho(x/t)$ in $C_{loc}$ and determination of spreading speed is given in the first part of the proof of Theorem \ref{thm:2.13} (see Subsection \ref{subsec:5.2}).
		
		For the case $A > A_0$, it follows that $\rho^*$ and $\rho_*$ are viscosity sub- and supersolutions of \eqref{eq:kc1} and satisfies $\rho^*(0) = \rho_*(0) = 0$  and $\rho_*(s)/s \to +\infty$ as $s \to +\infty$ (see Lemma \ref{lem:5.2}). Hence, Theorem \ref{thm:comparison} implies $\rho^* \leq \rho_*$. Arguing similarly as in the proof of the first part Theorem \ref{thm:2.13} (in Subsection \ref{subsec:5.2}), we conclude that (i) \eqref{eq:fl}-\eqref{eq:ishii2} has a unique FL-solution $\hat\rho_A(s)$; (ii) $w^\ep(t,x) \to t\hat\rho_A(x/t)$ in $C_{loc}$; (iii) 
		the spreading speed is given by $c_* = \hat{s}_A = \sup\{s\geq 0:~ \hat\rho_A(s)=0\}$.
	\end{proof}

	\appendix
	
	\section{Proof of Proposition \ref{prop:2.1}}
	\label{sec:prop21}

	\begin{proof}[Proof of Proposition \ref{prop:2.1}]
		For (a), observe that if $\phi'' + g(y)\phi \leq \Lambda \phi$ in $\mathbb{R}$ for some positive $\phi \in C^2_{loc}(\mathbb{R})$ and $\Lambda < g(+\infty)$, then $\phi'' <0$ for $y \gg 1$. Since $\phi>0$ in $\mathbb{R}$, we deduce that $\phi'>0$ for $y \gg 1$, and hence $\phi'(+\infty)\in[0,+\infty)$  and $\phi(+\infty)\in(0,+\infty]$ both exist.  However, this means $\limsup_{y \to \infty} \phi''(y) \leq (\Lambda - g(+\infty)) \phi(+\infty) <0$, which contradicts {$\phi'(+\infty)\ge 0$}. This proves $\Lambda_1 \geq g(+\infty)$. Similarly, we can show that $\Lambda_1 \geq g(-\infty)$. Next, we apply \cite[Theorem 1.4]{Berestycki2014generalizations} to infer that \eqref{eq:pev0} has a positive solution if and only if $\Lambda \in [\Lambda_1,\infty).$ (Note that $\Lambda_1 = -\lambda_1(L,\mathbb{R})$ in the notation of \cite{Berestycki2014generalizations}.) This proves (a).
		
	For (b)(i), we first apply \cite[Proposition 1.11(ii)]{Berestycki2014generalizations} to deduce that if $\Lambda_1 > \max\{g(\pm\infty)\}$, then $\Lambda_1$ is a simple eigenvalue and $\Phi_1$ converges exponentially to zero as $|y| \to\infty$.  To establish \eqref{est-decay-rate}, it suffices to prove the estimation of accurate decay rate of $\Phi_1$ at $+\infty$. Since $\Lambda_1>g(+\infty)$, there exists $a_0>0$ such that $\Lambda_1>g(y)$ for any $y\ge a_0$. Now we define $$\overline{\lambda}(a):=\sup_{y\ge a} \sqrt{\Lambda_1-g(y)}, \quad \underline{\lambda}(a):=\inf_{y\ge a} \sqrt{\Lambda_1-g(y)}, \quad \forall y\ge a_0.$$ 
		Fix $a\ge a_0$. Then for any $\epsilon>0$ and $M>0$, it is easy to check that
		$\bar\Phi^M_\epsilon(y)= Me^{-\underline{\lambda}(a)y}+\epsilon$ is a supersolution of \eqref{eq:pev0} on $[a,+\infty)$, that is,
		$$(\bar \Phi^M_\epsilon)''+(g(y)-\Lambda_1)\bar \Phi^M_\epsilon\le 0 \text{ on } [a,+\infty).$$
		For any given small $\epsilon_0>0$, there exists $M_1>{e^{\underline{\lambda}(a)a}\Phi_1(a)}$ such that $\bar \Phi^{M_1}_{\epsilon_0}(y)> \Phi_1(y)$ on $[a,+\infty)$. Using the sliding argument or strong maximum principle, we infer that  $\bar \Phi^{M_1}_{\epsilon}(y)> \Phi_1(y)$ on $[a,+\infty)$ for any $\epsilon\in(0,\epsilon_0]$. Letting $\epsilon\to 0^+$, we have 
		$M_1e^{-\underline{\lambda}(a)y}\ge \Phi_1(y)$ on $[a,+\infty)$.  
		
		Similarly, for any $M>0$ and $\epsilon\in(0,e^{-a\bar\lambda(a)})$, $\underline{\Phi}^M_\epsilon(y):= \max\{0, M(e^{-\overline{\lambda}(a)y}-\epsilon) \}$ is a subsolution of \eqref{eq:pev0} in $(a,+\infty)$, that is,
		$$(\underline{\Phi}^M_\epsilon)''+(g(y)-\Lambda_1)\underline{\Phi}^M_\epsilon\ge 0 \text{ on } [a,+\infty).$$
		For any given small $\epsilon_0>0$, there exists $M_2>0$  such that 
		$\underline{\Phi}^{M_2}_{\epsilon_0}(y)< \Phi_1(y)$ on $[a,+\infty)$. By the sliding argument again, it follows that $\underline{\Phi}^{M_2}_{\epsilon}(y)< \Phi_1(y)$ on $[a,+\infty)$ for any $\epsilon\in(0,\epsilon_0]$. Letting $\epsilon\to 0^+$, we have 
		$M_2e^{-\bar{\lambda}(a)y}\le \Phi_1(y)$ on $[a,+\infty)$.  
		Consequently, there exist $\underline{C}>0, {\bar C}>0$(dependent on $a$), such that 
		$$\underline{C} e^{-\bar\lambda(a)y}\le\Phi_1(y)\le \bar C e^{-\underline{\lambda}(a)y}, \quad \forall y\ge0.$$
		Noting that $\overline{\lambda}(a)$ and $\underline{\lambda}(a)$ are continuous on $[a_0,+\infty)$ with  $\overline{\lambda}(+\infty)=\underline{\lambda}(+\infty)=\lambda_+$. Therefore, for any sufficiently small $\eta>0$, there exists $a>a_0$  such that $\lambda_+-\eta\le\underline{\lambda}(a)\le\overline{\lambda}(a)\le \lambda_++\eta $. This implies the first inequality in \eqref{est-decay-rate} is valid.

		For (b)(ii), suppose $(\tilde\Lambda, \tilde\Phi)$ is an eigenpair of \eqref{eq:pev0}. By the first assertion of the Proposition, $\tilde\Lambda \in [\Lambda_1,\infty)$. On the one hand, if $\tilde\Lambda = \Lambda_1$, 
		then we can immediately conclude by the fact that  $\Lambda_1$ is simple {(by \cite[Proposition 1.1(ii)]{Berestycki2014generalizations}). }

		On the other hand, if $\tilde\Lambda >\Lambda_1$, then one can prove that 
		$\tilde\Phi(y) \sim e^{-\sqrt{\tilde\Lambda-g(+\infty)} y+o(y)}$ as $y \to +\infty$. It then follows that $\tilde\Phi(y) /\Phi_1(y) \to 0$ as $y \to +\infty$. By repeating the argument, we also obtain $\tilde\Phi(y) /\Phi_1(y) \to 0$ as $y \to -\infty$.
		We can then touch $\Phi_1(y)$ from below with $k\tilde\Phi(y)$ to obtain, from the strong maximum principle, that $\Phi_1(x) = k \tilde\Phi(x)$ for some $k>0$. In this case $\tilde\Lambda = \Lambda_1$, a contradiction.  
	
For (c), fix an arbitrary $\eta>0$ we choose, for each $k \in\mathbb{N}$, a continuous function such that
 $$\begin{cases}
		g_\eta^k(x)=g(x), & |x|\ge k+1,\\
		g(x)\le g_\eta^k(x)\le g(x)+\eta, & l\le |x| \le k+1,\\
		g_\eta^k(x)=g(x)+\eta,& |x|\le k.
	\end{cases}$$
If $\Lambda_1(g^k_\eta) > \Lambda_1(g)$ for some $k$, we are done. Suppose to the contrary that $\Lambda_1(g^k_\eta) \equiv \Lambda_1(g)$ for all $k$, and let  $\Phi_k \in C^2_{loc}$ be a positive eigenfunction of $\Lambda_1(g^k_\eta)$.
By Harnack inequality, there exists a positive function $C(R)$ independent of $k$ such that
$$
\frac{1}{C(R)} \leq \frac{\Phi_k(x)}{\Phi_k(0)} \leq C(R) \quad \text{ for } k\in \mathbb{N},~ |x|\leq R.
$$
Normalizing by $\Phi_k(0) = 1$, we see that $\{\Phi_k\}$ is bounded in $C^2([-R,R])$ for each $R$. It follows that (up to a subsequence) $\Phi_k$ converges in $C^1_{loc}(\mathbb{R})$ to a positive eigenfunction $\tilde\Phi \in C^2_{loc}(\mathbb{R})$ satisfying
$$
\tilde\Phi'' + (g(x) + \eta) \tilde\Phi = \Lambda_1(g) \tilde\Phi \quad \text{ in }\mathbb{R}.
$$
By assertion (a), it follows that $\Lambda_1(g) \geq  \Lambda_1(g + \eta)$, which is impossible since $\Lambda_1(g + \eta) = \Lambda_1(g) + \eta$.
	\end{proof}
	\section{Proof of Lemmas \ref{cor:subsolr} and  \ref{cor:supersolr}}\label{corsubsup}
	Following the same procedure in  		
	\cite[Lemmas 2.9 and 2.10]{Imbert2017flux}, it suffices to prove Lemma \ref{cor:supersolr} without weak continuity condition and then show the finiteness of $\overline p_+$ in  Lemma \ref{cor:subsolr} with weak continuity condition.
		\begin{proof}[Proof of Lemma \ref{cor:supersolr}] 
		By the definition of $\underline{p}_+$, we see that $\underline{p}_+\ge0$. For any sufficiently small $\epsilon>0$, there exists $r_\epsilon\in(0,\epsilon)$ such that 
		$$\overline \rho(s)\ge \varphi(s)+(\underline p_+-\epsilon)(s-c_1) ~\text{ for } 0\le s-c_1\le r_\epsilon$$
		and there exists $s_\epsilon\in(c_1, c_1+\frac{r_\epsilon}{2})$ such that
		$$\overline{\rho}(s_\epsilon)<\varphi(s_\epsilon)+(\underline p_++\epsilon)(s_\epsilon-c_1).$$
		Now construct a smooth function $\Psi:\mathbb{R}\to[-1,0]$ such that 
		$$\Psi(s)=\begin{cases}
			0 &\text{ for } s\in (-1/2, 1/2),\\
			-1 & \text{ for } s\not \in (-1,+1)
		\end{cases}$$
		and define
		$$\Phi(s)=\varphi(s)+2\epsilon \Psi_{r_\epsilon}(s)+\begin{cases}
			(\underline{p}_++\epsilon)(s-c_1)& \text{ if } s\in U\cap(c_1,+\infty)\\
			0& \text{ if } s\in U\cap (-\infty,c_1]
		\end{cases}$$
		with $\Psi_{r_\epsilon}(s)=r_\epsilon \Psi((s-c_1)/r_\epsilon)$. It then follows that $\Phi(c_1)= \varphi(c_1)\textcolor{blue}{=}\overline{\rho}(c_1)$
		and
		$$\begin{cases}
			\Phi(c_1+r_\epsilon)=\varphi(c_1+r_\epsilon)+(\underline{p}_+-\epsilon)r_\epsilon\le \overline{\rho}(c_1+r_\epsilon),\\
			\Phi(s_\epsilon)=\varphi(s_\epsilon)+(\underline p_++\epsilon)(s_\epsilon-c_1)>\overline{\rho}(s_\epsilon).
		\end{cases}$$
		This implies that there exists a point $\bar s_\epsilon\in(c_1,c_1+r_\epsilon)$ such that $\overline\rho-\Phi$ attains a local minimum at $\bar s_\epsilon$. Therefore,  by the definition of viscosity supersolution and $H(\cdot,p)$ is convex in $p$, we obtain
		$$\Phi(\bar s_\epsilon)+\sup_{s\in(c_1,c_1+r_\epsilon)}H(s,\Phi'(\bar s_\epsilon))\ge\overline{\rho}(\bar s_\epsilon)+H^*(\bar s_\epsilon, \Phi'(\bar s_\epsilon))\ge0,$$
		which yields		
		$$\Phi(\bar s_\epsilon)+\sup_{s\in(c_1,c_1+r_\epsilon)}H(s,\varphi'(\bar s_\epsilon)+2\epsilon\Psi'_{r_\epsilon}(\bar s_\epsilon)+\underline{p}_++\epsilon)\ge0.$$
		Letting $\epsilon\to 0^+$, we reach
		$$\varphi(c_1)+H(c_1+,\varphi'(c_1)+\underline{p}_+)\ge0.$$
		Now the conclusion follows immediately from the fact that $\bar \rho (c_1)\textcolor{blue}{=}\varphi(c_1).$		
	\end{proof}
\begin{proof}[Proof of Lemma \ref{cor:subsolr}]
	We only show that $\bar p_+>-\infty$. Without loss of generality, we might assume that $\varphi(c_1)=\underline\rho(c_1)>0$. Suppose by contradiction that $\bar p_+=-\infty$, then there exists $p_n\to-\infty$ and $r_n>0$ such that $$\varphi(s)+p_n(s-c_1)\ge \underline{\rho}(s)\, \text{ for }\, 0\le s-c_1<r_n.$$
	Modifying $\varphi$ if necessary (e.g.,  $\varphi+(s-c_1)^2$), we could further assume that
	\begin{equation}\label{eq: pn}
	\varphi(s)+p_n(s-c_1)> \underline{\rho}(s)\, \text{ for }\, 0< s-c_1\le r_n. 
	\end{equation}
For fixed $n$, since $\underline\rho$ satisfies the weak continuity condition \eqref{eq:weakcontrho}, it then follows that there exists $s_{m}\in(c_1,c_1+r_n)$ such that $s_m\to c_1$ and $\underline\rho(s_m)\to \underline{\rho}(c_1)$ as $m\to+\infty$.
	Define 
	 $$\Psi_m(s)=\varphi(s)+p_n(s-c_1)+\frac{(s_m-c_1)^2}{s-c_1}, \quad s>c_1.$$
	 For each $m\in \mathbb{N}$, there exists  $\hat s_m\in(c_1, c_1+r_n]$ such that $\Psi_m-\underline{\rho}$ attains the minimum at $\hat s_m$. Then
	 \begin{eqnarray}\label{eq: psim}
	 o(1)= \Psi_m(s_m)-\underline{\rho}(s_m)\ge\Psi_m(\hat s_m)-\underline{\rho}(\hat s_m)\ge
	 \varphi(\hat s_m)+p_n(\hat s_m-c_1)-\underline{\rho}(\hat s_m).
	 \end{eqnarray}
	 Suppose $\hat s_m\to s_0\not =c_1$ (up to a subsequence) as $m\to\infty$.  Then letting $m\to+\infty$ in \eqref{eq: psim}, we have
	 $$0\ge \varphi(s_0)+p_n(s_0-c_1)-\liminf_{m\to\infty}\underline\rho(\hat s_m)>0,$$ 
	 where the second inequality follows from \eqref{eq: pn} and the upper semicontinuity of $\underline{\rho}$. This is a contradiction. Therefore, we conclude that $\hat s_m\to c_1$ and $\underline\rho(\hat s_m)\to\underline{\rho}(c_1)>0$ as $m\to\infty$.  Now we might assume $\underline{\rho}(\hat s_m)>0$ for each $m\in\mathbb{N}$, by the definition of viscosity subsolution, we obtain 
	 $$\underline{\rho}(\hat s_m)+H_*\left(\hat s_m, \varphi'(\hat s_m)+p_n-\frac{(s_m-c_1)^2}{ (\hat s_m-c_1)^2}\right)\le0.$$
	 Note that $\inf\limits_{s\in(c_1,\hat s_m]} H(s,p)\le H_*(\hat s_m,p)$. Then we pass to the limit as $m\to+\infty$ in the inequality above and get
	 $$\underline{\rho}(c_1)+ H(c_1+, \varphi'(c_1+)+p^0_n)\le 0,$$
where $p^0_n=p_n-\limsup\limits_{m\to\infty}\frac{(s_m-c_1)^2}{ (\hat s_m-c_1)^2}\in[-\infty, 0]$.	 It then follows from $\liminf_{p\to-\infty}H(c_1+,p)\ge0$ that $p^0_n>-\infty$ and $p^0_n$ is bounded from below by a constant $C$ which only depends on $H(c_1+,p)$, $\underline{\rho}(c_1)$ and $\varphi'(c_1+)$. But this also implies $p_n\ge  C$, which leads to a contradiction. This completes the proof of the finiteness of $\overline{p}_+$.
 \end{proof}

	\section{Comparison Principle for problem with Kirchhoff condition}\label{sec:comp}

{The comparison principle for FL-solutions was first proved by Imbert and Monneau \cite{Imbert2017flux,Imbert2017quasi}. Subsequently, Lions and Souganidis gave an alternative proof by connecting it to the Kirchoff junction condition \cite{LionsSouganidis2016,Lions2017well}. We combine the arguments of the latter and of \cite{Lam2022asymptotic} to prove a comparison result that allows for solutions that grows superlinearly.}

	Let $\mathcal{P}=\{c_i\}_{i=1}^n$ for some $0<c_1<...<c_n$ and $B_i\in \mathbb{R}$ for all $i$.  We establish a comparison principle for viscosity sub- and supersolutions of the Hamilton-Jacobi equation 
	\begin{equation}\label{eq:hjeapp}
		\min\{\rho, \rho+ H(s,\rho')\} = 0 \quad \text{ in }(0,\infty)\setminus \{c_i\}
	\end{equation}
	with the following Kirchhoff junction condition at $c_i$
	\begin{equation}\label{eq:kirapp}
		\rho'(c_i-) - \rho'(c_i+) - B_i = 0 \quad {\text{ for } i=1,...,n},
	\end{equation}
{and boundary conditions $\rho(0)=0$ and $\rho (s)/s \to +\infty$ as $s\to +\infty$. 
	Here,} we assume the following for the Hamiltonian function $H(s,p)$. 
	\begin{description}
		\item[{\rm (A1)}] For any given $L>0$ and each $s_0 \in (\tfrac1L,L) \setminus \{c_i\}_{i=1}^n$, there exists $\delta_0=\delta_0(L)$ and $h_0 \in \{\pm 1\}$ such that
		$$
		H^*(s,p) - H_*(s',p) \leq \omega_L(|s-s'|(1+|p|) )
		$$
		for all $s,s'$ such that $0<|s-s_0|+|s'-s_0|<\delta_0$ and $(s'-s)h_0 <0$. Here  $\omega_L:[0,\infty) \to [0,\infty)$ is a modulus of continuity for each $L>0$, i.e. it is continuous with $\omega_L(0) = 0$, and $H^*$ (resp., $H_*$) is the upper (resp., lower) semicontinuous envelope of $H$ with respect to the first variable, that is,
		\begin{equation}\label{eq: Ulsc}
			H^*(s,p)=\limsup_{s'\to s}H(s',p)\, \quad (\text{resp.,} \quad  H_*(s,p)=\liminf_{s'\to s}H(s',p)). 
		\end{equation}

		\item[{\rm (A2)}] $p\mapsto H(s,p)$ is convex, {$\displaystyle\inf_{s>0}H(s,0)>0$} and $\lim\limits_{|p|\to\infty} \left[\inf\limits_{s \in K} H(s,p)\right] \to \infty$ for each compact set $K \subset (0,\infty)$.
	
			\item[{\rm (A3)}]  The one-sided limits 
		$H(c_i\pm,p)$ are well defined.
		
		\item[{\rm (A4)}] There exists $\overline s > \max\{c_i\}_{i=1}^n$ such that $H(s,p) =   -sp +\tilde{H}(s,p) + R(s)$ for $s \geq \overline s$  and such that 
  $\tilde H(s,p)$ is non-increasing in $s \in [\overline s,\infty)$.  Moreover, $R \in L^\infty$ is locally monotone in $s \in [\overline s,\infty)$. 
\begin{remark}
    The definition of local monotonicity of $R(s)$ is stated  in (H3') of Subsection \ref{extension};  see also \cite{Chen2008viscosity}. Note also that (A1) implies that $H(s,p)$ is locally monotone as a function of two variables.  
\end{remark}
  
		
	\end{description}
	\begin{remark}
		For our purpose, we will take $H(s,p) = -sp + p^2 + R(s)$, where $R$ equals to positive constants on $(0,c_1)$ and on $(c_1,\infty)$. It is obvious that (A1)-(A4) hold.
	\end{remark}
	
	To define what it means by a viscosity solution to \eqref{eq:hjeapp}--\eqref{eq:kirapp}, we recall the space of piecewise $C^1$ test functions
	$$
	C^1_{pw}= C^0_{loc}((0,\infty))\cap C^1((0,c_1])\cap C^1([c_1,c_2])\cap ... \cap C^1([c_{n-1},c_n]) \cap C^1_{loc}([c_n,\infty)).
	$$ and use the notations in \eqref{eq: Ulsc}.
	
	\begin{definition}
		\begin{itemize}
			\item[{\rm(a)}] We say that $\underline\rho$ is a viscosity subsolution to \eqref{eq:hjeapp}--\eqref{eq:kirapp} if $\underline\rho$ is upper semicontinuous on $(0,+\infty)$ and it satisfies, in the viscosity sense, 
			\begin{equation}
				\begin{cases}
					\min\{\rho, \rho + H(s,\rho')\} \leq 0 \quad \text{ for }s \in (0,\infty)\setminus\mathcal{P},\\
					\min\{\rho(c_i), \min \{\rho'(c_i-)-\rho'(c_i+) - B_i,\rho(c_i) + H(c_i\pm,\rho'(c_i\pm)\} \leq 0 \text{ for each }1\leq i\leq n.
				\end{cases}
			\end{equation}
			i.e. whenever $\underline\rho - \psi$ has a strict local maximum point at $s_0$ for some $\psi \in C^1_{pw}$ and $\underline\rho(s_0)>0$, we have 
			$$
			\underline\rho(s_0) + H_*(s_0,\psi'(s_0))\leq 0 \quad \text{ if }s_0 \in (0,\infty) \setminus \mathcal{P}, 
			$$
			
			$$
			\min \{\psi'(c_i-)-\psi'(c_i+) - B_i,\underline\rho(c_i) + H(c_i\pm,\psi'(c_i\pm)\}\leq 0 \quad \text{ if }s_0 =c_i.
			$$
			
			\item[{\rm(b)}] We say that $\overline\rho$ is a viscosity supersolution to \eqref{eq:hjeapp}--\eqref{eq:kirapp} if $\overline\rho$ is lower semicontinuous on $(0,+\infty)$  and it satisfies, in the viscosity sense, 
			\begin{equation}
				\begin{cases}
					\min\{\rho, \rho + H(s,\rho')\} \geq 0 \quad \text{ for }s \in (0,\infty)\setminus\mathcal{P},\\
					\min\{\rho(c_i), \max \{\rho'(c_i-)-\rho'(c_i+) - B_i,\rho(c_i) + H(c_i\pm,\rho'(c_i\pm))\}\} \geq 0 \text{ for each }i.
				\end{cases}
			\end{equation}
			i.e. $\overline\rho\geq 0$ for all $s > 0$ and whenever $\underline\rho - \psi$ has a strict local minimum point at $s_0$ for some $\psi \in C^1_{pw}$, we have
			$$
			\overline\rho(s_0) + H^*(s_0,\psi'(s_0))\geq 0 \quad \text{ if }s_0 \in (0,\infty) \setminus \mathcal{P}, 
			$$
			$$
			\max \{\psi'(c_i-)-\psi'(c_i+) - B_i,\overline\rho(c_i) + H(c_i\pm,\psi'(c_i\pm)\}\geq 0 \quad \text{ if }s_0 =c_i.
			$$

			\item[{\rm(c)}] We say that $\overline\rho$ is a viscosity solution to \eqref{eq:hjeapp} --\eqref{eq:kirapp}, if it is both a viscosity subsolution and viscosity supersolution.
		\end{itemize}
	\end{definition}

	\begin{remark}
		The above setting includes the case with general $r(x,t)$ with infinitely many shifts (in that case $\tilde{H}_i(p) = p^2$ and $R$ is locally monotone except possibly at $c_i$), as well as the case when there is finitely many shifts, but periodic homogenization in between (in that case $R\equiv 0$).\end{remark}
	
	\begin{theorem}\label{thm:comparison}
		Let $\underline\rho$ and $\overline\rho$ be, respectively, 
  viscosity sub- and supersolutions of \eqref{eq:hjeapp}--\eqref{eq:kirapp}, such that
		$$
		\underline\rho(0) \leq  0 \leq \overline\rho(0), \quad \underline\rho(s) <\infty \quad \text{ for all }s\geq 0, \quad \text{ and }\quad \frac{\overline\rho(s)}{s} \to +\infty \quad \text{ as }s\to +\infty.
		$$
		Then $\underline\rho(s) \leq \overline\rho(s)$ in $[0,+\infty)$.
	\end{theorem}
 \begin{remark}
     {The above theorem directly implies the corresponding comparison principle for KL-super/subsolutions (see Proposition \ref{prop:2.10}). This is because $\overline\rho$ and $\underline\rho$ being a pair FL-super/subsolution (for an arbitrarily given $A \in \mathbb{R}$) implies the corresponding super/subsolution property with Kirchhoff junction condition (for some $B = B(A)$); See Subsection \ref{sec:5.3}.}
 \end{remark}
	\begin{proof}[Proof of Theorem \ref{thm:comparison}]
{First of all, we may assume without loss that $\underline\rho$ is nonnegative, $\underline\rho(0)=0$ and $\underline\rho \in {\rm Lip}_{loc}([0,\infty))$. This can be achieved by replacing $\underline\rho$ by $\max\{0,\underline\rho\}$. Since $0$ and  $\underline\rho$ (and hence also $\max\{\underline\rho,0\}$ is automatically a viscosity subsolution (in Ishii sense) to $\min\{\rho,\rho+ H(s,\rho')\}=0$, Lemma \ref{cor:subsolr} is applicable and it is therefore locally Lipschitz continuous.}
 
		Suppose  to the contrary that
		\begin{equation}\label{eq:contt}
			\sup_{s \geq 0} [\underline\rho(s) - \overline\rho(s)] >0.
		\end{equation}
		
		\smallskip
		
		\noindent {\bf Step \#1.}  
	We may assume without loss of generality that $\underline\rho(s) -\overline\rho(s) \to -\infty$ as $s \to +\infty$. 
	
 If $\limsup_{s \to \infty}\frac{\underline\rho(s)}{s} <\infty$, then we are done. Otherwise, we proceed as in \cite[Proposition 2.11]{Lam2022asymptotic}. First, 
	observe that $\underline\rho \in {\rm Lip}_{loc}((0,\infty))$. By Rademacher's theorem, it is differentiable in some $[0,\infty) \setminus \mathcal{S}$ where $\mathcal{S}$ has zero Lebesgue measure. Hence, we may choose $s_k \in [\overline s,\infty)$, $s_k \to \infty$ such that $\underline\rho$ is differentiable at $s_k$, and 
	$$
	\inf_k \underline\rho(s_k)>0  \quad \text{ and }\quad R(s_k) \to \limsup_{s \to \infty} R(s).
	$$
{Note that the latter is a consequence of 
local monotonicity of $R$ thanks to (A4).}
Next, define
	$$
	\underline{\rho}_k(s):= \begin{cases}
		\underline\rho(s) - \nu_k &\text{ for }0 \leq s \leq s_k,\\
		\underline\rho(s_k) - \nu_k + \underline\rho'(s_k)(s-s_k) &\text{ for }s > s_k,
	\end{cases}
	$$
	where $\nu_k = \sup_{[s_k,\infty)} R-R(s_k)$ (note that $\nu_k \to 0$). Observe that $\underline\rho_k$ is a viscosity subsolution in $[0,\infty)$ with linear growth as $s\to\infty$. Indeed, $\underline\rho_k$ is viscosity subsolution of \eqref{eq:hjeapp} in $[0,s_k)$ on the one hand, and classical subsolution of \eqref{eq:hjeapp} in $[s_k,\infty)$ on the other hand, since
	\begin{align*}
		\underline\rho_k +  H(s,\underline\rho'_k) &= \underline\rho_k -s \underline\rho'_k + \tilde H(s,\underline\rho'_k) + R(s)\\
		&= [a_k - \nu_k + b_k  (s-s_k)] - sb_k + \tilde{H}(s,b_k) + R(s)\\
		&\le a_k - \nu_k -s_kb_k  +\tilde{H}(s_k,b_k) + R(s)\\
		&=(a_k - s_k b_k + \tilde{H}(s_k,b_k) + R(s_k)) + (R(s) - R(s_k) - \nu_k)\\
		&\le\underline{\rho}(s_k)+H(s_k,\underline{\rho}'(s_k)) \leq 0, 
	\end{align*}
	where we adopted the notation $a_k = \underline\rho(s_k)$, $b_k = \underline\rho'(s_k)$.
	
	We can then replace $\underline\rho$ by $\underline\rho_k$, if necessary. Note that \eqref{eq:contt} still holds provided $k$ is sufficiently large, since $\nu_k \to 0$ and $s_k \to +\infty$.

	In the rest of the proof, we will show the comparison result is valid, i.e. 
	$$
	\max\{\underline\rho_k(s),0 \}\leq \overline\rho(s) \quad \text{ in }[0,\infty), \quad \text{ for all sufficiently large }k.
	$$  
	Granted, then we can take $k \to \infty$ to deduce that $\underline\rho\leq \overline\rho.$

	\medskip
	
	\noindent {\bf Step \#2.} For $\lambda \in (0,1)$, define $W (s) = \lambda \underline\rho(s) - \overline\rho(s)$. Then choose $\lambda \nearrow 1$, $0 < s_0 < \overline{s}_0$ such that 
	\begin{equation}\label{eq:WW}
		\eta_0:= W(s_0) =\max_{s\in[0,\infty)}W(s) >  \max\left\{\sup_{[\overline{s}_0-1,\infty)}W, (1-\lambda)\sup_{s \in [0,\bar{s}_0] \atop |p|\leq 2|\bar{s}_0|} |H(s,p)|\right\} 
	\end{equation}
	{For given $1\le i\le n$, we consider two cases . Either (i) there is a sequence $\lambda_j \nearrow 1$ such that $s_0 \neq c_i$ for all $j$, or (ii) there exists a sequence $\lambda_j \nearrow 1$ such that $s_0 = c_i$ for all $j$. 
		We first consider case (i), $s_0 \neq c_i$.}

	\medskip
	
	\noindent {\bf Step \#3.} 
	Next, define
	\begin{equation}
	\Psi_\alpha(s,t) = \lambda\underline\rho(s+\alpha^{-1/2}h_0) - \overline\rho(t) - \frac{\alpha}{2}|s-t|^2 - \frac{1-\lambda}{2}|s-s_0|^2, 
	\end{equation}
	{where $h_0$ is given in (A1).}
	\begin{claim}\label{claim:1}
		There exists $\bar \alpha>0$ such that if $\alpha>\bar\alpha$, then the following statements hold.
		\begin{itemize}
			\item[{\rm(i)}] $\Psi_\alpha$ has an interior local maximum $(s_1,t_1)$ in $(\frac{s_0}{2},\frac{\bar{s}_0+s_0}{2})\times (0, \bar{s}_0)$.
			\item[{\rm(ii)}] $\Psi_\alpha(s_1,t_1) \geq \Psi_\alpha(s_0,s_0)= \eta_0+o(1)>0$.
			\item[{\rm(iii)}] $\alpha|s_1-t_1|^2 \to 0$ as $\alpha \to \infty$.
			\item[{\rm(iv)}] $(s_1,t_1) \to (s_0,s_0)$ as $\alpha \to \infty.$
		\end{itemize}
	\end{claim}
	Denote $Q = [\frac{s_0}{2},\frac{\bar{s}_0+s_0}{2}]\times [0,\bar{s}_0]$. Clearly, $\Psi_{\alpha}$ is upper semicontinuous in $Q$, therefore, it attains the maximum at $(s_1,t_1)\in Q$.  
	By \eqref{eq:WW} and $\underline\rho \in {\rm Lip}_{loc}((0,\infty))$,  it follows that
	$$
	\sup_{Q} \Psi_\alpha \ge \Psi_\alpha(s_0,s_0) = \eta_0+\lambda(\underline{\rho}(s_0+\alpha^{-1/2}h_0)-\underline{\rho}(s_0))=\eta_0+o(1)
	$$
{where $o(1)$ is considered with respect to $\alpha \to +\infty$.}
	This proves (ii). 
	For (iii), first observe that $\alpha|s_1-t_1|^2 = O(1)$ which is a direct consequence of statement (ii).
	In particular, $(s_1,t_1) \to (\hat{s},\hat{s})$ for some $\hat{s} \in [\frac{s_0}{2},\frac{\bar{s}_0+s_0}{2}]$. By (ii), we can write
	$$
	\frac{\alpha}{2}|s_1-t_1|^2 \leq -W(s_0) + [W(s_1) + \overline\rho(s_1)] - \overline\rho(t_1) - \frac{1-\lambda}{2}|s_1-s_0|^2+o(1).
	$$
	Since $s\mapsto W(s) + \overline\rho(s)$ and $t\mapsto -\overline\rho(t)$ are both upper semicontinuous, we can take $\alpha \to \infty$ to obtain
	$$
	0 \leq \limsup_{\alpha\to\infty}\frac{\alpha}{2}|s_1-t_1|^2 \leq -W(s_0) + W(\hat s) - \frac{1-\lambda}{2}|\hat s-s_0|^2\leq 0,
	$$
	where the last inequality follows from the fact that $W$ attains global maximum at $s_0$.  
	This proves (iii) and (iv).
	Finally, statement (iv) yields $(s_1, t_1)$ must be an interior point of $Q$. This proves (i).
	
	\noindent {\bf Step \#4. }
	Fixing $t=t_1$, observe that for $\alpha>\bar \alpha$,
	$\underline\rho(s_1+\alpha^{-1/2}h_0) \geq \eta_0+{\overline{\rho}(t_1)}+o(1) >0$ {since $\overline{\rho}\geq 0$,} and 
	$s\mapsto  \underline\rho(s) - \varphi(s)$ attains a local maximum at $s= \hat s_1:=s_1+\alpha^{-1/2}h_0$, where
	$$
	\varphi(s) = \frac{1}{\lambda} \left[ \overline\rho(t_1) + \frac{\alpha}{2}|s-\alpha^{-1/2}h_0 - t_1|^2 + \frac{1-\lambda}{2}|s-\alpha^{-1/2}h_0-s_0|^2 \right].
	$$
	By definition of viscosity subsolution, we have
	\begin{equation}\label{eq:subsol11}
		\underline\rho(\hat s_1) + H_*\left(\hat s_1, \frac{\alpha(s_1-t_1)}{\lambda } + \frac{1-\lambda}{\lambda}(s_1-s_0)\right) \leq 0. 
	\end{equation}
	Using the convexity of $H_*$, we have
	$$
	\lambda H_*\left(\hat s_1, \frac{\alpha(s_1-t_1)}{\lambda } + \frac{1-\lambda}{\lambda}(s_1-s_0)\right) + (1-\lambda) H_*\left(\hat s_1,  - (s_1-s_0)\right) \geq H_*\left(\hat s_1,\alpha(s_1-t_1)\right).
	$$
	Substitute into \eqref{eq:subsol11}
	\begin{equation}\label{eq:subsol12}
		\lambda \underline\rho(\hat s_1) + H_*\left(\hat s_1, \alpha(s_1-t_1)\right) \leq (1-\lambda) H_*(\hat s_1,-(s_1 - s_0)) \leq (1-\lambda) \sup_{s \in [0,\bar{s}_0] \atop |p|\leq 2|\bar{s}_0|} H(s,p).
	\end{equation}
	
	Next, we fix $s=s_1$ and observe that $t\mapsto \overline\rho(t) - \psi(t)$ has an interior local minimum at $t=t_1$, where
	$$
	\psi(t) = \lambda \underline\rho(s_1) - \frac{\alpha}{2}|t-s_1|^2 - \frac{1-\lambda}{2}|s_1-s_0|^2.
	$$
	Hence, 
	\begin{equation}\label{eq:subsol13}
		\overline\rho(t_1) + H^*(t_1, \alpha(s_1-t_1))\geq 0 
	\end{equation}
{since $s_0\notin\mathcal{P}$ and thus $t_1\notin \mathcal{P}$.}
	Combining \eqref{eq:subsol12} and \eqref{eq:subsol13}, we have
	\begin{align} 
		&\quad \lambda \underline\rho(\hat s_1) - \overline\rho(t_1)  \leq H^*(t_1, \alpha(s_1-t_1))- H_*(\hat s_1,\alpha(s_1-t_1))  + (1-\lambda)   \sup_{s \in [0,\bar{s}_0] \atop |p|\leq 2|\bar{s}_0|} H(s,p) 
		\label{eq:subsol14}
	\end{align}
	
	\noindent {\bf Step \#5.} 
Observe that  $\lambda\underline\rho(\hat s_1) - \overline\rho(t_1) \geq \Psi_\alpha(s_1,t_1) \geq \eta_0+o(1)$.
Using also $|h_0|=1$ (by (A1)) and $\alpha|s_1-t_1|^2 = o(1)$ {(by Claim \ref{claim:1}(iii))}, we have
	$$
	(t_1-\hat s_1)h_0=[t_1 - s_1-\alpha^{-1/2}h_0]h_0 \leq |t_1-s_1| - \alpha^{-1/2} <0.
	$$

	We can then apply 
	(A1) to \eqref{eq:subsol14} to get
	\begin{align}\label{eq:subsol15}
	o(1)+\eta_0 &\leq \omega_L(|t_1-s_1-\alpha^{-1/2}h_0|(\alpha|t_1-s_1|+1)) + (1-\lambda) \sup_{s \in [0,\bar{s}_0], |p|\leq 2|\bar{s}_0|} H(s,p).
	\end{align}
	Using $\alpha|s_1-t_1|^2 = o(1)$ and $$|t_1-s_1-\alpha^{-1/2}h_0|(\alpha|t_1-s_1|+1)\le (\alpha^{1/2}|t_1-s_1|+1)(\alpha^{1/2}|t_1-s_1|+\alpha^{-1/2}),$$
we let $\alpha\to \infty$ in \eqref{eq:subsol15} to 
 obtain 
 $$\eta_0\le (1-\lambda) \sup_{s \in [0,\bar{s}_0], |p|\leq 2|\bar{s}_0|} H(s,p),$$
which is a contradiction. This concludes the proof when $W(s) = \lambda\underline\rho(s) - \overline\rho(s)$ attains a local maximum at $s_0$, such that $s_0 \neq c_i$ for all $i$.
	
	\medskip
	
	It remains to consider the case (ii) (see Step \#2), when there is $\lambda_j \nearrow 1$ such that $\lambda_j\underline\rho - \overline\rho$ has a global maximum at $s_0 = c_i$ for some $i$. Since this holds for a sequence of $\lambda_j \nearrow 1$, we reduce to the case that $\underline\rho-\overline\rho$ attains the global maximum at $c_i$ for some $i$. For convenience, let's assume $s_0 = c_1$. Next, define 
	$$
	a = \underline{\rho}(c_1) \quad \text{ and }\quad b = \overline{\rho}(c_1)
	$$
	and assume to the contrary that $a > b\ge0$. (We will show that $a \leq b$ so there is a contradiction.)
	
	\medskip
	
	\noindent {\bf Step \#6.} We claim that the critical slopes of $\underline\rho$, given as follows, are finite.
	\begin{equation}
		p_- = \sup\{\bar p \in \mathbb{R}:~\exists r>0,~ \underline\rho(c_1) +\bar p(s-c_1) \geq \underline{\rho}(s)~\text{ for }-r < s-c_1 \leq 0\}
	\end{equation}
	\begin{equation}
		p_+ = \sup\{\bar p\in \mathbb{R}:~\exists r>0,~  \underline\rho(c_1)-\bar p(s-c_1) \geq \underline{\rho}(s)~\text{ for }0 \leq s-c_1 < r\}.
	\end{equation}
	
	Indeed, they are finite because $\underline\rho$ is locally Lipschitz continuous. 
 Moreover, we have 
		\begin{equation}\label{eq:lions1r}
			\begin{cases}
				a + H(c_1-, p_-) \leq 0, \quad a + H(c_1+, -p_+) \leq 0,\quad \text{ and }\\
				\min\{ p'_- + p'_+- B_1, a + H(c_1\pm, \mp p'_\pm)\} \leq 0 \, \text{ for }(p'_-, p'_+) \in (-\infty, p_-] \times (-\infty, p_+], 
			\end{cases}
		\end{equation}
		where the former is due to Lemma \ref{cor:subsolr} and Remark \ref{rem:subsoll} 
  {(note that $\underline\rho \in {\rm Lip}_{loc}$, so there exists at least one test function in $C^1_{pw}$),}  
while the latter holds by considering the test function $\psi(s) = \underline\rho(c_1) + p'_- \min\{s-c_1,0\} - p'_+\max\{s-c_1,0\}$, which touches $\underline\rho$ from above at $c_1$. 
		
		Next, define
		\begin{equation}\label{eq:lions2rr}
			{p}^*_- = \limsup_{s\to c_1-} \frac{\underline\rho(s) - \underline\rho(c_1)}{s-c_1} \quad\text{ and }\quad {p}^*_+ = -\liminf_{s\to c_1+} \frac{\underline\rho(s) - \underline\rho(c_1)}{s-c_1}.
		\end{equation}
		Note that 
		\begin{equation}\label{eq:lions2r}
			p^*_- \geq p_-\quad \text{ and }\quad p^*_+ \geq p_+,
		\end{equation} since $\displaystyle p_- = \liminf_{s\to c_1-} \frac{\underline\rho(s) - \underline\rho(c_1)}{s-c_1}$ and $\displaystyle p_+ =  -\limsup_{s\to c_1+}\frac{\underline\rho(s) - \underline\rho(c_1)}{s-c_1}$.
		
		
		\medskip

		\noindent {\bf Step \#7.} We improve \eqref{eq:lions1r} to
		\begin{equation}\label{eq:lions5r}
			\begin{cases}
				a + H(c_1-, p^*_-) \leq 0, \quad a + H(c_1+, -p^*_+) \leq 0, \quad \text{ and }\\
				\min\{ p'_- + p'_+ -B_1, a + H(c_1\pm, \mp p'_\pm)\} \leq 0 \, \text{ for }(p'_-, p'_+) \in (-\infty, p^*_-] \times (-\infty, p^*_+]. 
			\end{cases}
		\end{equation}
		
		Indeed, since $\underline\rho$ is locally Lipschitz, $\underline\rho'$ exists a.e. and $\underline\rho(s) - \underline\rho(c_1) =\int_{c_1}^s \underline\rho'(t)\,dt$, the definition of $p^*_+$ implies that, for each $\delta>0$, the set $\{s \in (c_1,c_1+\delta):~ \underline\rho'(s) < -p^*_+ +\delta\}$ has positive measure. This implies that there is a sequence $s_k \searrow c_1$ such that $\underline\rho$ is differentiable at $s_k$ and also that $\limsup\limits_{k\to\infty}\underline\rho'(s_k)\leq -p^*_+$. Hence, letting $k \to \infty$ in $\underline\rho(s_k) + H(s_k, \rho'(s_k)) \leq 0$, we obtain
		$$
		a + H(c_1+, \liminf_{k\to\infty}\underline\rho'(s_k)) \leq 0.
		$$
		Noting that $H$ is convex in $p$ variable, $\liminf\limits_{k\to\infty}\underline\rho'(s_k) \leq -p^*_+ \leq -p_+$, and using the first part of \eqref{eq:lions1r}, we deduce
		\begin{equation}\label{eq:lions3r}
			a + H(c_1+, -p'_+) \leq 0 \quad \text{ for all }p'_+ \in [p_+, p^*_+].
		\end{equation}
		By a completely similar argument, we also have
		\begin{equation}\label{eq:lions4r}
			a + H(c_1-, p'_-) \leq 0 \quad \text{ for all }p'_- \in [p_-, p^*_-].
		\end{equation}
		Combining \eqref{eq:lions3r} and \eqref{eq:lions4r} into \eqref{eq:lions1r}, we obtain
		\eqref{eq:lions5r}.

		\medskip

		\noindent {\bf Step \#8.} We claim that the critical slopes of $\overline\rho$, defined as follows, are well-defined but possibly equals $-\infty$.
		\begin{equation}
			q_- = \inf\{\bar q\in \mathbb{R}:~\exists r>0,~ \overline\rho(c_1) + \bar q (s-c_1) \leq \overline{\rho}(s)~\text{ for }-r <s-c_1\leq 0\}
		\end{equation} 
		\begin{equation}
			q_+ = \inf\{\bar q \in \mathbb{R}:~\exists r>0,~ \bar\rho(c_1)-\bar q (s-c_1) \leq \overline{\rho}(s)~\text{ for }0\leq s-c_1 <r\}.
		\end{equation} 
		Indeed, 
		\begin{equation}\label{eq:lions-1}
			\underline\rho(s) -\overline\rho(s) \leq \underline\rho(c_1) -\overline\rho(c_1)\quad \text{ for all }s, \quad \text{ with equality holds at }s=c_1, 
		\end{equation}
		i.e. the locally Lipschitz function $\underline\rho(s) - \underline\rho(c_1)+\overline\rho(c_1)$ touches $\overline\rho(s)$ from below at $s=c_1$. This shows that $q_-$ and $q_+$ are well-defined in $\mathbb{R}\cup \{-\infty\}$.
		
		Next, we observe that
		\begin{equation}\label{eq:lions0}
			q_- \leq p^*_- \quad \text{ and }\quad q_+ \leq p^*_+,
		\end{equation}
		which is due to \eqref{eq:lions2rr}, \eqref{eq:lions-1}, and 
		$$ q_- = \limsup_{s\to c_1-} \frac{\overline\rho(s) - \overline\rho(c_1)}{s-c_1}\quad \text{ and  } q_+ =  -\liminf_{s\to c_1+}\frac{\overline\rho(s) - \overline\rho(c_1)}{s-c_1}.$$

		
		\noindent {\bf Step \#9.} Suppose $q_{\pm}>-\infty$, then we have
		\begin{subequations}\label{eq:lions1}
			\begin{empheq}[left=\empheqlbrace]{align}
				& b + H(c_1-,q_-) \geq 0, \quad b + H(c_1+,- q_+) \geq 0 \quad \text{ and }\quad  \label{eq:lions1a}\\
				& \max\{ q'_- + q'_+ - B_1,  b + H(c_1\pm, \mp q'_\pm)\} \geq 0, ~\text{ for } (q'_-, q'_+) \in [q_-,\infty) \times [q_+,\infty),\label{eq:lions1b}
			\end{empheq}
		\end{subequations}
		where the former holds by virtue of the critical slope lemma (Lemma \ref{cor:supersolr} and Remark \ref{rem: supsoll}), and the latter holds by considering the $C^1_{pw}$ test function $\psi(s) = \overline\rho(c_1) + q'_-\min\{s-c_1,0\} - q'_+\max\{s-c_1,0\}$.

		If $q_-=-\infty$ (resp.  $q_+ = -\infty$) then take $q_-$ (resp. $q_+$) large and negative enough (but finite) to satisfy both \eqref{eq:lions0} and \eqref{eq:lions1a}.
		Then, for any $(q'_-, q'_+) \in [q_-,\infty) \times [q_+,\infty)$, the test function $\psi(s) = \bar\rho(c_1)+q'_-\min\{s-c_1,\,0\} - q'_+ \max\{s-c_1,\,0\}$ touches $\overline\rho$ at $s = c_1$. Then it follows that \eqref{eq:lions1b} holds. 
		
		\medskip
		
		\noindent {\bf Step \#10.} 
		
		In view of \eqref{eq:lions5r}, \eqref{eq:lions0}, and \eqref{eq:lions1}, we may apply the Lemma \ref{lem:lions} with ($H_1(p)=H(c_1-,p)$ and $H_2(p)=H(c_1+,-p)$, $p\in\mathbb{R}$, $p_1=p^*_-$, $p_2=p^*_+$, $q_1=q_-$ and $q_2=q_+$) to conclude that $a \leq b$. This is a contradiction to the assumption that $a >b$. The proof is now complete.
	\end{proof}
	The following key lemma is due to Lions and Souganidis \cite{Lions2017well}.
	\begin{lemma}		Assume that $H_1,..., H_m \in C(\mathbb{R})$, $p_1,...,p_m, q_1,...,q_m \in \mathbb{R}$, and $a,b \in \mathbb{R}$ are such that, for all $i=1,...,m$, 
		\begin{enumerate}
			\item $p_i\geq q_i$, $a + H_i(p_i) \leq 0 \leq b + H_i(q_i)$ for all $i$,
			\item $\min\left( \sum_i p'_i-B, \min_i(a + H_i(p'_i)) \right) \leq 0$ for each $p'_i \in (-\infty, p_i]$,
			\item $\max\left( \sum_i q'_i-B, \max_i(b + H_i(q'_i)) \right) \leq 0$ for each $q'_i \in [q_i, \infty)$.
		\end{enumerate}
		Then $a \leq b$.
		\label{lem:lions}
	\end{lemma}
	\begin{remark}
		By replacing $p_1$ by $p_1+ B$ and $q_1$ by $q_1 + B$, and redefining
		$$
		H_1(\cdot) \quad \text{ to be }\quad H_1(\cdot +B),
		$$
		one can reduce Lemma \ref{lem:lions} to the case $B=0$, which is exactly \cite[Lemma 3.1]{Lions2017well}. 
	\end{remark}
	
	\section{Definition of viscosity solution in sense of Ishii}
 	\begin{definition}\label{def:1.2}
		Let $\hat\rho : (0,\infty) \to \mathbb{R}$. 
		\begin{itemize}
			\item[{\rm(a)}] We say that $\hat\rho$ is a viscosity subsolution of 
   \begin{equation}\label{eq:ishii11}
		\min\{\rho, \rho + \tilde{H}(s,\rho')\} = 0 \quad \text{ for }s \geq 0
	\end{equation}
in the sense of Ishii provided (i) $\hat\rho$ is upper semicontinuous, (ii) if $\hat\rho - \psi$ has a local maximum point at $s_0>0$ such that $\psi \in C^1$ and $\hat\rho(s_0)>0$, then
			$$
			\hat\rho(s_0) + {\tilde H}_*(s,\psi'(s_0)) \leq 0,
			$$
			where ${\tilde H}_*(s,p)$ is the lower envelope of ${\tilde H}(s,p)$, i.e.
			$$
			{\tilde H}_*(s,p) =\begin{cases}
				-sp + p^2 + g(-\infty) &\text{ for }s < c_1,\\
				{\tilde H}(c_1-,p) \wedge {\tilde H}(c_1+,p)=-c_1p + p^2 + \min\{g(-\infty),g(+\infty)\} &\text{ for }s = c_1,\\
				-sp + p^2 + g(+\infty) &\text{ for }s > c_1.
			\end{cases}
			$$
			
			\item[{\rm(b)}] We say that $\hat\rho$ is a viscosity supersolution of \eqref{eq:ishii11} in the sense of Ishii provided (i) $\hat\rho$ is lower semicontinuous, (ii) $\hat\rho \geq 0$ for all $s >0$, (iii) if $\hat\rho - \psi$ has a local minimum point at $s_0>0$ such that $\psi \in C^1$, then
			$$
			\hat\rho(s_0) + {\tilde H}^*(s,\psi'(s_0)) \geq 0,
			$$
			where ${\tilde H}^*(s,p)$ is the upper envelope of $H(s,p)$, i.e.
			$$
			{\tilde H}^*(s,p) =\begin{cases}
				-sp + p^2 + g(-\infty) &\text{ for }s < c_1,\\
				{\tilde H}(c_1-,p) \vee {\tilde H}(c_1+,p)=-c_1p + p^2 + \max\{g(-\infty),g(+\infty)\} &\text{ for }s = c_1,\\
				-sp + p^2 + g(+\infty) &\text{ for }s > c_1.
			\end{cases}
			$$
			\item[{\rm(c)}] We say that $\hat\rho$ is a viscosity solution of \eqref{eq:ishii11} in the sense of Ishii if it is both subsolution and supersolution of \eqref{eq:ishii1} in the sense of Ishii.
		\end{itemize}
	\end{definition}


\begin{thebibliography}{10}
{\small
\bibitem{Weinberger1978}
{\sc D.~G. Aronson and H.~F. Weinberger}, {\em Multidimensional nonlinear
  diffusion arising in population genetics}, Adv. in Math., 30 (1978),
  pp.~33--76.

\bibitem{Barles2013introduction}
{\sc G.~Barles}, {\em An introduction to the theory of viscosity solutions for
  first-order {H}amilton-{J}acobi equations and applications}, in
  Hamilton-{J}acobi equations: approximations, numerical analysis and
  applications, vol.~2074 of Lecture Notes in Math., Springer, Heidelberg,
  2013, pp.~49--109.

\bibitem{barles2023modern}
{\sc G.~Barles and E.~Chasseigne}, {\em On Modern Approaches of Hamilton-Jacobi
  Equations and Control Problems with Discontinuities: A Guide to Theory,
  Applications, and Some Open Problems}, vol.~104, Springer Nature, 2023.

\bibitem{BarlesPerthame87}
{\sc G.~Barles and B.~Perthame}, {\em Discontinuous solutions of deterministic
  optimal stopping time problems}, RAIRO Mod\'el. Math. Anal. Num\'er., 21
  (1987), pp.~557--579.

\bibitem{BarlesPerthame88}
\leavevmode\vrule height 2pt depth -1.6pt width 23pt, {\em Exit time problems
  in optimal control and vanishing viscosity method}, SIAM J. Control Optim.,
  26 (1988), pp.~1133--1148.

\bibitem{Barles1990comparison}
\leavevmode\vrule height 2pt depth -1.6pt width 23pt, {\em Comparison principle
  for {D}irichlet-type {H}amilton-{J}acobi equations and singular perturbations
  of degenerated elliptic equations}, Appl. Math. Optim., 21 (1990),
  pp.~21--44.

\bibitem{Berestycki2009}
{\sc H.~Berestycki, O.~Diekmann, C.~J. Nagelkerke, and P.~A. Zegeling}, {\em
  Can a species keep pace with a shifting climate?}, Bull. Math. Biol., 71
  (2009), pp.~399--429.

\bibitem{Berestycki2018forced}
{\sc H.~Berestycki and J.~Fang}, {\em Forced waves of the {F}isher-{KPP}
  equation in a shifting environment}, J. Differential Equations, 264 (2018),
  pp.~2157--2183.

\bibitem{BHN2008}
{\sc H.~Berestycki, F.~Hamel, and G.~Nadin}, {\em Asymptotic spreading in
  heterogeneous diffusive excitable media}, J. Funct. Anal., 255 (2008),
  pp.~2146--2189.

\bibitem{Nadin2012}
{\sc H.~Berestycki and G.~Nadin}, {\em Spreading speeds for one-dimensional
  monostable reaction-diffusion equations}, J. Math. Phys., 53 (2012),
  pp.~115619, 23.

\bibitem{Nadin2020}
{\sc H.~Berestycki and G.~Nadin}, {\em Asymptotic spreading for general
  heterogeneous fisher-kpp type equations}, Mem. Amer. Math. Soc.,  (2022).

\bibitem{Berestycki2008reaction}
{\sc H.~Berestycki and L.~Rossi}, {\em Reaction-diffusion equations for
  population dynamics with forced speed. {I}. {T}he case of the whole space},
  Discrete Contin. Dyn. Syst., 21 (2008), pp.~41--67.

\bibitem{Berestycki2009reaction}
\leavevmode\vrule height 2pt depth -1.6pt width 23pt, {\em Reaction-diffusion
  equations for population dynamics with forced speed. {II}. {C}ylindrical-type
  domains}, Discrete Contin. Dyn. Syst., 25 (2009), pp.~19--61.

\bibitem{Berestycki2014generalizations}
\leavevmode\vrule height 2pt depth -1.6pt width 23pt, {\em Generalizations and
  properties of the principal eigenvalue of elliptic operators in unbounded
  domains}, Comm. Pure Appl. Math., 68 (2015), pp.~1014--1065.

\bibitem{Chen2008viscosity}
{\sc X.~Chen and B.~Hu}, {\em Viscosity solutions of discontinuous
  {H}amilton-{J}acobi equations}, Interfaces Free Bound., 10 (2008),
  pp.~339--359.

\bibitem{Dong2021persistence}
{\sc F.-D. Dong, J.~Shang, W.~Fagan, and B.~Li}, {\em Persistence and spread of
  solutions in a two-species {L}otka-{V}olterra competition-diffusion model
  with a shifting habitat}, SIAM J. Appl. Math., 81 (2021), pp.~1600--1622.

\bibitem{DHL23}
{\sc Y.~Du, Y.~Hu, and X.~Liang}, {\em A climate shift model with free
  boundary: enhanced invasion}, J. Dynam. Differential Equations, 35 (2023),
  pp.~771--809.

\bibitem{DWZ18}
{\sc Y.~Du, L.~Wei, and L.~Zhou}, {\em Spreading in a shifting environment
  modeled by the diffusive logistic equation with a free boundary}, J. Dynam.
  Differential Equations, 30 (2018), pp.~1389--1426.

\bibitem{Ducrot2021asymptotic}
{\sc A.~Ducrot, T.~Giletti, J.-S. Guo, and M.~Shimojo}, {\em Asymptotic
  spreading speeds for a predator-prey system with two predators and one prey},
  Nonlinearity, 34 (2021), pp.~669--704.

\bibitem{evans1992periodic}
{\sc L.~C. Evans}, {\em Periodic homogenisation of certain fully nonlinear
  partial differential equations}, Proceedings of the Royal Society of
  Edinburgh Section A: Mathematics, 120 (1992), pp.~245--265.

\bibitem{Evans1989pde}
{\sc L.~C. Evans and P.~E. Souganidis}, {\em A {PDE} approach to geometric
  optics for certain semilinear parabolic equations}, Indiana Univ. Math. J.,
  38 (1989), pp.~141--172.

\bibitem{FLW2016}
{\sc J.~Fang, Y.~Lou, and J.~Wu}, {\em Can pathogen spread keep pace with its
  host invasion?}, SIAM J. Appl. Math., 76 (2016), pp.~1633--1657.

\bibitem{FYZ2017}
{\sc J.~Fang, X.~Yu, and X.-Q. Zhao}, {\em Traveling waves and spreading speeds
  for time-space periodic monotone systems}, J. Funct. Anal., 272 (2017),
  pp.~4222--4262.

\bibitem{Faye2022asymptotic}
{\sc G.~Faye, T.~Giletti, and M.~Holzer}, {\em Asymptotic spreading for
  {F}isher-{KPP} reaction-diffusion equations with heterogeneous shifting
  diffusivity}, Discrete Contin. Dyn. Syst. Ser. S, 15 (2022), pp.~2467--2496.

\bibitem{Freidlin1985limit}
{\sc M.~Freidlin}, {\em Limit theorems for large deviations and
  reaction-diffusion equations}, Ann. Probab., 13 (1985), pp.~639--675.

\bibitem{Freidlin1996wave}
{\sc M.~I. Freidlin and T.-Y. Lee}, {\em Wave front propagation and large
  deviations for diffusion-transmutation process}, Probab. Theory Related
  Fields, 106 (1996), pp.~39--70.

\bibitem{Garnier2012maximal}
{\sc J.~Garnier, T.~Giletti, and G.~Nadin}, {\em Maximal and minimal spreading
  speeds for reaction diffusion equations in nonperiodic slowly varying media},
  J. Dynam. Differential Equations, 24 (2012), pp.~521--538.

\bibitem{Giga2013hamilton}
{\sc Y.~Giga and N.~Hamamuki}, {\em Hamilton-{J}acobi equations with
  discontinuous source terms}, Comm. Partial Differential Equations, 38 (2013),
  pp.~199--243.

\bibitem{girardin2024spreading}
{\sc L.~Girardin, T.~Giletti, and H.~Matano}, {\em Spreading properties of the
  fisher--kpp equation when the intrinsic growth rate is maximal in a moving
  patch of bounded size}, 2024.

\bibitem{Girardin2019}
{\sc L.~Girardin and K.-Y. Lam}, {\em Invasion of open space by two
  competitors: spreading properties of monostable two-species
  competition-diffusion systems}, Proc. Lond. Math. Soc. (3), 119 (2019),
  pp.~1279--1335.

\bibitem{Guerand2017effective}
{\sc J.~Guerand}, {\em Effective nonlinear {N}eumann boundary conditions for
  1{D} nonconvex {H}amilton-{J}acobi equations}, J. Differential Equations, 263
  (2017), pp.~2812--2850.

\bibitem{Hamel1997reaction}
{\sc F.~Hamel}, {\em Reaction-diffusion problems in cylinders with no
  invariance by translation. {II}. {M}onotone perturbations}, Ann. Inst. H.
  Poincar\'{e} C Anal. Non Lin\'{e}aire, 14 (1997), pp.~555--596.

\bibitem{Hamel2012spreading}
{\sc F.~Hamel and G.~Nadin}, {\em Spreading properties and complex dynamics for
  monostable reaction-diffusion equations}, Comm. Partial Differential
  Equations, 37 (2012), pp.~511--537.

\bibitem{Holzer2014accerlated}
{\sc M.~Holzer and A.~Scheel}, {\em Accelerated fronts in a two-stage invasion
  process}, SIAM J. Math. Anal., 46 (2014), pp.~397--427.

\bibitem{HSL2019}
{\sc C.~Hu, J.~Shang, and B.~Li}, {\em Spreading speeds for reaction-diffusion
  equations with a shifting habitat}, J. Dynam. Differential Equations, 32
  (2020), pp.~1941--1964.

\bibitem{Imbert2017flux}
{\sc C.~Imbert and R.~Monneau}, {\em Flux-limited solutions for quasi-convex
  {H}amilton-{J}acobi equations on networks}, Ann. Sci. \'{E}c. Norm.
  Sup\'{e}r. (4), 50 (2017), pp.~357--448.

\bibitem{Imbert2017quasi}
\leavevmode\vrule height 2pt depth -1.6pt width 23pt, {\em Quasi-convex
  {H}amilton-{J}acobi equations posed on junctions: the multi-dimensional
  case}, Discrete Contin. Dyn. Syst., 37 (2017), pp.~6405--6435.

\bibitem{Ishii1985hamilton}
{\sc H.~Ishii}, {\em Hamilton-{J}acobi equations with discontinuous
  {H}amiltonians on arbitrary open sets}, Bull. Fac. Sci. Engrg. Chuo Univ., 28
  (1985), pp.~33--77.

\bibitem{Kolmogorov1937}
{\sc A.~N. Kolmogorov, I.~G. Petrovskii, and N.~S. Piskunov}, {\em Etude de
  l'\'{e}quation de diffusion avec accroissement de la quantit\'{e} de
  mati\`{e}re, et son application \`{a} un probl\`{e}me biologique}, Bjul.
  Moskowskogo Gos. Univ., 17 (1937), pp.~1--26.

\bibitem{Lam2022asymptotic}
{\sc K.-Y. Lam and X.~Yu}, {\em Asymptotic spreading of {KPP} reactive fronts
  in heterogeneous shifting environments}, J. Math. Pures Appl. (9), 167
  (2022), pp.~1--47.

\bibitem{Li2014}
{\sc B.~Li, S.~Bewick, J.~Shang, and W.~F. Fagan}, {\em Persistence and spread
  of a species with a shifting habitat edge}, SIAM J. Appl. Math., 74 (2014),
  pp.~1397--1417.

\bibitem{LiangZhao2007}
{\sc X.~Liang and X.-Q. Zhao}, {\em Asymptotic speeds of spread and traveling
  waves for monotone semiflows with applications}, Comm. Pure Appl. Math., 60
  (2007), pp.~1--40.

\bibitem{lions1987homogenization}
{\sc P.-L. Lions, G.~Papanicolaou, and S.~R. Varadhan}, {\em Homogenization of
  hamilton-jacobi equations}, Unpublished preprint,  (1987).

\bibitem{Lions2017well}
{\sc P.-L. Lions and P.~Souganidis}, {\em Well-posedness for multi-dimensional
  junction problems with {K}irchoff-type conditions}, Atti Accad. Naz. Lincei
  Rend. Lincei Mat. Appl., 28 (2017), pp.~807--816.

\bibitem{LionsSouganidis2016}
{\sc P.-L. Lions and P.~E. Souganidis}, {\em Viscosity solutions for junctions:
  well posedness and stability}, Atti Accad. Naz. Lincei Rend. Lincei Mat.
  Appl., 27 (2016), pp.~535--545.

\bibitem{Liu2021stacked}
{\sc Q.~Liu, S.~Liu, and K.-Y. Lam}, {\em Stacked invasion waves in a
  competition-diffusion model with three species}, J. Differential Equations,
  271 (2021), pp.~665--718.

\bibitem{Liu2021asymptotic}
{\sc S.~Liu, Q.~Liu, and K.-Y. Lam}, {\em Asymptotic spreading of interacting
  species with multiple fronts {II}: {E}xponentially decaying initial data}, J.
  Differential Equations, 303 (2021), pp.~407--455.

\bibitem{Potapov2004}
{\sc A.~B. Potapov and M.~A. Lewis}, {\em Climate and competition: the effect
  of moving range boundaries on habitat invasibility}, Bull. Math. Biol., 66
  (2004), pp.~975--1008.

\bibitem{Shen2010}
{\sc W.~Shen}, {\em Variational principle for spreading speeds and generalized
  propagating speeds in time almost periodic and space periodic {KPP} models},
  Trans. Amer. Math. Soc., 362 (2010), pp.~5125--5168.

\bibitem{WLFQ2022}
{\sc J.-B. Wang, W.-T. Li, F.-D. Dong, and S.-X. Qiao}, {\em Recent
  developments on spatial propagation for diffusion equations in shifting
  environments}, Discrete Contin. Dyn. Syst. Ser. B, 27 (2022), pp.~5101--5127.

\bibitem{Weinberger1982}
{\sc H.~F. Weinberger}, {\em Long-time behavior of a class of biological
  models}, SIAM J. Math. Anal., 13 (1982), pp.~353--396.

\bibitem{Yi2020}
{\sc T.~Yi and X.-Q. Zhao}, {\em Propagation dynamics for monotone evolution
  systems without spatial translation invariance}, J. Funct. Anal., 279 (2020),
  pp.~108722, 50.

\bibitem{YZDCDS23}
\leavevmode\vrule height 2pt depth -1.6pt width 23pt, {\em Global dynamics of
  evolution systems with asymptotic annihilation}, Discrete Contin. Dyn. Syst.,
  43 (2023), pp.~2693--2720.
}
\end{thebibliography}

 \end{document}